\documentclass[reqno,10pt]{amsart}
\usepackage{amscd,amssymb,verbatim,array}
\usepackage{hyperref}
\usepackage{amsmath}
\usepackage[active]{srcltx} 

\numberwithin{equation}{section} \theoremstyle{plain}

\newtheorem{theorem}{Theorem}[section]
\newtheorem{lemma}[theorem]{Lemma}

\newtheorem{corollary}[theorem]{Corollary}
\newtheorem{definition}[theorem]{Definition}

\theoremstyle{definition}

\theoremstyle{remark}

\numberwithin{equation}{section}

\newcommand{\Det}{\operatorname{Det}}
\newcommand{\B}{\mathcal{B}}
\newcommand{\E}{\mathcal{E}}

\newcommand{\even}{\operatorname{even}}
\newcommand{\trivial}{\operatorname{trivial}}

\newcommand{\cyl}{\operatorname{cyl}}
\newcommand{\odd}{\operatorname{odd}}
\newcommand{\gr}{\operatorname{gr}}

\newcommand{\Tan}{\operatorname{tan}}
\newcommand{\Nor}{\operatorname{nor}}
\newcommand{\Dim}{\operatorname{dim}}
\newcommand{\Mod}{\operatorname{mod}}
\newcommand{\End}{\operatorname{End}}

\newcommand{\rel}{\operatorname{rel}}
\newcommand{\Abs}{\operatorname{abs}}
\newcommand{\rank}{\operatorname{rank}}
\newcommand{\pr}{\operatorname{pr}}

\newcommand{\Ker}{\operatorname{ker}}

\newcommand{\Tr}{\operatorname{Tr}}
\newcommand{\re}{\operatorname{Re}}

\newcommand{\Imm}{\operatorname{Im}}
\newcommand{\Dom}{\operatorname{Dom}}

\newcommand{\I}{\operatorname{I}}
\newcommand{\II}{\operatorname{II}}

\newcommand{\Id}{\operatorname{Id}}
\newcommand{\SF}{\operatorname{SF}}
\newcommand{\Mas}{\operatorname{Mas}}

\newcommand{\Vol}{\operatorname{vol}}

\begin{document}

\title[The refined analytic torsion and a well-posed boundary condition]
{The refined analytic torsion and a well-posed boundary condition for the odd signature operator}
\author{Rung-Tzung Huang}

\address{Department of Mathematics, National Central University, Chung-Li 320, Taiwan, Republic of China}

\email{rthuang@math.ncu.edu.tw}

\author{Yoonweon Lee}

\address{Department of Mathematics, Inha University, Incheon, 402-751, Korea}

\email{yoonweon@inha.ac.kr}

\thanks{The first author was supported by Taiwan Ministry of Science and Technology project 105-2115-M-008-008-MY2.
The second author was supported by the KRF Grant funded by KRF-2008-313-C00072. Both authors would like to thank the Institute of Mathematics at Academia Sinica for its hospitality and financial support during their stay in August, 2017.}

\begin{abstract}
In this paper we discuss the refined analytic torsion on an odd dimensional compact oriented Riemannian manifold with boundary
under some assumption.
For this purpose we introduce two boundary conditions which are complementary to each other and well-posed for
the odd signature operator $\B$ in the sense of Seeley.
We then show that the zeta-determinants of $\B^{2}$ and eta-invariants of $\B$ subject to these boundary
conditions are well defined by using the method of the asymptotic expansions of the traces of the heat kernels.
We use these facts to define the refined analytic torsion on a compact manifold with boundary
and show that it is invariant on the change of metrics in the interior of the manifold.
We finally describe the refined analytic torsion under these boundary conditions as an element of the determinant line.

\end{abstract}
\maketitle

\section{Introduction}

\vspace{0.2 cm}

The refined combinatorial torsion  was introduced by V. Turaev (\cite{Tu1}, \cite{Tu2}) and further developed by M. Farber and V. Turaev (\cite{FT1}, \cite{FT2}).
It is defined by the representation of the fundamental group to $GL(n, {\Bbb C})$, the Euler structure and the cohomology orientation.
As an analytic analogue of the refined combinatorial torsion, M. Braverman and T. Kappeler introduced the refined analytic torsion
on an odd dimensional closed Riemannian manifold (\cite{BK1}, \cite{BK2}), which is an element of the determinant line $\Det(H^{\bullet}(M, E))$
and is defined by using the graded zeta-determinant of the odd signature operator $\B$.
Even though these two objects do not coincide exactly, they are closely related.

In this paper we are going to discuss the refined analytic torsion on a compact Riemannian manifold with boundary under some assumption (Assumption A in Section 2.3).
Roughly, the refined analytic torsion consists of two ingredients, which are the Ray-Singer analytic torsion and the eta invariant of the odd signature operator.
To extend the refined analytic torsion to a compact manifold with boundary, we need a boundary condition which is able to define
both the Ray-Singer analytic torsion and the eta invariant.
Local boundary conditions such as the absolute and relative boundary conditions are natural to the Ray-Singer analytic torsion
but they do not fit to the eta invariant.
Similarly, the Atiyah-Patodi-Singer (APS) boundary condition is natural to the eta invariant but it does not fit to the Ray-Singer analytic torsion.
Hence, none of them is a good choice to define the refined analytic torsion.
For this reason we introduce two new boundary conditions which are complementary to each other
and show that the refined analytic torsion is well defined under these boundary conditions.

Let $(M, Y, g^{M})$ be a compact oriented Riemannian manifold with boundary $Y$.
Throughout this paper we assume that the metric $g^{M}$ is a product one on a small collar neighborhood of $Y$.
Let $E \rightarrow M$ be a complex flat vector bundle
with a flat connection $\nabla$
associated to the representation $\rho : \pi_{1}(M) \rightarrow GL(n, {\Bbb C})$.
Then $\nabla$ is extended to the de Rham operator acting on $E$-valued differential forms $\Omega^{\bullet}(M, E)$,
which we denote by $\nabla$ again.
Using $\nabla$ and the chirality operator $\Gamma$ (cf. (\ref{E:2.1})),
we define the odd signature operator $\B$ by $\B = \nabla \Gamma + \Gamma \nabla$,
which is a Dirac type operator.
It is a well-known fact that the odd signature operator $\B$ may not admit a local elliptic
boundary condition (\cite{GS}), which implies that we need to find a well-posed boundary condition for $\B$ in the sense of R. Seeley (\cite{Gr}, \cite{Se}).
The APS boundary condition is a well-posed boundary condition for a Dirac type operator, but
simple computation shows that the APS boundary condition does not fit to the de Rham operator nor the analytic torsion
and hence it is not a proper one for the refined analytic torsion.

In this paper we introduce new boundary conditions ${\mathcal P}_{-, {\mathcal L}_{0}}$ and ${\mathcal P}_{+, {\mathcal L}_{1}}$,
which are complementary to each other and are defined
by using the Hodge decomposition of $\Omega^{\bullet}(Y, E|_{Y})$ and the symplectic structure of $H^{\bullet}(Y, E|_{Y})$
under the assumption A in Section 2.3.
We show that they are well-posed boundary conditions for $\B$ in the sense of R. Seeley (\cite{Gr}, \cite{Se}).
It is not difficult to see that these boundary conditions fit to the de Rham operator and the odd signature operator $\B$
(Lemma \ref{Lemma:2.12})
and induce the following cochain complex (see (\ref{E:4.2}))

\begin{equation} \label{E:1.1}
 0 \longrightarrow \Omega^{0, \infty}_{{\mathcal P}_{-, {\mathcal L}_{0}}}(M,E) \stackrel{\nabla}{\longrightarrow}
\Omega^{1, \infty}_{{\mathcal P}_{+, {\mathcal L}_{1}}}(M,E) \stackrel{\nabla}{\longrightarrow} \cdots
\stackrel{\nabla}{\longrightarrow}\Omega^{m-1, \infty}_{{\mathcal P}_{-, {\mathcal L}_{0}}}(M,E)
\stackrel{\nabla}{\longrightarrow}\Omega^{m, \infty}_{{\mathcal P}_{+, {\mathcal L}_{1}}}(M,E) \longrightarrow 0.
\end{equation}

\noindent
The boundary conditions ${\mathcal P}_{-, {\mathcal L}_{0}}$, ${\mathcal P}_{+, {\mathcal L}_{1}}$ and
the realizations $\B_{{\mathcal P}_{-, {\mathcal L}_{0}}}$, $\B_{{\mathcal P}_{+, {\mathcal L}_{1}}}$ of $\B$ satisfy the relations
${\mathcal P}_{-, {\mathcal L}_{0}} \Gamma^{Y} = \Gamma^{Y} {\mathcal P}_{+, {\mathcal L}_{1}}$ and
$\B_{{\mathcal P}_{-, {\mathcal L}_{0}}} \Gamma = \Gamma \B_{{\mathcal P}_{+, {\mathcal L}_{1}}}$ ((\ref{E:2.16}), (\ref{E:2.19})),
which shows that (\ref{E:1.1}) satisfies the Poincar\'e duality.

We show that the realizations $\B_{{\mathcal P}_{-, {\mathcal L}_{0}}/ {\mathcal P}_{+, {\mathcal L}_{1}}}$ and
$\B^{2}_{{\mathcal P}_{-, {\mathcal L}_{0}}/ {\mathcal P}_{+, {\mathcal L}_{1}}}$ have their spectra in an arbitrarily small sector containing
the real axis except only finitely many ones (Theorem \ref{Theorem:2.22}),
which shows that we can choose Agmon angles for the zeta and eta functions arbitrarily close to any given angle $\phi$ with $ - \frac{\pi}{2} < \phi < 0$.
We use the method of the asymptotic expansions of the trace of heat kernels to show that
the zeta-determinant of $\B^{2}$ and the eta invariant of $\B$ subject to ${\mathcal P}_{-, {\mathcal L}_{0}}$
and ${\mathcal P}_{+, {\mathcal L}_{1}}$ are well-defined,
from which we define the graded determinant of $\B$ acting on even forms under these boundary conditions and finally define
the refined analytic torsion for the complex (\ref{E:1.1}) on a compact oriented Riemannian manifold with boundary.

The boundary conditions ${\mathcal P}_{-, {\mathcal L}_{0}}$ and ${\mathcal P}_{+, {\mathcal L}_{1}}$ are comparable with the absolute and
relative boundary conditions in the following sense. For each $0 \leq q \leq m$,
$\ker \Delta_{q, {\mathcal P}_{-, {\mathcal L}_{0}}} = {\mathcal H}^{q}_{\rel}(M, E) \hspace{0.1 cm} \cong  \hspace{0.1 cm}  H^{q}(M, Y ; \rho)$ and
$\ker \Delta_{q, {\mathcal P}_{+, {\mathcal L}_{1}}} = {\mathcal H}^{q}_{\Abs}(M, E) \hspace{0.1 cm} \cong  \hspace{0.1 cm} H^{q}(M ; \rho)$
(Lemma \ref{Lemma:2.11}). In particular, if $\nabla$ is a Hermitian connection, then $\Delta_{q, {\mathcal P}_{-, {\mathcal L}_{0}}}=\B^2_{q, {\mathcal P}_{-, {\mathcal L}_{0}}}$ and $\Delta_{q, {\mathcal P}_{+, {\mathcal L}_{1}}}=\B^2_{q, {\mathcal P}_{+, {\mathcal L}_{1}}}$.

When the odd signature operator $\B$ is defined from an acyclic Hermitian connection on a closed manifold, the refined analytic torsion is a complex number
whose modulus part is the Ray-Singer analytic torsion and the phase part is the rho invariant.
In this point of view we compared, in \cite{HL1} and \cite{HL2}, the analytic torsions subject to
${\mathcal P}_{-, {\mathcal L}_{0}}/{\mathcal P}_{+, {\mathcal L}_{1}}$ with the analytic torsion subject to the absolute/relative boundary conditions
(Theorem \ref{Theorem:4.1111}).
We also compared the eta invariant of $\B_{\even, {\mathcal P}_{-, {\mathcal L}_{0}}/ {\mathcal P}_{+, {\mathcal L}_{1}}}$
with the eta invariant of $\B_{\even, \Pi_{>, {\mathcal L}_{0}}/\Pi_{>, {\mathcal L}_{1}}}$, where $\Pi_{>, {\mathcal L}_{0}}$ and $\Pi_{>, {\mathcal L}_{1}}$
are generalized APS boundary conditions (Theorem \ref{Theorem:4.1112}).
Using these results, we proved in \cite{HL1} the gluing formula of the refined analytic torsion with
respect to
${\mathcal P}_{-, {\mathcal L}_{0}}$ and ${\mathcal P}_{+, {\mathcal L}_{1}}$ (Theorem \ref{Theorem:4.1113}).

Our paper is organized as follows.
In Section 2 we introduce the boundary conditions ${\mathcal P}_{-, {\mathcal L}_{0}}$ and ${\mathcal P}_{+, {\mathcal L}_{1}}$ for $\B$
under the Assumption A and
show that they are well posed for $\B$.
In Section 3 we show that the zeta functions associated to
$\B^{2}_{q, {\mathcal P}_{-, {\mathcal L}_{0}}/ {\mathcal P}_{+, {\mathcal L}_{1}}}$ acting on $q$-forms and
the eta functions associated to $\B_{\even, {\mathcal P}_{-, {\mathcal L}_{0}}/ {\mathcal P}_{+, {\mathcal L}_{1}}}$ acting on even forms
are regular at $s=0$ by computing the asymptotic expansions of the traces of the heat kernels.
In Section 4 we define the graded determinant of $\B_{\even, {\mathcal P}_{-, {\mathcal L}_{0}}/ {\mathcal P}_{+, {\mathcal L}_{1}}}$
and define the refined analytic torsion under ${\mathcal P}_{-, {\mathcal L}_{0}}/ {\mathcal P}_{+, {\mathcal L}_{1}}$
when the cochain complex (\ref{E:1.1}) is acyclic
and $\B_{\even, {\mathcal P}_{-, {\mathcal L}_{0}}/ {\mathcal P}_{+, {\mathcal L}_{1}}}$ is invertible.
We then show that the refined analytic torsion is invariant on the change of metrics in the interior of the manifold.
In Section 5 we finally define the refined analytic torsion as an element of the determinant line $\Det(H^{\bullet}(M, E))$.

As related works, B. Vertman has already studied the refined analytic torsion on a compact Riemannian manifold with boundary (\cite{Ve1}, \cite{Ve2})
but our approach is completely different from what he presented.
The comparison of these two constructions has been discussed in \cite{HL2}.
Burghelea and Haller have studied the complex-valued analytic torsion associated to a non-degenerate
symmetric bilinear form on a flat vector bundle (\cite{BH1}, \cite{BH2}), which we call the Burghelea-Haller analytic torsion.
Cappell and Miller used non-self-adjoint Laplace operator to define another complex valued analytic torsion and proved the extension of
the Cheeger-M\"uller theorem (\cite{CM}), which we call the Cappell-Miller analytic torsion.
Inspired by the result of B. Vertman, G. Su (\cite{Su}) and the first author (\cite{Hu})
studied the Burghelea-Haller analytic torsion and the Cappell-Miller analytic torsion, respectively,
on a compact oriented Riemannian manifolds with boundary.
O. M. Molina (\cite{Mo}) discussed the Burghelea-Haller analytic torsion on a compact Riemannian manifold with boundary by using the
relative/absolute boundary conditions.

The original form of this paper appeared in 2010 in arXiv. Recently the authors rewrote the paper more precisely under the Assumption A.
\vspace{0.2 cm}


\section{The boundary conditions ${\mathcal P}_{-, {\mathcal L}_{0}}$ and ${\mathcal P}_{+, {\mathcal L}_{1}}$ for the odd signature operator $\B$}

We begin this section by describing $\B$ on a compact manifold with boundary
when the product metric is given near the boundary. And then, we discuss the boundary conditions
${\mathcal P}_{-, {\mathcal L}_{0}}$ and ${\mathcal P}_{+, {\mathcal L}_{1}}$.

\subsection{The odd signature operator on a manifold with boundary}
Let $(M, Y, g^{M})$ be a compact oriented odd dimensional Riemannian
manifold with boundary $Y$, where $g^{M}$ is assumed to be a product metric near $Y$.
We denote the dimension of $M$ by $m = 2r - 1$. Suppose that $\rho : \pi_{1}(M) \rightarrow
GL(n, {\Bbb C})$ is a representation and $E = {\widehat M} \times_{\rho} {\Bbb C}^{n}$ is the associated flat vector bundle
with the flat connection $\nabla$, where ${\widehat M}$ is a universal covering space of $M$.
Without loss of generality, we may assume that $\nabla$ is a flat connection in temporal gauge (Section 3 in \cite{Ve2}).

We extend $\nabla$ to the de Rham operator
$$
\nabla : \Omega^{\bullet}(M, E) \rightarrow \Omega^{\bullet + 1}(M, E).
$$
Using the Hodge star operator $\ast_{M}$, we define the involution
$\Gamma = \Gamma(g^{M}) : \Omega^{\bullet}(M, E) \rightarrow \Omega^{m - \bullet}(M, E)$ by
\begin{equation}\label{E:2.1}
\Gamma \omega := i^{r} (-1)^{\frac{q(q+1)}{2}} \ast_{M} \omega, \qquad \omega \in \Omega^{q}(M, E),
\end{equation}
where $r = \frac{m+1}{2}$.
It is straightforward that $\Gamma^{2} = \Id$.
We define the odd signature operator $\B$ by
\begin{equation}\label{E:2.2}
\B \ = \  \B(\nabla,g^M) \ := \ \Gamma\,\nabla \ + \ \nabla\,\Gamma:\,\Omega^\bullet(M,E)\ \longrightarrow \  \Omega^\bullet(M,E).
\end{equation}
Then $\B$ is an elliptic differential operator of order $1$.
Let $N$ be a collar neighborhood of $Y$ which is isometric to $[0, \epsilon_{0}) \times Y$ for some $\epsilon_{0} > 0$. We have a natural isomorphism
\begin{equation}\label{E:2.3}
\Psi : \Omega^{p}(N, E|_{N}) \rightarrow C^{\infty}([0, \epsilon_{0}), \Omega^{p}(Y, E|_{Y}) \oplus \Omega^{p-1}(Y, E|_{Y})),
\end{equation}
defined by $\Psi(\omega_{1} + du \wedge \omega_{2}) = (\omega_{1}, \omega_{2})$,
where $u$ is the coordinate normal to $Y$ on $N$.
We denote by $\nabla^{Y}$ the restriction of $\nabla$ on $E|_{Y}$, and define the Hodge
star operator $\ast_{Y} : \Omega^{\bullet}(Y, E|_{Y}) \rightarrow
\Omega^{m-1-\bullet}(Y, E|_{Y})$ from the orientation defined by $d \Vol(M) = du \wedge d \Vol(Y)$.
We define two maps $\beta$, $\Gamma^{Y}$ by

\vspace{0.2 cm}

\begin{equation}\label{E:2.4}
 \begin{aligned}
  \beta  & :  \Omega^{p}(Y, E|_{Y}) \rightarrow  \Omega^{p}(Y, E|_{Y}), \quad \beta(\omega) = (-1)^{p} \omega \\
  \Gamma^{Y}  & :  \Omega^{p}(Y, E|_{Y}) \rightarrow  \Omega^{m-1-p}(Y,
E|_{Y}),\quad \Gamma^{Y}(\omega) = i^{r-1} (-1)^{\frac{p(p+1)}{2}}
\ast_{Y} \omega.
\end{aligned}
\end{equation}

\vspace{0.2 cm}

\noindent
It is straightforward that
\begin{equation}\label{E:2.5}
\beta^{2} = \Id,  \qquad  \left( ~ \Gamma^{Y} ~ \right)^{2} = \Id.
\end{equation}
Simple computation shows that
\begin{equation}\label{E:2.6}
 \Gamma = i \beta \Gamma^{Y} \left( \begin{array}{clcr} 0 & -1 \\ 1 & 0 \end{array} \right), \qquad
\nabla = \left( \begin{array}{clcr} 0 & 0 \\ 1 & 0 \end{array} \right) \nabla_{\partial_{u}} +
\left( \begin{array}{clcr} 1 & 0 \\ 0 & -1 \end{array} \right) \nabla^{Y}.
\end{equation}

\noindent
Hence the odd signature operator $\B$ is expressed, on $N$, by

\begin{equation}\label{E:2.8}
 \B = - i \beta \Gamma^{Y} \left\{ \left( \begin{array}{clcr} 1 & 0 \\
0 & 1 \end{array} \right) \nabla_{\partial_{u}} + \left( \begin{array}{clcr}  0 & -1
\\ -1 & 0 \end{array} \right) \left( \nabla^{Y} + \Gamma^{Y} \nabla^{Y} \Gamma^{Y} \right) \right\} \hspace{0.1 cm} = \hspace{0.1 cm}
{\mathcal \gamma} \left( \nabla_{\partial_{u}} + {\mathcal A} \right),
\end{equation}
where
\begin{equation}  \label{E:2.888}
{\mathcal \gamma} :=  - i \beta \Gamma^{Y}, \qquad {\mathcal A} := \left( \begin{array}{clcr}  0 & -1
\\ -1 & 0 \end{array} \right) \left( \nabla^{Y} + \Gamma^{Y} \nabla^{Y} \Gamma^{Y} \right).
\end{equation}
The equations (\ref{E:2.4}) and (\ref{E:2.5}) show that
\begin{equation}\label{E:2.9}
{\mathcal \gamma}^{2} = - \Id, \qquad {\mathcal \gamma} {\mathcal A} = - {\mathcal A} {\mathcal \gamma} .
\end{equation}
Since $\nabla_{\partial_{u}} \nabla^{Y} = \nabla^{Y} \nabla_{\partial_{u}}$, we have
\begin{equation}\label{E:2.10}
\B^{2} = - \left( \begin{array}{clcr}  1 & 0 \\ 0 & 1 \end{array} \right) \nabla_{\partial_{u}}^{2} +
\left( \begin{array}{clcr}  1 & 0 \\ 0 & 1 \end{array} \right) \left( \nabla^{Y} + \Gamma^{Y} \nabla^{Y} \Gamma^{Y} \right)^{2} =
\left( - \nabla_{\partial_{u}}^{2} + \B_{Y}^{2} \right) \left( \begin{array}{clcr}  1 & 0 \\ 0 & 1 \end{array} \right),
\end{equation}
where
$$
\B_{Y} = \Gamma^{Y} \nabla^{Y} + \nabla^{Y} \Gamma^{Y}.
$$

\vspace{0.3 cm}

\subsection{Green formula for the odd signature operator}

\vspace{0.2 cm}

We choose a fiber metric $h^{E}$ for the flat bundle $E$ so that
together with the Riemannian metric $g^{M}$ we define an
$L^{2}$-inner product $\langle \hspace{0.1 cm}, \hspace{0.1 cm} \rangle_{M}$ on $\Omega^{\bullet}(M, E)$.
We assume that on a collar neighborhood $N$ $h^{E}$ satisfies the following property: $h^{E}|_{N} = \pi_{Y}^{\ast} h^{E}|_{Y}$,
where $\pi_{Y} : Y \times [0, ~ \epsilon_{0}) \rightarrow Y$ is the natural projection.
We choose the dual connection $\nabla^{\prime}$ with respect to
$h^{E}$ satisfying the following property. For $\phi$, $\psi \in
C^{\infty}(M, E)$,
$$
d ( h^{E}(\phi, \psi) ) = h^{E}(\nabla\phi, \psi) + h^{E}(\phi,
\nabla^{\prime}\psi).
$$
We then extend $\nabla^{\prime}$ to the de Rham operator
$\nabla^{\prime} : \Omega^{\bullet}(M, E) \rightarrow \Omega^{\bullet + 1}(M, E)$.
The following lemma shows that the formal adjoint of $\nabla$ is $\Gamma \nabla^{\prime} \Gamma$
and the formal adjoint of $\B$ is $\B^{\prime} := \Gamma \nabla^{\prime} + \nabla^{\prime} \Gamma$.
When $\nabla = \nabla^{\prime}$, we call $\nabla$ a Hermitian connection with respect to the metric $h^{E}$.
We also define an $L^{2}$-inner product $\langle \hspace{0.1 cm}, \hspace{0.1 cm} \rangle_{Y}$ on $\Omega^{\bullet}(Y, E|_{Y})$
using $g^{Y}$, $\ast_{Y}$ and $h^{E}|_{Y}$, where $g^{Y}$, $\ast_{Y}$ are the metric and Hodge star operator on $Y$ induced from
$g^{M}$ and $\ast_{M}$.

\begin{definition} \label{Definition:2.1}
For $\phi \in \Omega^{q}(M, E)$ we write $\phi$, on a collar neighborhood $[0, \epsilon_{0}) \times Y$, by
$$\phi = \phi_{\Tan} + du \wedge \phi_{\Nor},$$
where $\phi_{\Tan} \in C^{\infty}([0, \epsilon_{0}), \Omega^{q}(Y, E|_{Y}))$ and
$\phi_{\Nor} \in C^{\infty}([0, \epsilon_{0}), \Omega^{q-1}(Y, E|_{Y}))$.
For the natural inclusion $\iota : Y \rightarrow M$,
we denote $~ \phi_{\Tan}|_{Y} := \iota^{\ast} \phi ~$ and $~ \psi_{\Nor}|_{Y} := (-1)^{q-1} \ast_{Y} \iota^{\ast} \left( \ast_{M} \phi \right)$.
\end{definition}

\noindent
The following lemma is well-known.

\vspace{0.2 cm}

\begin{lemma} \label{Lemma:2.2}
(1) For $\phi \in \Omega^{q}(M, E)$, $\psi \in \Omega^{m-q}(M,
\hspace{0.1 cm} E)$, $\langle \Gamma \phi, \hspace{0.1 cm} \psi
\rangle_{M} = \langle \phi, \hspace{0.1 cm} \Gamma \psi
\rangle_{M}$.  \newline (2) For $\phi \in \Omega^{q}(M, E)$, $\psi
\in \Omega^{q+1}(M, E)$,
$$
\langle \nabla \phi, \hspace{0.1 cm} \psi \rangle_{M} \hspace{0.1 cm} = \hspace{0.1 cm} \langle
\phi, \hspace{0.1 cm} \Gamma \nabla^{\prime} \Gamma \psi
\rangle_{M} \hspace{0.1 cm} + \hspace{0.1 cm}
\langle \phi_{\Tan}|_{Y}, \hspace{0.1 cm} \psi_{\Nor}|_{Y} \rangle_{Y}.
$$
\end{lemma}

\begin{proof}
The first assertion is straightforward. For the second statement
we may assume that
$$
\phi = \omega \otimes {\frak a}, \quad \psi = \xi \otimes {\frak b},
$$
where $\omega \in \Omega^{q}(M)$, $\xi \in \Omega^{q+1}(M)$, and
${\frak a}$, ${\frak b} \in C^{\infty}(M, E)$.
Then simple computation shows that
\begin{eqnarray*}
\langle \nabla \phi, \hspace{0.1 cm} \psi \rangle_{M} & = &
\int_{M} d \left( h^{E}({\frak a}, {\frak b}) \omega \wedge \ast_{M} \xi
\right) - \int_{M}  h^{E}({\frak a}, \nabla^{\prime} {\frak b}) \wedge \omega \wedge
\ast_{M} \xi
- (-1)^{q} \int_{M}  h^{E}({\frak a}, {\frak b}) \omega \wedge d \ast_{M} \xi  \\
& = & \int_{\partial M} h^{E}({\frak a}, {\frak b}) \omega \wedge \ast_{M}
\xi  - \int_{M}  h^{E}({\frak a}, \nabla^{\prime} {\frak b}) \wedge \omega \wedge
\ast_{M} \xi + \int_{M}  h^{E}({\frak a}, {\frak b}) \omega \wedge \ast_{M} d^{\ast} \xi.
\end{eqnarray*}
On the other hand, we note that $\Gamma \psi = \Gamma \xi \otimes {\frak b}$ and $\nabla^{\prime}\Gamma \psi = d \Gamma \xi \otimes
{\frak b} + (-1)^{q} \Gamma \xi \wedge \nabla^{\prime} {\frak b}$, which leads to

\begin{eqnarray*}
\Gamma \nabla^{\prime} \Gamma \psi & = & \Gamma d \Gamma \xi
\otimes {\frak b} +
(-1)^{q} \Gamma ( \Gamma \xi \wedge \nabla^{\prime} {\frak b}) \hspace{0.1 cm} = \hspace{0.1 cm}  d^{\ast} \xi \otimes {\frak b} - \ast_{M} ( (\ast_{M} \xi)
\wedge \nabla^{\prime} {\frak b} ).
\end{eqnarray*}

\noindent
By direct computation we have
$$
\langle \phi, \hspace{0.1 cm} \Gamma \nabla^{\prime} \Gamma \psi
\rangle_{M} = \int_{M} h^{E}({\frak a}, {\frak b}) \omega \wedge \ast_{M}
d^{\ast} \xi - \int_{M} h^{E}({\frak a}, \nabla^{\prime} {\frak b}) \wedge \omega \wedge \ast_{M} \xi,
$$
which proves the second statement.
\end{proof}

\vspace{0.2 cm}


\noindent
The following corollary gives the Green formula for the odd signature operator $\B$.

\begin{corollary} \label{Corollary:2.3}
(1) For $\phi \in \Omega^{q}(M, E)$, $\psi \in \Omega^{m-q-1}(M,E)$,
$$
\langle \Gamma \nabla \phi, \hspace{0.1 cm} \psi \rangle_{M} \hspace{0.1 cm}
= \hspace{0.1 cm} \langle \phi, \hspace{0.1 cm} \Gamma \nabla^{\prime} \psi \rangle_{M} \hspace{0.1 cm} + \hspace{0.1 cm}
\langle \phi_{\Tan}|_{Y}, \hspace{0.1 cm} i \beta \Gamma^{Y} \psi_{\Tan}|_{Y} \rangle_{Y}.
$$
(2) For $\phi \in \Omega^{q}(M, E)$, $\psi \in \Omega^{m-q+1}(M,E)$,
$$
\langle \nabla \Gamma \phi, \hspace{0.1 cm} \psi \rangle_{M} \hspace{0.1 cm} = \hspace{0.1 cm}
\langle \phi, \hspace{0.1 cm} \nabla^{\prime} \Gamma \psi \rangle_{M} \hspace{0.1 cm} + \hspace{0.1 cm}
\langle \phi_{\Nor}|_{Y}, \hspace{0.1 cm} i \beta \Gamma^{Y} \psi_{\Nor}|_{Y} \rangle_{Y}.
$$
(3) For $\phi$, $\psi \in \Omega^{\even}(M, E)$ or $\Omega^{\odd}(M,E)$,
$$
\langle \B \phi, \hspace{0.1 cm} \psi \rangle_{M} \hspace{0.1 cm} = \hspace{0.1 cm}
\langle \phi, \hspace{0.1 cm} \B^{\prime} \psi \rangle_{M} \hspace{0.1 cm} + \hspace{0.1 cm}
\langle \phi_{\Tan}|_{Y}, \hspace{0.1 cm} i \beta \Gamma^{Y} (\psi_{\Tan}|_{Y}) \rangle_{Y} \hspace{0.1 cm} + \hspace{0.1 cm}
\langle \phi_{\Nor}|_{Y}, \hspace{0.1 cm} i \beta \Gamma^{Y} (\psi_{\Nor}|_{Y}) \rangle_{Y}.
$$
\end{corollary}

\vspace{0.3 cm}

\subsection{Boundary conditions for $\B$.}

In this subsection we define boundary conditions ${\mathcal P}_{-, {\mathcal L}_{0}}$ and ${\mathcal P}_{+, {\mathcal L}_{1}}$ under some assumption on the boundary $Y$ of $M$.
We begin with analyzing the flat bundle $E|_{Y} \rightarrow Y$.
Let $p_{Y} : {\widehat Y} \rightarrow Y$ be a universal covering space of $Y$ and
$\iota : Y \rightarrow M$ be the natural inclusion with $\iota_{\sharp} : \pi_{1}(Y) \rightarrow \pi_{1}(M)$.
From the free action of $\pi_{1}(Y)$ on ${\widehat Y}$, we put $Z := {\widehat Y}/\Ker \iota_{\sharp}$.
Then, $Z \rightarrow Y$ is a covering space of $Y$, where $\pi_{1}(Z) \cong \Ker \iota_{\sharp}$
and each fiber of $y \in Y$ is isomorphic to $\pi_{1}(Y)/\Ker \iota_{\sharp}$.
For a representation $\rho : \pi_{1}(M) \rightarrow GL(n, {\mathbb C})$, we consider representations
$\rho \circ \iota_{\sharp} : \pi_{1}(Y) \rightarrow GL(n, {\mathbb C})$ and
${\widetilde{ \rho \circ \iota_{\sharp}}} : \pi_{1}(Y)/\Ker \iota_{\sharp} \rightarrow GL(n, {\mathbb C})$.
Since $E|_{Y} \rightarrow Y$ is a flat vector bundle with the flat connection $\nabla^{Y}$, $E|_{Y}$ can be obtained by some representation of
$\pi_{1}(Y) \rightarrow GL(n, {\mathbb C})$, which is $\rho \circ \iota_{\sharp}$.
We recall that $\pi : {\widehat M} \rightarrow M$ is a universal covering space of $M$ and put $\pi^{-1}(Y) = \cup_{\alpha} {\widetilde Y}_{\alpha}$,
where each ${\widetilde Y}_{\alpha}$ is a component of $\pi^{-1}(Y)$.

\begin{lemma}
Each $q_{\alpha} : {\widetilde Y}_{\alpha} \rightarrow Y$ is a covering space of $Y$ with $\pi_{1}({\widetilde Y}_{\alpha}) = \Ker \iota_{\sharp}$
 and fiber $\pi_{1}(Y)/\Ker \iota_{\sharp}$, where $q_{\alpha} = \pi|_{{\widetilde Y}_{\alpha}}$. Moreover, each ${\widetilde Y}_{\alpha}$
 is isomorphic to $Z$ as covering spaces of $Y$.
\end{lemma}

\begin{proof}
It is straightforward that $q_{\alpha} : {\widetilde Y}_{\alpha} \rightarrow Y$ is a covering space of $Y$.
We fix ${\widetilde y}_{\alpha} \in {\widetilde Y}_{\alpha}$ and $y_{0} \in Y$ with $q_{\alpha}({\widetilde y}_{\alpha}) = y_{0}$.
For $[\beta] \in \pi_{1}(Y, y_{0})$, we consider $\iota_{\sharp}([\beta])) \in \pi_{1}(M, y_{0})$.
We lift the curve $\beta(t)$ to ${\widetilde \beta} : [0, ~1] \rightarrow {\widehat M}$ with
$\pi({\widetilde \beta}(t)) = \beta(t)$ and  ${\widetilde \beta} (0) = {\widetilde y}_{\alpha}$.
By connectedness, ${\widetilde \beta}(t) \in {\widetilde Y}_{\alpha}$ and hence ${\widetilde \beta}(t)$ is a lift of $\beta(t)$ to ${\widetilde Y}_{\alpha}$.
Moreover, $[\beta] \in \Ker \iota_{\sharp} \subset \pi_{1}(Y)$ if and only if ${\widetilde \beta} (1) = {\widetilde y}_{\alpha}$.
In general, ${\widetilde \beta} (1) = \iota_{\sharp}([\beta]) ({\widetilde y}_{\alpha})$, the action of $\iota_{\sharp}([\beta])$ to ${\widetilde y}_{\alpha}$,
which shows that the fiber of $q_{\alpha} : {\widetilde Y}_{\alpha} \rightarrow Y$ has one to one correspondence with $\pi_{1}(Y, y_{0})/\Ker \iota_{\sharp}$.
It is straightforward that ${\widetilde Y}_{\alpha}$ is isomorphic to $Z$ as covering spaces of $Y$.
\end{proof}

\begin{corollary}
\begin{eqnarray*}
E|_{Y} = p_{M}^{-1}(Y) \times_{\rho} {\mathbb C}^{n} & \cong & {\widetilde Y}_{\alpha} \times_{{\widetilde{ \rho \circ \iota_{\sharp}}}} {\mathbb C}^{n}
~ \cong ~ {\widehat Y} \times_{\rho \circ \iota_{\sharp}} {\mathbb C}^{n}.
\end{eqnarray*}
\end{corollary}

\vspace{0.2 cm}

All through this paper, we assume the following assumption.

\vspace{0.2 cm}
\noindent
{\bf Assumption A} : The representation $\rho \circ \iota_{\sharp} : \pi_{1}(Y) \rightarrow GL(n, {\mathbb C})$ is equivalent to a
unitary representation.

\vspace{0.2 cm}
\noindent
{\it Remark} : (1) Every Hermitian connection $\nabla$ satisfies the Assumption A (Proposition 1.4.21 in \cite{Ko}).

\noindent
(2) If $\iota_{\sharp} (\pi_{1}(Y))$ is a finite subgroup of $\pi_{1}(M)$, the Assumption A is satisfied.

\vspace{0.2 cm}

If the Assumption A is satisfied,
there exists a Hermitian fiber metric ${\frak h}^{Y}$ such that $\nabla^{Y}$ is a Hermitian connection with respect to ${\frak h}^{Y}$ on
$\Omega^{\bullet}(Y, E|_{Y})$ and hence $\Omega^{\bullet}(Y, E|_{Y})$ satisfies the following Hodge decomposition.
\begin{eqnarray}    \label{E:2.1000}
\Omega^{q}(Y, E|_{Y}) & = & \Omega^{q}_{-}(Y, E|_{Y}) \hspace{0.1 cm} \oplus \hspace{0.1 cm}  \Omega^{q}_{0}(Y, E|_{Y})  \oplus  \hspace{0.1 cm}
\Omega^{q}_{+}(Y, E|_{Y}) ,
\end{eqnarray}
where
\begin{eqnarray}    \label{E:2.1001}
& & \Omega^{q}_{-}(Y, E|_{Y}) ~ = ~ \nabla^{Y} \left( \Omega^{q-1}(Y, E|_{Y}) \right), \qquad
\Omega^{q}_{+}(Y, E|_{Y}) ~ = ~ \Gamma^{Y} \nabla^{Y} \Gamma^{Y} \left( \Omega^{q+1}(Y, E|_{Y}) \right),  \nonumber \\
& & \Omega_{0}^{q}(Y, E|_{Y}) ~ = ~ \Ker \nabla^{Y} \cap \Ker \Gamma^{Y} \nabla^{Y} \Gamma^{Y} \cap \Omega^{q}(Y, E|_{Y}) =:
{\mathcal H}^{q}(Y, E|_{Y}).
\end{eqnarray}

\noindent
We extend ${\frak h}^{Y}$ to $[0, ~ \epsilon_{0}) \times Y$ by using the product structure, which we denote it ${\frak h}^{Y}$ again.
Finally, we extend ${\frak h}^{Y}$ again arbitrary to obtain a Hermitian fiber metric ${\frak h}^{E}$.
In the remaining part of this paper, we fix a Hermitian fiber metric ${\frak h}^{E}$ of $E$ obtained in this way.
Let $\nabla^{\prime}$ be a dual connection with respect to ${\frak h}^{E}$. Then $\nabla$ and $\nabla^{\prime}$ induce the same flat connection $\nabla^{Y}$ on $Y_{\epsilon} := Y \times \{ \epsilon \}$, $0 \leq \epsilon < \epsilon_{0}$.
The connection $\nabla$ itself is not a Hermitian connection but its restriction $\nabla^{Y}$ to $Y$ is a Hermitian connection.

\vspace{0.2 cm}

We denote
\begin{eqnarray}     \label{E:2.1002}
{\widetilde \nabla} & = & \frac{1}{2} \left( \nabla + \nabla^{\prime} \right).
\end{eqnarray}
Then, ${\widetilde \nabla}$ is a Hermitian connection with respect to ${\frak h}^{E}$, which is not necessarily a flat connection.
If ${\widetilde \nabla} \phi = \Gamma {\widetilde \nabla} \Gamma \phi = 0$ for $\phi \in \Omega^{\bullet}(M, E)$,
simple computation shows that $\phi$ is expressed on $Y$ by

\begin{equation} \label{E:2.45}
\phi|_{Y} = \left( \nabla^{Y} \varphi_{1} + \varphi_{2} \right) + du \wedge \left( \Gamma^{Y} \nabla^{Y} \Gamma^{Y} \psi_{1} + \psi_{2} \right), \quad
\varphi_{1}, \hspace{0.1 cm} \psi_{1} \in \Omega^{\bullet}(Y, E|_{Y}),
\quad \varphi_{2}, \hspace{0.1 cm} \psi_{2} \in {\mathcal H}^{\bullet}(Y, E|_{Y}).
\end{equation}

\noindent
In other words, $\varphi_{2}$ and $\psi_{2}$ are harmonic parts of $\iota^{\ast} \phi$ and $\beta \left( \ast_{Y} \iota^{\ast} ( \ast_{M} \phi ) \right)$.
We denote ${\mathcal K}$ by

\begin{equation} \label{E:2.46}
{\mathcal K} := \{ \varphi_{2} \in {\mathcal H}^{\bullet}(Y, E|_{Y}) \mid {\widetilde \nabla} \phi = \Gamma {\widetilde \nabla} \Gamma \phi = 0 \},
\end{equation}

\noindent
where $\phi$ has the form (\ref{E:2.45}). The first assertion in Corollary \ref{Corollary:2.3} shows that
${\mathcal K}$ is perpendicular to $\Gamma^{Y} {\mathcal K}$.
If $\phi$ satisfies ${\widetilde \nabla} \phi = \Gamma {\widetilde \nabla} \Gamma \phi = 0$, so is $\Gamma \phi$. Hence

\begin{equation} \label{E:2.47}
\Gamma^{Y} {\mathcal K} = \{ \psi_{2} \in {\mathcal H}^{\bullet}(Y, E|_{Y}) \mid {\widetilde \nabla} \phi = \Gamma {\widetilde \nabla} \Gamma \phi = 0 \},
\end{equation}

\noindent
where $\phi$ has the form (\ref{E:2.45}).
We then have the following lemma (cf. Corollary 8.4 in \cite{KL}).

\vspace{0.2 cm}

\begin{lemma} \label{Lemma:2.22}
We have the following equality.
$$
{\mathcal K} \oplus \Gamma^{Y} {\mathcal K} = {\mathcal H}^{\bullet}(Y, E|_{Y}).
$$
Hence,
$( {\mathcal H}^{\bullet}(Y, E|_{Y}), \hspace{0.1 cm} \langle \hspace{0.1 cm}, \hspace{0.1 cm} \rangle_{Y}, \hspace{0.1 cm} - i \beta \Gamma^{Y} )$
is a symplectic vector space with Lagrangian subspaces ${\mathcal K}$ and $\Gamma^{Y} {\mathcal K}$.
\end{lemma}

\begin{proof}
Since ${\mathcal K} \oplus \Gamma^{Y} {\mathcal K} \subset {\mathcal H}^{\bullet}(Y, E|_{Y})$, we have
$\dim {\mathcal K} \leq \frac{1}{2} \dim {\mathcal H}^{\bullet}(Y, E|_{Y})$.
For the opposite direction of the inequality
we are going to use the scattering theory for Dirac operators (\cite{Gu}, \cite{Mu}).
Let $M_{\infty} := M \cup_{Y} \left( ( - \infty , 0 ] \times Y \right)$. We extend $E$, ${\widetilde \nabla}$,
${\widetilde \B} := \Gamma {\widetilde \nabla} + {\widetilde \nabla} \Gamma$ to $M_{\infty}$ canonically,
which we denote by $E_{\infty}$, ${\widetilde \nabla}_{\infty}$, ${\widetilde \B}_{\infty}$.
Let ${\mathcal L}_{\lim}$ be the space of the limiting values of extended $L^{2}$-solutions of ${\widetilde \B}_{\infty}$.
We refer to \cite{APS}, \cite{BW}, \cite{Gu}, \cite{Mu} for the definitions of the limiting values and extended $L^{2}$-solutions of ${\widetilde \B}_{\infty}$.
Obviously,
$
{\mathcal L}_{\lim} \subset \left( \begin{array}{clcr} {\mathcal K} \\ \Gamma^{Y} {\mathcal K} \end{array} \right).
$
It is well known that $\hspace{0.1 cm} 2 \cdot \dim {\mathcal L}_{\lim} \hspace{0.1 cm}$ is equal to the dimension
of the kernel of the tangential operator of ${\widetilde \B}$ (cf. \cite{Gu}, \cite{Mu}).
Hence,

\begin{eqnarray*}
\frac{1}{2} \dim \left( \begin{array}{clcr} {\mathcal H}^{\bullet}(Y, E|_{Y}) \\ {\mathcal H}^{\bullet}(Y, E|_{Y}) \end{array} \right) & = &
\frac{1}{2} \dim \Ker \left( \begin{array}{clcr} 0 & -1 \\ -1 & 0 \end{array} \right) ( \nabla^{Y} + \Gamma^{Y} \nabla^{Y} \Gamma^{Y} ) \\
& = &  \dim {\mathcal L}_{\lim}  \hspace{0.1 cm} \leq \hspace{0.1 cm} \dim
\left( \begin{array}{clcr} {\mathcal K} \\ \Gamma^{Y} {\mathcal K} \end{array} \right) \hspace{0.1 cm} \leq
\hspace{0.1 cm} \frac{1}{2} \dim \left( \begin{array}{clcr} {\mathcal H}^{\bullet}(Y, E|_{Y}) \\ {\mathcal H}^{\bullet}(Y, E|_{Y}) \end{array} \right),
\end{eqnarray*}
from which the result follows.
\end{proof}

\vspace{0.2 cm}

\noindent
{\it Remark} : The above lemma shows that ${\mathcal K}$ and $\Gamma^{Y} {\mathcal K}$ are the sets of all tangential and normal parts of
the limiting values of extended $L^{2}$-solutions to ${\widetilde \B}_{\infty}$ on $M_{\infty}$, respectively (cf. Corollary 8.4 in \cite{KL}).

\vspace{0.2 cm}

We put

\begin{equation} \label{E:2.14}
{\mathcal L}_{0} := \left( \begin{array}{clcr} {\mathcal K} \\ {\mathcal K} \end{array} \right), \qquad
{\mathcal L}_{1} := \left( \begin{array}{clcr} \Gamma^{Y} {\mathcal K} \\ \Gamma^{Y} {\mathcal K} \end{array} \right)
\end{equation}

\noindent
and denote by ${\mathcal P}_{{\mathcal L}_{0}/{\mathcal L}_{1}}$ the orthogonal projections onto ${\mathcal L}_{0}/{\mathcal L}_{1}$.
We next define orthogonal projections ${\mathcal P}_{-}$ and ${\mathcal P}_{+}$ as follows.

\begin{eqnarray*}
{\mathcal P}_{-}, \hspace{0.2 cm} {\mathcal P}_{+} & : &
\Omega^{\bullet}(Y, E|_{Y}) \oplus \Omega^{\bullet}(Y, E|_{Y}) \rightarrow
\Omega^{\bullet}(Y, E|_{Y}) \oplus \Omega^{\bullet}(Y, E|_{Y})  \\
{\mathcal P}_{-} & = & \nabla^{Y} \Gamma^{Y} \nabla^{Y} \Gamma^{Y}
\left( \B_{Y}^{2} + \pr_{\Ker \B_{Y}^{2}} \right)^{-1} \B_{Y}^{2}
\left( \begin{array}{clcr}  1 & 0 \\ 0 & 1 \end{array} \right), \\
{\mathcal P}_{+} & = & \Gamma^{Y} \nabla^{Y} \Gamma^{Y} \nabla^{Y}
\left( \B_{Y}^{2} + \pr_{\Ker \B_{Y}^{2}} \right)^{-1} \B_{Y}^{2}
\left( \begin{array}{clcr}  1 & 0 \\ 0 & 1 \end{array} \right).
\end{eqnarray*}

\noindent
Then ${\mathcal P}_{-}$, ${\mathcal P}_{+}$, ${\mathcal P}_{{\mathcal L}_{0}}$, and ${\mathcal P}_{{\mathcal L}_{1}}$
are projections onto
$\left( \begin{array}{clcr} \Omega^{\bullet}_{-}(Y, E|_{Y}) \\ \Omega^{\bullet}_{-}(Y, E|_{Y}) \end{array} \right)$,
$\left( \begin{array}{clcr} \Omega^{\bullet}_{+}(Y, E|_{Y}) \\ \Omega^{\bullet}_{+}(Y, E|_{Y}) \end{array} \right)$,
$\left( \begin{array}{clcr} {\mathcal K} \\ {\mathcal K} \end{array} \right)$ and
$\left( \begin{array}{clcr} \Gamma^{Y} {\mathcal K} \\ \Gamma^{Y} {\mathcal K} \end{array} \right)$, respectively.
Moreover, ${\mathcal P}_{\pm}$ are $\Psi$DO's of order $0$ and
${\mathcal P}_{{\mathcal L}_{0}/{\mathcal L}_{1}}$ are smoothing operators.

\vspace{0.2 cm}

\noindent
We define
\begin{equation}\label{E:2.15}
{\mathcal P}_{-, {\mathcal L}_{0}} = {\mathcal P}_{-} + {\mathcal P}_{{\mathcal L}_{0}},   \qquad
{\mathcal P}_{+, {\mathcal L}_{1}} = {\mathcal P}_{+} + {\mathcal P}_{{\mathcal L}_{1}}.
\end{equation}

\noindent
Then ${\mathcal L}_{0}$, ${\mathcal L}_{1}$, ${\mathcal P}_{-, {\mathcal L}_{0}}$ and ${\mathcal P}_{+, {\mathcal L}_{1}}$ satisfy

\begin{equation}\label{E:2.16}
\Gamma {\mathcal L}_{0} = {\mathcal L}_{1}, \quad \Gamma {\mathcal L}_{1} = {\mathcal L}_{0}, \qquad
{\mathcal P}_{-, {\mathcal L}_{0}} \hspace{0.1 cm} \Gamma \hspace{0.1 cm} =
\hspace{0.1 cm} \Gamma \hspace{0.1 cm} {\mathcal P}_{+, {\mathcal L}_{1}}.
\end{equation}

\noindent
In the next subsection 2.5 we show that ${\mathcal P}_{-, {\mathcal L}_{0}}$ and ${\mathcal P}_{+, {\mathcal L}_{1}}$ give
well-posed boundary conditions for $\B$.

\vspace{0.2 cm}

We define the realizations $\B_{{\mathcal P}_{-, {\mathcal L}_{0}}}$, $\B^{2}_{{\mathcal P}_{-, {\mathcal L}_{0}}}$
by the operators $\B$, $\B^{2}$ with domains

\begin{eqnarray}\label{E:2.17}
\Dom \left( \B_{{\mathcal P}_{-, {\mathcal L}_{0}}} \right) & = &
\left\{ \psi \in \Omega^{\bullet}(M, E) \mid {\mathcal P}_{-, {\mathcal L}_{0}} \left( \psi|_{Y} \right) = 0 \right\},  \nonumber\\
\Dom \left( \B^{2}_{{\mathcal P}_{-, {\mathcal L}_{0}}} \right) & = &
\left\{ \psi \in \Omega^{\bullet}(M, E) \mid {\mathcal P}_{-, {\mathcal L}_{0}} \left( \psi|_{Y} \right) = 0, \hspace{0.1 cm}
{\mathcal P}_{-, {\mathcal L}_{0}} \left( (\B \psi)|_{Y} \right) = 0 \right\}.
\end{eqnarray}

\noindent
We define $\B_{{\mathcal P}_{+, {\mathcal L}_{1}}}$, $\B^{2}_{{\mathcal P}_{+, {\mathcal L}_{1}}}$,
$\B^{2}_{\rel}$, $\B^{2}_{\Abs}$, and $\B_{\Pi_{>, {\mathcal L}_{0}}}$, $\B_{\Pi_{>, {\mathcal L}_{1}}}$
(cf. Theorem 4.12) in a similar way.
The equality (\ref{E:2.16}) shows that $\Gamma$ maps $\Dom \left( \B_{{\mathcal P}_{-, {\mathcal L}_{0}}} \right)$ isomorphically onto
$\Dom \left( \B_{{\mathcal P}_{+, {\mathcal L}_{0}}} \right)$ and vice versa. Moreover, we have

\begin{equation}\label{E:2.19}
\B_{{\mathcal P}_{-, {\mathcal L}_{0}}} \hspace{0.1 cm} \Gamma  \hspace{0.1 cm} = \hspace{0.1 cm}
\Gamma \hspace{0.1 cm} \B_{{\mathcal P}_{+, {\mathcal L}_{1}}}.
\end{equation}

\vspace{0.2 cm}

The next lemma shows that the boundary conditions ${\mathcal P}_{-, {\mathcal L}_{0}}$ and
${\mathcal P}_{+, {\mathcal L}_{1}}$ fit well to the graded structure of $\Omega^{\bullet}(M, E)$ and
the odd signature operator $\B$.

\vspace{0.2 cm}

\begin{lemma}     \label{Lemma:2.12}
(1) If $\psi \in \Dom \left( \B^{2}_{{\mathcal P}_{-, {\mathcal L}_{0}}} \right)$, then
$\nabla \psi$ belongs to $\Dom \left( \B_{{\mathcal P}_{+, {\mathcal L}_{1}}} \right)$.   \newline
(2) If $\psi \in \Dom \left( \B^{2}_{{\mathcal P}_{-, {\mathcal L}_{0}}} \right)$ is an eigenform of $\B^{2}$ with eigenvalue
$\lambda^{2}$, then $\nabla \psi$ belongs to $\Dom \left( \B^{2}_{{\mathcal P}_{+, {\mathcal L}_{1}}} \right)$ and an eigenform of $\B^{2}$ with eigenvalue
$\lambda^{2}$.  \newline
(3) Let $\psi_{1}, \cdots, \psi_{k}$ be generalized eigenforms of $\B^{2}$ with generalized eigenvalue $\lambda^{2}$
in $\Dom \left( \B^{2}_{{\mathcal P}_{-, {\mathcal L}_{0}}} \right)$ satisfying $(\B^{2} - \lambda^{2}) \psi_{j} = \psi_{j-1}$,
where $\psi_{0} = 0$.
Then $\nabla \psi_{1}, \cdots, \nabla \psi_{k}$ belong to $\Dom \left( \B^{2}_{{\mathcal P}_{+, {\mathcal L}_{1}}} \right)$
and are generalized eigenforms of $\B^{2}$ with generalized eigenvalue $\lambda^{2}$. \newline
(4) Similar statements hold for $\Dom \left( \B^{2}_{{\mathcal P}_{+, {\mathcal L}_{1}}} \right)$.
\end{lemma}

\vspace{0.1 cm}

\begin{proof}
Since the proofs of the assertion (2) and (3) are similar to that of the assertion (1), we are going to prove (1).
We note that for $\psi \in \Omega^{\bullet}(M, E)$, $\psi$ is written on a collar neighborhood of $Y$ by
$$
\psi = \psi_{\Tan}^{-} + \psi_{\Tan}^{+} + du \wedge ( \psi_{\Nor}^{-} + \psi_{\Nor}^{+}),
$$
where $\left( \begin{array}{clcr} \psi_{\Tan}^{-} \\ \psi_{\Nor}^{-} \end{array} \right) \in
\Imm {\mathcal P}_{-, {\mathcal L}_{0}}$ and
$\left( \begin{array}{clcr} \psi_{\Tan}^{+} \\ \psi_{\Nor}^{+} \end{array} \right) \in \Imm {\mathcal P}_{+, {\mathcal L}_{1}}$.
The fact $\psi \in \Dom \left( \B_{{\mathcal P}_{-, {\mathcal L}_{0}}} \right)$ implies that

\begin{equation}\label{E:2.20}
\psi_{\Tan}^{-}|_{Y} = \psi_{\Nor}^{-}|_{Y}= 0.
\end{equation}

\noindent
We note that by (\ref{E:2.8})
\begin{eqnarray}\label{E:2.21}
  \B \psi  & = &
  - i \beta \Gamma^{Y} \left\{ \nabla_{\partial_{u}} +
\left( \begin{array}{clcr}  0 & -1 \\ -1 & 0 \end{array} \right)
\left( \nabla^{Y} + \Gamma^{Y} \nabla^{Y} \Gamma^{Y} \right) \right\}
\left(\begin{array}{clcr}  \psi_{\Tan}^{-} + \psi_{\Tan}^{+} \\
\psi_{\Nor}^{-} + \psi_{\Nor}^{+}   \end{array} \right) \nonumber \\
 & = &  - i \beta \Gamma^{Y}
\left(\begin{array}{clcr} \nabla_{\partial_{u}} \psi_{\Tan}^{-} + \nabla_{\partial_{u}} \psi_{\Tan}^{+}
- \nabla^{Y} \psi_{\Nor}^{+} - \Gamma^{Y} \nabla^{Y} \Gamma^{Y} \psi_{\Nor}^{-} \\
\nabla_{\partial_{u}} \psi_{\Nor}^{-} + \nabla_{\partial_{u}} \psi_{\Nor}^{+}
- \nabla^{Y} \psi_{\Tan}^{+} - \Gamma^{Y} \nabla^{Y} \Gamma^{Y} \psi_{\Tan}^{-} \end{array} \right).
\end{eqnarray}
The above equality with (\ref{E:2.20}) shows that
\begin{equation}\label{E:2.22}
(\B \psi)|_{Y} = - i \beta \Gamma^{Y}
\left(\begin{array}{clcr} \nabla_{\partial_{u}} \psi_{\Tan}^{-}|_{Y} + \nabla_{\partial_{u}} \psi_{\Tan}^{+}|_{Y}
- \nabla^{Y} \psi_{\Nor}^{+}|_{Y} \\
\nabla_{\partial_{u}} \psi_{\Nor}^{-}|_{Y} + \nabla_{\partial_{u}} \psi_{\Nor}^{+}|_{Y}
- \nabla^{Y} \psi_{\Tan}^{+}|_{Y}  \end{array} \right).
\end{equation}
The fact
$\B \psi \in \Dom \left( \B_{{\mathcal P}_{-, {\mathcal L}_{0}}} \right)$ implies
\begin{equation}\label{E:2.23}
\nabla_{\partial_{u}} \psi_{\Tan}^{+}|_{Y} = \nabla_{\partial_{u}} \psi_{\Nor}^{+}|_{Y} = 0.
\end{equation}
Hence, if $\psi \in \Dom \left( \B^{2}_{{\mathcal P}_{-, {\mathcal L}_{0}}} \right)$, then
(\ref{E:2.20}) and (\ref{E:2.23}) show that $\psi$ satisfies

\begin{equation}\label{E:2.24}
\psi_{\Tan}^{-}|_{Y} = \psi_{\Nor}^{-}|_{Y} = 0, \qquad
\nabla_{\partial_{u}} \psi_{\Tan}^{+}|_{Y} = \nabla_{\partial_{u}} \psi_{\Nor}^{+}|_{Y} = 0.
\end{equation}

\noindent
We next note that by (\ref{E:2.6})
\begin{eqnarray*}
 \nabla \psi & = & \left( \begin{array}{clcr} \nabla^{Y} & 0 \\
\nabla_{\partial_{u}} & - \nabla^{Y} \end{array} \right)
\left(\begin{array}{clcr}  \psi_{\Tan}^{-} + \psi_{\Tan}^{+} \\
\psi_{\Nor}^{-} + \psi_{\Nor}^{+}   \end{array} \right)  \hspace{0.1 cm}
 = \hspace{0.1 cm} \left(\begin{array}{clcr} \nabla^{Y} \psi_{\Tan}^{+} \\ \nabla_{\partial_{u}} \psi_{\Tan}^{-} + \nabla_{\partial_{u}} \psi_{\Tan}^{+}
- \nabla^{Y} \psi_{\Nor}^{+} \\
\end{array} \right).
\end{eqnarray*}
This equality together with (\ref{E:2.24}) leads to
$$
\left( \nabla \psi \right)|_{Y}
=  \left(\begin{array}{clcr} \nabla^{Y} \psi_{\Tan}^{+}|_{Y} \\ \nabla_{\partial_{u}} \psi_{\Tan}^{-}|_{Y} - \nabla^{Y} \psi_{\Nor}^{+}|_{Y} \\
\end{array} \right) \in \Imm {\mathcal P}_{-, {\mathcal L}_{0}},
$$
which shows that $\nabla \psi \in \Dom \left( \B_{{\mathcal P}_{+, {\mathcal L}_{1}}} \right)$.
\end{proof}

\vspace{0.2 cm}

The boundary conditions ${\mathcal P}_{-, {\mathcal L}_{0}}$ and ${\mathcal P}_{+, {\mathcal L}_{1}}$
have similar properties with the relative and absolute boundary conditions.

\begin{definition} \label{Definition:2.4}
Let $\phi \in \Omega^{q}(M, E)$ be expressed by $\phi = \phi_{1} + du \wedge \phi_{2}$
on a collar neighborhood of $Y$.   \newline
(1) $\phi$ satisfies the absolute boundary condition if $\phi_{2}|_{Y} = 0$ and
$(\nabla_{\partial_{u}} \phi_{1})|_{Y} = 0$.   \newline
(2) $\phi$ satisfies the relative boundary condition if $\phi_{1}|_{Y} = 0$ and
$(\nabla_{\partial_{u}} \phi_{2})|_{Y} = 0$.
\end{definition}

\noindent
We denote by $\Omega^{\bullet}_{\Abs/\rel} (M, E)$,
$\Omega^{\bullet}_{{\mathcal P}_{-, {\mathcal L}_{0}}/{\mathcal P}_{+, {\mathcal L}_{1}}} (M, E)$ the spaces of $E$-valued differential forms
satisfying the absolute/relative and ${\mathcal P}_{-, {\mathcal L}_{0}}/{\mathcal P}_{+, {\mathcal L}_{1}}$ boundary conditions, respectively.
Since $\Gamma \nabla^{\prime} \Gamma$ is the formal adjoint of $\nabla$, we define the Laplacian $\Delta_{q}$ acting on $q$-forms by
\begin{eqnarray}    \label{E:700}
\Delta_{q}& = & \Gamma \nabla^{\prime} \Gamma \nabla + \nabla \Gamma \nabla^{\prime} \Gamma.
\end{eqnarray}
If $\nabla$ is a Hermitian connection, then $\Delta_{q} = \B^{2}_{q}$. We denote by
\begin{eqnarray}   \label{E:701}
{\mathcal H}^{q}_{\Abs} (M, E) = \ker \nabla \cap \ker \Gamma \nabla^{\prime} \Gamma \cap \Omega^{q}_{\Abs} (M, E), \qquad
{\mathcal H}^{q}_{\rel} (M, E) = \ker \nabla \cap \ker \Gamma \nabla^{\prime} \Gamma \cap \Omega^{q}_{\rel} (M, E).
\end{eqnarray}

\noindent
Since $\Delta_{q}$ is same as $\B^{2}$ on a collar neighborhood of $Y$, we consider $\Delta^{q}$ on
$\Omega^{q}_{{\mathcal P}_{-, {\mathcal L}_{0}}} (M, E)$ and $\Omega^{q}_{{\mathcal P}_{+, {\mathcal L}_{1}}} (M, E)$.
Then we have the following result.

\vspace{0.2 cm}

\begin{lemma} \label{Lemma:2.11}
For each $q$, $\Delta_{q, {\mathcal P}_{-, {\mathcal L}_{0}}/{\mathcal P}_{+, {\mathcal L}_{1}}}$  is a self-adjoint operator on
$\Omega^{q}_{{\mathcal P}_{-, {\mathcal L}_{0}}/{\mathcal P}_{+, {\mathcal L}_{1}}} (M, E)$ and
$$
\ker \Delta_{q, {\mathcal P}_{-, {\mathcal L}_{0}}} = {\mathcal H}^{q}_{\rel}(M, E) \hspace{0.1 cm} \cong  \hspace{0.1 cm}  H^{q}(M, Y ; \rho), \qquad
\ker \Delta_{q, {\mathcal P}_{+, {\mathcal L}_{1}}} = {\mathcal H}^{q}_{\Abs}(M, E) \hspace{0.1 cm} \cong  \hspace{0.1 cm} H^{q}(M ; \rho).
$$
\end{lemma}

\begin{proof}
We are going to prove the lemma for the operator $\Delta_{q, {\mathcal P}_{-, {\mathcal L}_{0}}}$.
The same argument works for $\Delta_{q, {\mathcal P}_{+, {\mathcal L}_{1}}}$.
The self-adjointness comes from Corollary \ref{Corollary:2.3} and hence $\ker \Delta_{q, {\mathcal P}_{-, {\mathcal L}_{0}}} =
\ker \nabla \cap \ker \Gamma \nabla^{\prime} \Gamma \cap \Omega^{q}_{{\mathcal P}_{-, {\mathcal L}_{0}}}(M, E)$.
Let $\phi \in {\mathcal H}^{q}_{\rel}(M, E)$. Then $\phi$ on $Y$ has the form of (cf. (\ref{E:2.45}) - (\ref{E:2.47}))
$$
\phi|_{Y} = du \wedge \left( \Gamma^{Y} \nabla^{Y} \Gamma^{Y} \psi_{1} + \psi_{2} \right) \quad \text{with} \quad \psi_{2} \in \Gamma^{Y} {\mathcal K}.
$$
Hence,
$$
{\mathcal P}_{-, {\mathcal L}_{0}}
\left( \begin{array}{clcr} 0 \\ \Gamma^{Y} \nabla^{Y} \Gamma^{Y} \psi_{1} + \psi_{2} \end{array} \right) =
\left( \begin{array}{clcr} 0 \\ 0 \end{array} \right),
$$
which shows that $\phi \in \ker \Delta_{q, {\mathcal P}_{-, {\mathcal L}_{0}}} \cap \Omega^{q}(M, E)$.
Let $\phi \in \ker \Delta_{q, {\mathcal P}_{-, {\mathcal L}_{0}}} \cap \Omega^{q}(M, E)$. Then on $Y$
$$
\phi|_{Y} = \left( \nabla^{Y} \varphi_{1} + \varphi_{2} \right) + du \wedge \left( \Gamma^{Y} \nabla^{Y} \Gamma^{Y} \psi_{2} + \psi_{2} \right), \quad
\phi_{1} \in {\mathcal K}, \quad \psi_{2} \in \Gamma^{Y} {\mathcal K}.
$$
Since $\phi \in \Dom \left( \B_{q, {\mathcal P}_{-, {\mathcal L}_{0}}} \right)$,
$$
{\mathcal P}_{-, {\mathcal L}_{0}}
\left( \begin{array}{clcr}  \nabla^{Y} \varphi_{1} + \varphi_{2} \\ \Gamma^{Y} \nabla^{Y} \Gamma^{Y} \psi_{2} + \psi_{2}
\end{array} \right)   ~ = ~
\left( \begin{array}{clcr}  \nabla^{Y} \varphi_{1} + \varphi_{2} \\ 0 \end{array} \right)
~ = ~ \left( \begin{array}{clcr}  0 \\ 0 \end{array} \right),
$$
which shows that $\phi \in {\mathcal H}^{q}_{\rel}(M, E)$. This completes the proof of the lemma.
\end{proof}

\vspace{0.3 cm}

\subsection{Well-posed Boundary value problem for Dirac type operators}

In this subsection we show that both ${\mathcal P}_{-, {\mathcal L}_{0}}$ and ${\mathcal P}_{+, {\mathcal L}_{1}}$
give well-posed boundary conditions for the odd signature operator $\B$.
We begin with brief description of the well-posed boundary value problem for a Dirac type operator on a
compact oriented manifold with boundary. We refer to \cite{Gr} and \cite{Se} for more details.
Let $X$ be a compact manifold with boundary $Y$ and $F \rightarrow X$ be a Hermitian vector bundle of rank $k$ over $X$.
We suppose that ${\mathcal D} : C^{\infty}(F) \rightarrow C^{\infty}(F)$ is a first order elliptic differential operator on $X$ which
is represented, on a small collar neighborhood of $Y$, by
\begin{equation}\label{E:2.28}
{\mathcal D} = G ( \partial_{u} + A ),
\end{equation}
where $G$ is a unitary bundle automorphism of $F|_{Y}$ and $A$ is a self-adjoint elliptic differential operator of order $1$
acting on $C^{\infty}(F|_{Y})$.
We call ${\mathcal D}$ a Dirac type operator.
For $s > \frac{1}{2}$, we define the Cauchy data space $N_{+}^{s}$ by
$$
N_{+}^{s} = \{ \hspace{0.1 cm} \phi|_{Y} \mid \phi \in H^{s}(M, F), \hspace{0.1 cm} {\mathcal D} \phi = 0 \hspace{0.1 cm} \}
$$
and define the Calder\'on projector ${\mathcal C}^{+}$ by the orthogonal projection from
$H^{s - \frac{1}{2}}(Y, F|_{Y})$ to $N_{+}^{s}$.
We put ${\mathcal C}^{-} = I - {\mathcal C}^{+}$ and recall that
the principal symbols $\sigma_{L}({\mathcal C}^{\pm})$ of ${\mathcal C}^{\pm}$ are bundle maps
$\sigma_{L}({\mathcal C}^{\pm}) : T^{\ast}Y \rightarrow \End (F|_{Y})$.
We put $N_{\pm}(x^{\prime}, \xi^{\prime}) = \Imm \left( \sigma_{L}({\mathcal C}^{\pm})(x^{\prime}, \xi^{\prime}) \right) \subset
F_{x^{\prime}} = {\Bbb C}^{k}$.
The following definition is due to Seeley (\cite{Gr}, \cite{Se}).

\vspace{0.2 cm}

\begin{definition}\label{Definition:2.14}
(Well-posedness). Let $X$ be a compact oriented manifold with boundary $Y$ and
${\mathcal D} : C^{\infty}(F) \rightarrow C^{\infty}(F)$ be a Dirac type operator. Suppose that
$B : C^{\infty}(F|_{Y}) \rightarrow C^{\infty}(F|_{Y})$ is a classical $\Psi$DO of order $0$. Then $B$ is well-posed
for ${\mathcal D}$ when :  \newline
(1) The mapping defined by $B$ in $H^{s}(F|_{Y})$ has closed range for each $s \in {\Bbb R}$.  \newline
(2) For each $(x^{\prime}, \xi^{\prime}) \in T^{\ast}Y$ with $\parallel \xi^{\prime} \parallel = 1$, the principal
symbol $b^{0}(x^{\prime}, \xi^{\prime})$ for $B$ maps $N_{+}(x^{\prime}, \xi^{\prime})$ injectively onto the
range of $b^{0}(x^{\prime}, \xi^{\prime})$ in ${\Bbb C}^{k}$.
\end{definition}

\vspace{0.2 cm}

When $B$ is well-posed for ${\mathcal D}$, the realization ${\mathcal D}_{B}$ is a Fredholm operator
and has a compact resolvent. In particular, its spectrum is discrete and each generalized eigenvalue
has a finite multiplicity.
Let $\Pi_{\geq}$ be an orthogonal projection onto the non-negative eigenspaces of $A$,
where $A$ is the self-adjoint tangential
operator in (\ref{E:2.28}). It is a well-known fact that $\sigma_{L}(\Pi_{\geq})(x^{\prime}, \xi^{\prime})$
is an orthogonal projection onto
the space of positive eigenvectors of $\sigma_{L}(A)(x^{\prime}, \xi^{\prime})$.
If ${\mathcal D}$ is a Dirac type operator, it is also a well-known fact that
${\mathcal C}^{+} - \Pi_{\geq}$ is a classical $\Psi$DO of order $-1$ and hence
$\sigma_{L}({\mathcal C}^{+})(x^{\prime}, \xi^{\prime})$ is an orthogonal projection onto
the space of positive eigenvectors of $\sigma_{L}(A)(x^{\prime}, \xi^{\prime})$.

Now we go back to the odd signature operator. From the Assumption A, we have
$$
A = \left( \begin{array}{clcr}  0 & -1 \\ -1 & 0 \end{array} \right)
\left( \nabla^{Y} + \Gamma^{Y} \nabla^{Y} \Gamma^{Y} \right).
$$
It is well-known (cf. p.47 in \cite{Gi}) that

\begin{eqnarray} \label{E:2.29}
& & \sigma_{L}(\nabla^{Y}), \hspace{0.1 cm} \sigma_{L}(\Gamma^{Y} \nabla^{Y} \Gamma^{Y}) \hspace{0.1 cm} : \hspace{0.1 cm} T^{\ast}Y \rightarrow
\End \left( \wedge^{\bullet} T^{\ast}Y \otimes E|_{Y} \right) \nonumber \\
& & \sigma_{L}(\nabla^{Y}) (x^{\prime}, \xi^{\prime}) (\omega) = i \xi^{\prime} \wedge \omega, \quad
\sigma_{L}(\Gamma^{Y} \nabla^{Y} \Gamma^{Y}) (x^{\prime}, \xi^{\prime}) (\omega) =
- i {\xi^{\prime}} \lrcorner \hspace{0.1 cm} \omega,
\end{eqnarray}
where ${\xi^{\prime}} \lrcorner \hspace{0.1 cm}$ is the interior product with $\xi^{\prime}$.
This leads to
\begin{eqnarray*}
\sigma_{L}(A) (x^{\prime}, \xi^{\prime}) = \left( \begin{array}{clcr}
0 & (- i \xi^{\prime} \wedge ~)  + (i  {\xi^{\prime}} \lrcorner ~) \ \\
 (- i \xi^{\prime} \wedge ~)  + (i {\xi^{\prime}} \lrcorner ~)  &  0
\end{array} \right) : \left( \begin{array}{clcr} \left( \wedge^{\bullet} T^{\ast}Y \otimes E \right)_{x^{\prime}} \\
\left( \wedge^{\bullet} T^{\ast}Y \otimes E \right)_{x^{\prime}} \end{array} \right) \rightarrow
\left( \begin{array}{clcr} \left( \wedge^{\bullet} T^{\ast}Y \otimes E \right)_{x^{\prime}} \\
\left( \wedge^{\bullet} T^{\ast}Y \otimes E \right)_{x^{\prime}} \end{array} \right).
\end{eqnarray*}
Simple computation shows that
\begin{eqnarray*}
\text{the positive eigenspace of} \hspace{0.1 cm} (- i \xi^{\prime} \wedge ~) ~ + ~ (i \hspace{0.1 cm} {\xi^{\prime}} \lrcorner ~) \hspace{0.1 cm}
& = &
\{ \parallel \xi^{\prime} \parallel \omega - i \xi^{\prime} \wedge \omega \mid {\xi^{\prime}} \lrcorner \hspace{0.1 cm} \omega = 0 \},  \\
\text{the negative eigenspace of} \hspace{0.1 cm} (- i \xi^{\prime} \wedge ~) ~ + ~ (i \hspace{0.1 cm} {\xi^{\prime}} \lrcorner ~)
\hspace{0.1 cm} & = &
\{ \parallel \xi^{\prime} \parallel \omega + i \xi^{\prime} \wedge \omega \mid {\xi^{\prime}} \lrcorner \hspace{0.1 cm} \omega = 0 \}.
\end{eqnarray*}
Hence, the positive eigenspace of $\sigma_{L}(A)(x^{\prime}, \xi^{\prime})$, which is $N_{+}(x^{\prime}, \xi^{\prime})$, is spanned by

\begin{equation} \label{E:2.32}
\left( \begin{array}{clcr} \parallel \xi^{\prime} \parallel \omega - i \xi^{\prime} \wedge \omega \\
\parallel \xi^{\prime} \parallel \omega - i \xi^{\prime} \wedge \omega
\end{array} \right)
\quad \text{and} \quad
\left( \begin{array}{clcr} \parallel \xi^{\prime} \parallel \omega + i \xi^{\prime} \wedge \omega \\
- \parallel \xi^{\prime} \parallel \omega - i \xi^{\prime} \wedge \omega
\end{array} \right),
\end{equation}

\noindent
and $\sigma_{L}({\mathcal C}^{+})(x^{\prime}, \xi^{\prime})$ is an orthogonal projection onto the above space.

We next compute the principal symbols of ${\mathcal P}_{-}$ and ${\mathcal P}_{+}$ by using (\ref{E:2.29}).
We first note that

\begin{eqnarray*}
& & \sigma_{L}\left( \nabla^{Y} \Gamma^{Y} \nabla^{Y} \Gamma^{Y} \left( \Gamma^{Y} \nabla^{Y} \Gamma^{Y} \nabla^{Y} + \nabla^{Y} \Gamma^{Y} \nabla^{Y} \Gamma^{Y} +
\pr_{\ker \B_{Y}^{2}} \right)^{-1} \right)(x^{\prime}, \xi^{\prime})
\hspace{0.2 cm} =  \hspace{0.1 cm}  \frac{1}{\parallel \xi^{\prime} \parallel^{2}} (i \xi^{\prime} \wedge ) (- i {\xi^{\prime}} \lrcorner \hspace{0.1 cm} ) \\
& = & \text{the orthogonal projection onto} \quad
\{\xi^{\prime} \wedge \omega \in \left( \wedge^{\bullet} T^{\ast}Y \otimes E \right)_{x^{\prime}} \mid {\xi^{\prime}} \lrcorner \hspace{0.1 cm} \omega = 0 \},
\end{eqnarray*}
\begin{eqnarray*}
& & \sigma_{L} \left(  \Gamma^{Y} \nabla^{Y} \Gamma^{Y} \nabla^{Y} \left( \Gamma^{Y} \nabla^{Y} \Gamma^{Y} \nabla^{Y} + \nabla^{Y} \Gamma^{Y} \nabla^{Y} \Gamma^{Y}
+ \pr_{\ker \B_{Y}^{2}} \right)^{-1} \right) (x^{\prime}, \xi^{\prime})
\hspace{0.2 cm} =  \hspace{0.1 cm} \frac{1}{\parallel \xi^{\prime} \parallel^{2}} (- i {\xi^{\prime}} \lrcorner \hspace{0.1 cm}) (i \xi^{\prime} \wedge) \\
& = & \text{the orthogonal projection onto} \quad
\{ \omega \in \left( \wedge^{\bullet} T^{\ast}Y \otimes E \right)_{x^{\prime}} \mid {\xi^{\prime}} \lrcorner \hspace{0.1 cm} \omega = 0 \},
\end{eqnarray*}

\vspace{0.2 cm}

\noindent
which shows that
$$
\sigma_{L}({\mathcal P}_{-}) (x^{\prime}, \xi^{\prime}), \quad \sigma_{L}({\mathcal P}_{+}) (x^{\prime}, \xi^{\prime}) :
 \left( \begin{array}{clcr} \left( \wedge^{\bullet} T^{\ast}Y \otimes E \right)_{x^{\prime}} \\
\left( \wedge^{\bullet} T^{\ast}Y \otimes E \right)_{x^{\prime}} \end{array} \right) \rightarrow
\left( \begin{array}{clcr} \left( \wedge^{\bullet} T^{\ast}Y \otimes E \right)_{x^{\prime}} \\
\left( \wedge^{\bullet} T^{\ast}Y \otimes E \right)_{x^{\prime}} \end{array} \right)
$$
are orthogonal projections onto
\begin{eqnarray*}
\left( \begin{array}{clcr}
\{\xi^{\prime} \wedge \omega \in \left( \wedge^{\bullet} T^{\ast}Y \otimes E \right)_{x^{\prime}} \mid {\xi^{\prime}} \lrcorner \hspace{0.1 cm} \omega = 0 \} \\
\{\xi^{\prime} \wedge \omega \in \left( \wedge^{\bullet} T^{\ast}Y \otimes E \right)_{x^{\prime}} \mid {\xi^{\prime}} \lrcorner \hspace{0.1 cm} \omega = 0 \}
\end{array} \right),
\qquad
\left(\begin{array}{clcr}
\{ \omega \in \left( \wedge^{\bullet} T^{\ast}Y \otimes E \right)_{x^{\prime}} \mid {\xi^{\prime}} \lrcorner \hspace{0.1 cm} \omega = 0 \}   \\
\{ \omega \in \left( \wedge^{\bullet} T^{\ast}Y \otimes E \right)_{x^{\prime}} \mid {\xi^{\prime}} \lrcorner \hspace{0.1 cm} \omega = 0 \}
\end{array} \right),
\end{eqnarray*}
respectively.

\begin{lemma}\label{Lemma:2.15}
${\mathcal P}_{-}$ and ${\mathcal P}_{+}$ are well-posed boundary conditions for $\B$.
\end{lemma}
\begin{proof}
We are going to check that ${\mathcal P}_{-}$ is well-posed for $\B$. The same argument works for ${\mathcal P}_{+}$.
We note that for each $s \in {\Bbb R}$,
\begin{eqnarray*}
H^{s}\left( (\wedge^{\bullet} T^{\ast}M \otimes E)|_{Y} \right)
& = & \left( \Imm {\mathcal P}_{-} \cap H^{s}\left( (\wedge^{\bullet} T^{\ast}M \otimes E)|_{Y} \right) \right)
\oplus \left( \Imm {\mathcal P}_{+} \cap H^{s}\left( (\wedge^{\bullet} T^{\ast}M \otimes E)|_{Y} \right) \right) \\
& & \oplus \left( \Imm {\mathcal P}_{\left( \Ker \nabla^{Y} \cap \Ker \Gamma^{Y} \nabla^{Y} \Gamma^{Y} \right)} \cap
H^{s}\left( (\wedge^{\bullet} T^{\ast}M \otimes E)|_{Y} \right) \right),
\end{eqnarray*}
which shows that the range of ${\mathcal P}_{-}$ in $H^{s}\left( (\wedge^{\bullet} T^{\ast}M \otimes E)|_{Y} \right)$
is a closed subspace of $H^{s}\left( (\wedge^{\bullet} T^{\ast}M \otimes E)|_{Y} \right)$.
In view of (\ref{E:2.32}) we next note that
\begin{eqnarray*}
\sigma_{L}({\mathcal P}_{-}) ((x^{\prime}, \xi^{\prime}))
\left( \begin{array}{clcr} \parallel \xi^{\prime} \parallel \omega - i \xi^{\prime} \wedge \omega \\
\parallel \xi^{\prime} \parallel \omega - i \xi^{\prime} \wedge \omega
\end{array} \right)
& = & \left( \begin{array}{clcr} - i \xi^{\prime} \wedge \omega \\ - i \xi^{\prime} \wedge \omega
\end{array} \right),  \\
\sigma_{L}({\mathcal P}_{-}) ((x^{\prime}, \xi^{\prime}))
\left( \begin{array}{clcr} \parallel \xi^{\prime} \parallel \omega + i \xi^{\prime} \wedge \omega \\
- \parallel \xi^{\prime} \parallel \omega - i \xi^{\prime} \wedge \omega
\end{array} \right)
& = & \left( \begin{array}{clcr} i \xi^{\prime} \wedge \omega \\ - i \xi^{\prime} \wedge \omega
\end{array} \right).
\end{eqnarray*}
The above equalities show that
$\sigma_{L}({\mathcal P}_{-}) ((x^{\prime}, \xi^{\prime})) \left( N_{+}(x^{\prime}, \xi^{\prime}) \right)$
is spanned by $\left( \begin{array}{clcr} i \xi^{\prime} \wedge \omega \\ 0 \end{array} \right)$ and
$\left( \begin{array}{clcr} 0 \\ i \xi^{\prime} \wedge \omega \end{array} \right)$,
which is same as $\Imm \left( \sigma_{L}({\mathcal P}_{-}) ((x^{\prime}, \xi^{\prime})) \right)$.
Hence, ${\mathcal P}_{-}$ is well-posed for $\B$.
\end{proof}

\vspace{0.3 cm}

\subsection{Agmon angles for the operators $\B_{{\mathcal P}_{-, {\mathcal L}_{0}}}$ and $\B^{2}_{{\mathcal P}_{-, {\mathcal L}_{0}}}$}

In this subsection we prove an analogue of Lemma 4.1 in \cite{CM},
which shows the distribution of generalized eigenvalues of $\B_{{\mathcal P}_{-, {\mathcal L}_{0}}}$ and $\B^{2}_{{\mathcal P}_{-, {\mathcal L}_{0}}}$. From this fact we can choose an Agmon angle arbitrarily close to any given angle $\phi$ for $- \frac{\pi}{2} < \phi < 0$.

Since $\nabla^{Y}$ is a Hermitian connection with respect to ${\frak h}^{Y}$, Corollary 2.3 and (\ref{E:2.1000}) show that
for $\phi$, $\psi \in \Dom \left( \B_{{\mathcal P}_{-, {\mathcal L}_{0}}} \right)$, we have
 we have  $ \hspace{0.2 cm}
\langle \B \phi, \hspace{0.1 cm} \psi \rangle_{M} \hspace{0.1 cm} = \hspace{0.1 cm}
\langle \phi, \hspace{0.1 cm} \B^{\prime} \psi \rangle_{M}$.
We define operators ${\mathcal U}$ and ${\mathcal F}$ by
\begin{equation}  \label{E:2.30}
{\mathcal U} = \frac{1}{2} \left( \B + \B^{\prime} \right), \qquad
{\mathcal F} = \frac{1}{2} \left( \B -  \B^{\prime} \right).
\end{equation}
Then ${\mathcal U}$ is an elliptic $\Psi$DO of order $1$ having the same principal symbol as $\B$ and
${\mathcal F}$ is a $\Psi$DO of order $0$. In particular, ${\mathcal F}$ is a bounded operator and
${\mathcal U}_{{\mathcal P}_{-, {\mathcal L}_{0}}}$ is a self-adjoint operator with
$\B_{{\mathcal P}_{-, {\mathcal L}_{0}}} = {\mathcal U}_{{\mathcal P}_{-, {\mathcal L}_{0}}} + {\mathcal F}$,
which leads to the following result.

\vspace{0.2 cm}

\begin{theorem}  \label{Theorem:2.22}
In the decomposition $\B = {\mathcal U} + {\mathcal F}$, we put
${\mathcal N}_{0} = \parallel {\mathcal F} \parallel $.
Then : \newline
(1) If $\lambda$ is an eigenvalue of $\B_{{\mathcal P}_{-, {\mathcal L}_{0}}}$, then
$\vert \Imm \lambda \vert \leq {\mathcal N}_{0}$.  \newline
(2) If $\mu = \lambda^{2}$ is an eigenvalue of $\B^{2}_{{\mathcal P}_{-, {\mathcal L}_{0}}}$, then
$\re \mu  \geq \frac{1}{4 {\mathcal N}^{2}_{0}} \left( \Imm \mu \right)^{2} - {\mathcal N}^{2}_{0}$.
\end{theorem}
\begin{proof}
Suppose that $\B \psi = \lambda \psi$ with $\psi \in \Dom \left( \B_{{\mathcal P}_{-, {\mathcal L}_{0}}} \right)$
and $\parallel \psi \parallel_{M} = 1$.
Then from $\B = {\mathcal U} + {\mathcal F}$, we have
$$
\lambda \hspace{0.1 cm} \langle \psi, \hspace{0.1 cm} \psi \rangle_{M} \hspace{0.1 cm} = \hspace{0.1 cm}
\langle \B \psi, \hspace{0.1 cm} \psi \rangle_{M}
\hspace{0.1 cm} = \hspace{0.1 cm} \langle {\mathcal U} \psi, \hspace{0.1 cm} \psi \rangle_{M} +
\langle {\mathcal F} \psi, \hspace{0.1 cm} \psi \rangle_{M}.
$$
Since ${\mathcal U}$ is self-adjoint, we have $(\Imm \lambda)  \langle \psi, \hspace{0.1 cm} \psi \rangle_{M} =
\Imm (\langle {\mathcal F} \psi, \hspace{0.1 cm} \psi \rangle_{M})$, from which
the assertion (1) follows.
Putting $\lambda = x + i y$ with $x, y \in {\Bbb R}$, then $\lambda^{2} = (x^{2} - y^{2}) + 2xyi$.
Hence,
$$
\left( \Imm \mu \right)^{2} = \left( \Imm \lambda^{2}\right)^{2} \hspace{0.1 cm} = \hspace{0.1 cm}
4 x^{2} y^{2} \hspace{0.1 cm} = \hspace{0.1 cm} 4 y^{2} (x^{2} - y^{2} + y^{2})
\hspace{0.1 cm} = \hspace{0.1 cm} 4 y^{2} \left( \re(\lambda^{2}) + y^{2} \right)
\hspace{0.1 cm} \leq  \hspace{0.1 cm} 4 {\mathcal N}_{0}^{2} \left( \re \mu  + {\mathcal N}_{0}^{2} \right),
$$
from which the second assertion follows.
\end{proof}

\vspace{0.1 cm}

\begin{definition} \label{Definition:2.16}
The angle $\theta$ is called an Agmon angle for an elliptic operator ${\frak D}$ if : \newline
(1) Spec $\left( \sigma_{L}({\frak D}) (x, \xi) \right) \cap R_{\theta} = \emptyset$ for all $x \in M$ and $\xi \in T_{x}^{\ast}M - \{ 0 \}$,
where $R_{\theta} = \{ \rho e^{i \theta} \mid 0 < \rho < \infty \}$. \newline
(2) Spec $\left( {\frak D} \right) \cap L_{[\theta - \epsilon, \theta + \epsilon]} = \emptyset$ for some $\epsilon > 0$, where
$L_{[\theta - \epsilon, \theta + \epsilon]} =
\{ \rho e^{i \phi} \mid 0 < \rho < \infty, \hspace{0.1 cm} \theta - \epsilon \leq \phi \leq \theta + \epsilon \}$.
\end{definition}

\vspace{0.1 cm}

Theorem \ref{Theorem:2.22} shows that if $\phi$ is an angle with $- \frac{\pi}{2} < \phi < 0$,
both operators $\B_{{\mathcal P}_{-, {\mathcal L}_{0}}}$ and $\B^{2}_{{\mathcal P}_{-, {\mathcal L}_{0}}}$
have only finitely many eigenvalues in the sectors $L_{[- \frac{\pi}{2} + \phi, \phi]}$ and $L_{[-2\phi, 2\pi + 2\phi]}$.
Moreover, for each $(x, \xi)$, $\sigma_{L}(\B)(x, \xi)$ is a symmetric matrix and has real eigenvalues.
This shows that we can choose an angle $\theta$ arbitrarily close to $\phi$ so that
$\theta$ is an Agmon angle for
$\B_{{\mathcal P}_{-, {\mathcal L}_{0}}}$ and $2 \theta$ for
$\B^{2}_{{\mathcal P}_{-, {\mathcal L}_{0}}}$.

\vspace{0.3 cm}

\section{The heat kernel asymptotics of
$\Tr(e^{-t \B^{2}_{q, {\mathcal P}_{-, {\mathcal L}_{0}}}})$ and
$\Tr(\B_{\even} e^{-t \B^{2}_{\even, {\mathcal P}_{-, {\mathcal L}_{0}}}})$}

\vspace{0.2 cm}

In this section we discuss the small time asymptotic expansions of the traces of heat kernels of $e^{-t \B^{2}_{q, {\mathcal P}_{-, {\mathcal L}_{0}}}}$ and
$\B_{\even} e^{-t \B^{2}_{\even, {\mathcal P}_{-, {\mathcal L}_{0}}}}$ to compute the pole structures of the zeta and eta functions associated to $\B^{2}_{q, {\mathcal P}_{-, {\mathcal L}_{0}}}$ and $\B_{\even, {\mathcal P}_{-, {\mathcal L}_{0}}}$,
where $\B^{2}_{q, {\mathcal P}_{-, {\mathcal L}_{0}}}$ and $\B_{\even, {\mathcal P}_{-, {\mathcal L}_{0}}}$ are
the restrictions of the operators $\B^{2}_{{\mathcal P}_{-, {\mathcal L}_{0}}}$ and $\B_{{\mathcal P}_{-, {\mathcal L}_{0}}}$
to $\Omega^{q} (M, E)$ and  $\Omega^{\even} (M, E)$, respectively.
For this purpose we adopt the method of \cite{APS}
to construct parametrices for the heat kernels of $e^{-t \B^{2}_{q, {\mathcal P}_{-, {\mathcal L}_{0}}}}$ and
$\B_{\even} e^{-t \B^{2}_{\even, {\mathcal P}_{-, {\mathcal L}_{0}}}}$ by combining the heat kernels on the interior part and the heat kernels on the
collar of the boundary part. We begin with the computation of
the heat kernels on the half-infinite cylinder $Z := [0, \infty) \times Y$.

\vspace{0.2 cm}

\subsection{The heat kernels on the half-infinite cylinder}

We define the odd signature operator $\B_{\cyl}$ and its square $\B_{\cyl}^{2}$ on $Z$ by (\ref{E:2.8}) and (\ref{E:2.10}).
We denote by $\B_{\cyl, \even}$, $\B^{2}_{\cyl, q}$, $\B^{2}_{\cyl, \even}$ the restrictions $\B_{\cyl}$ and $\B^{2}_{\cyl}$
to the space of even forms or $q$-forms and
denote by $\E^{\cyl}_{q}(t, (u, y), (v, y^{\prime}))$ and $\E^{\cyl}_{\even}(t, (u, y), (v, y^{\prime}))$
the kernels of
$e^{-t \B^{2}_{\cyl, q, {\mathcal P}_{-, {\mathcal L}_{0}}}}$ and $e^{-t \B^{2}_{\cyl, \even, {\mathcal P}_{-, {\mathcal L}_{0}}}}$ on $Z$.
The boundary condition that we impose is equal to the Dirichlet condition on
$\hspace{0.1 cm} \Imm {\mathcal P}_{-, {\mathcal L}_{0}} \hspace{0.1 cm}$ and the Neumann condition on
$\hspace{0.1 cm} \Imm {\mathcal P}_{+, {\mathcal L}_{1}} \hspace{0.1 cm}$.

We denote ${\mathcal K}^{q} := {\mathcal K} \cap \Omega^{q}(Y, E|_{Y})$ and
$(\Gamma^{Y} {\mathcal K})^{q} := \Gamma^{Y} {\mathcal K} \cap \Omega^{q}(Y, E|_{Y})$
and note (cf. (\ref{E:2.14}), (\ref{E:2.15})) that for each $0 \leq q \leq m-1$,

\begin{eqnarray} \label{E:3.0}
 \Imm {\mathcal P}_{-, {\mathcal L}_{0}} \cap \Omega^{q}(M, E)|_{Y} & = &  \left( \begin{array}{clcr}
\Omega^{q}_{-}(Y, E|_{Y}) \oplus {\mathcal K}^{q} \\ \Omega^{q}_{-}(Y, E|_{Y}) \oplus {\mathcal K}^{q} \end{array} \right) , \nonumber  \\
\Imm {\mathcal P}_{+, {\mathcal L}_{1}} \cap \Omega^{q}(M, E)|_{Y} & = & \left( \begin{array}{clcr}
\Omega^{q}_{+}(Y, E|_{Y}) \oplus (\Gamma^{Y} {\mathcal K})^{q} \\ \Omega^{q}_{+}(Y, E|_{Y}) \oplus (\Gamma^{Y} {\mathcal K})^{q}
\end{array} \right).
\end{eqnarray}

\vspace{0.2 cm}

\noindent
Each component in (\ref{E:3.0}) is decomposed into the sums of the eigenspaces of $\B_{Y, q}^{2, \mp}$ and $\B_{Y, q-1}^{2, \mp}$,
where $\B_{Y, q}^{2, \mp}$ are the restriction of $\B_{Y, q}^{2}$ to $\Omega^{q}_{-}(Y, E|_{Y}) \oplus {\mathcal K}^{q}$ and
$\Omega^{q}_{+}(Y, E|_{Y}) \oplus (\Gamma^{Y} {\mathcal K})^{q}$, respectively.
For example,

\begin{equation} \label{E:3.10}
\Dom \left( \B_{Y, q}^{2, -} \right) \hspace{0.1 cm} = \hspace{0.1 cm}
\Omega^{q}_{-}(Y, E|_{Y}) \oplus {\mathcal K}^{q} \hspace{0.1 cm} = \hspace{0.1 cm}
\oplus_{k=0}^{\infty} \Lambda_{k}^{q, -}(Y, E|_{Y}),
\end{equation}

\noindent
where $\Lambda_{k}^{q, -}(Y, E|_{Y})$ is the eigenspace of $\B_{Y, q}^{2, -}$ with eigenvalue $\lambda^{-}_{q, k}$ and eigenform $\phi^{-}_{q, k}$.
The corresponding heat kernel is given by
\begin{eqnarray}  \label{E:3.13}
{\mathcal E}^{\cyl, -}_{q}(t, (u, y), (v, y^{\prime}))
& = & \sum_{\lambda^{-}_{q, k}} \frac{e^{-t \lambda^{-}_{q, k}}}{\sqrt{4 \pi t}}
\left\{ e^{- \frac{(u - v)^{2}}{4t}} - e^{- \frac{(u + v)^{2}}{4t}} \right\} \phi^{-}_{q, k} \otimes (\phi^{-}_{q, k})^{\ast}  \nonumber .
\end{eqnarray}

\noindent
We can construct the heat kernel
${\mathcal E}^{\cyl, -}_{q-1}(t, (u, y), (v, y^{\prime}))$ for $\B^{2, -}_{Y, q-1}$ in the same way.

\vspace{0.3 cm}

We next consider the case of Neumann condition. We recall that

\begin{equation}  \label{E:3.14}
\Dom \left( \B_{Y, q}^{2, +} \right) \hspace{0.1 cm} = \hspace{0.1 cm}
\Omega^{q}_{+}(Y, E|_{Y}) \oplus (\Gamma^{Y} {\mathcal K})^{q} \hspace{0.1 cm} = \hspace{0.1 cm}
\oplus_{k=0}^{\infty} \Lambda_{k}^{q, +}(Y, E|_{Y}),
\end{equation}

\noindent
where $\Lambda_{k}^{q, +}(Y, E|_{Y})$ is the eigenspace of $\B_{Y, q}^{2, +}$ with eigenvalue $\lambda^{+}_{q, k}$  and eigenform $\phi^{+}_{q, k}$.
The corresponding heat kernel $\hspace{0.1 cm} {\mathcal E}^{\cyl, +}_{q, k}(t, (u, y), (v, y^{\prime})) \hspace{0.1 cm}$
 is given by

\begin{eqnarray} \label{E:3.15}
{\mathcal E}^{\cyl, +}_{q}(t, (u, y), (v, y^{\prime}))
& = &  \sum_{\lambda^{+}_{q, k}} \frac{e^{-t \lambda^{+}_{q, k}}}{\sqrt{4 \pi t}}
\left\{ e^{- \frac{(u - v)^{2}}{4t}} + e^{- \frac{(u + v)^{2}}{4t}} \right\} \phi^{+}_{q, k} \otimes (\phi^{+}_{q, k})^{\ast}   .
\end{eqnarray}

\vspace{0.2 cm}
\noindent
We can construct the heat kernel
${\mathcal E}^{\cyl, +}_{q-1, k}(t, (u, y), (v, y^{\prime}))$ for $\B^{2, +}_{Y, q-1}$ in the same way.
Finally, the heat kernel ${\mathcal E}^{\cyl}_{q}(t, (u, y), (v, y^{\prime}))$ is given by

\begin{eqnarray}  \label{E:3.16}
{\mathcal E}^{\cyl}_{q}(t, (u, y), (v, y^{\prime})) & = & \left( {\mathcal E}^{\cyl, -}_{q}(t, (u, y), (v, y^{\prime}))
+ {\mathcal E}^{\cyl, +}_{q}(t, (u, y), (v, y^{\prime})) \right)  \left(\begin{array}{clcr} 1 & 0\\ 0 & 0 \end{array} \right)  \nonumber \\
& + & \left( {\mathcal E}^{\cyl, -}_{q-1}(t, (u, y), (v, y^{\prime}))  + {\mathcal E}^{\cyl, +}_{q-1}(t, (u, y), (v, y^{\prime})) \right)
\left(\begin{array}{clcr} 0 & 0 \\ 0 &  1 \end{array} \right).
\end{eqnarray}

\vspace{0.2 cm}

\subsection{Construction of parametrices for the heat kernels of $e^{-t \B^{2}_{q, {\mathcal P}_{-, {\mathcal L}_{0}}}}$ and
$\B_{\even} e^{-t \B^{2}_{\even, {\mathcal P}_{-, {\mathcal L}_{0}}}}$}

Let ${\widetilde M}$ be the closed double of $M$, {\it i.e.}, ${\widetilde M} = M \cup_{Y} M$.
We can extend $\B$ and $E$ on $M$ to ${\widetilde M}$, which
we denote by ${\widetilde {\B}}$ and ${\widetilde E}$.
We also denote by ${\widetilde {\B_{q}}}$, ${\widetilde {\B}}_{\even}$ the operator ${\widetilde {\B}}$ acting on the space of $q$-forms and even forms
and denote by
${\widetilde \E_{q}}(t, x, x^{\prime})$, ${\widetilde \E_{\even}}(t, x, x^{\prime})$
the kernels of
$e^{-t {\widetilde \B_{q}}^{2}}$, $e^{-t {\widetilde \B_{\even}}^{2}}$, respectively.
It is a well-known fact (cf. p.225 in \cite{BW}) that
\begin{equation} \label{E:3.444}
| {\widetilde \E_{q}}(t, x, x^{\prime}) | \leq c_{1} t^{- \frac{m}{2}} e^{- c_{2} \frac{d(x, x^{\prime})^{2}}{t}} \qquad \text{and} \qquad
| D_{x} {\widetilde \E_{q}}(t, x, x^{\prime}) | \leq c_{3} t^{- \frac{m+1}{2}} e^{- c_{4} \frac{d(x, x^{\prime})^{2}}{t}},
\end{equation}
where $c_{i}$'s are positive constants and $D$ is a differential operator of order $1$.

Recall that $N = [0, \epsilon_{0}) \times Y$ is a collar neighborhood of $Y$.
Let $\rho(a, b)$ be a smooth increasing function of real variable such that
\[ \rho(a, b) (u) = \left\{ \begin{array}{ll} 0 & \mbox{for $u \leq a$} \\
1 & \mbox{for $u \geq b$} \hspace{0.1 cm}.
\end{array} \right. \]
We put

\begin{eqnarray*}
\phi_{1} := 1 - \rho(\frac{5}{7} \epsilon_{0}, \frac{6}{7} \epsilon_{0}), \quad
\psi_{1} := 1 - \rho(\frac{3}{7} \epsilon_{0}, \frac{4}{7} \epsilon_{0}), \quad
 \phi_{2} := \rho(\frac{1}{7} \epsilon_{0}, \frac{2}{7} \epsilon_{0}), \quad  \psi_{2} := \rho(\frac{3}{7} \epsilon_{0}, \frac{4}{7} \epsilon_{0}),
\end{eqnarray*}

\noindent
and
\begin{equation} \label{E:3.3}
{\mathcal Q}_{q}(t, (u, y), (v, y^{\prime})) = \phi_{1}(u) \E^{\cyl}_{q}(t, (u, y), (v, y^{\prime})) \psi_{1}(v) +
\phi_{2}(u) {\widetilde \E}_{q}(t, (u, y), (v, y^{\prime})) \psi_{2}(v).
\end{equation}

\noindent
Then ${\mathcal Q}_{q}(t, (u, y), (v, y^{\prime}))$ is a parametrix for the kernel of $e^{-t \B^{2}_{q, {\mathcal P}_{-, {\mathcal L}_{0}}}}$.
Let $\E_{q} (t, (u, y), (v, y^{\prime}))$ be the kernel of the heat operator
 $e^{-t \B^{2}_{q, {\mathcal P}_{-, {\mathcal L}_{0}}}}$ on $M$.
Then standard computation using (\ref{E:3.13}), (\ref{E:3.15}) and (\ref{E:3.444}) shows (cf. \cite{BW}, \cite{DW}, \cite{KW}) that for $0 < t \leq 1$,

\begin{equation} \label{E:3.4}
| \E_{q} (t, (u, y), (u, y)) - {\mathcal Q}_{q} (t, (u, y), (u, y)) | \leq c_{1} e^{- \frac{c_{2}}{t}}
\end{equation}

\noindent
for some positive constants $c_{1}$ and $c_{2}$.
Similarly, we put

\begin{eqnarray} \label{E:3.5}
& & \hspace{0.5 cm} {\mathcal R}_{\even}(t, (u, y), (v, y^{\prime}))  \nonumber \\
& = &   \phi_{1}(u) \B_{\cyl, \even} \E^{\cyl}_{\even}(t, (u, y), (v, y^{\prime})) \psi_{1}(v) +
\phi_{2}(u) {\widetilde \B}_{\even} {\widetilde \E}_{\even}(t, (u, y), (v, y^{\prime})) \psi_{2}(v).
\end{eqnarray}

\noindent
Then ${\mathcal R}_{\even}(t, (u, y), (v, y^{\prime}))$ is a parametrix for
$\B_{\even} \E_{\even}(t, (u, y), (v, y^{\prime}))$, the kernel of
$\B_{\even} e^{-t \B^{2}_{\even, {\mathcal P}_{-, {\mathcal L}_{0}}}}$ on $M$ and
the standard computation shows that for $0 < t \leq 1$,

\begin{equation} \label{E:3.6}
| \B_{\even} \E_{\even}(t, (u, y), (u, y)) - {\mathcal R}_{\even}(t, (u, y), (u, y)) | \leq c_{3} e^{-\frac{c_{4}}{t}}
\end{equation}

\noindent
for some positive constants $c_{3}$ and $c_{4}$.
It is also a well-known fact that for $t \rightarrow 0^{+}$,

\begin{eqnarray}  \label{E:3.7}
\int_{M} \Tr {\widetilde \E}_{q}(t, (u, y), (u, y) ) \psi_{2}(u) \hspace{0.1 cm} d vol(M)
& \sim & \sum_{j=0}^{\infty} a_{j} t^{- \frac{m}{2}+j}, \nonumber \\
\int_{M} \Tr \left( {\widetilde \B_{\even}} {\widetilde \E}_{\even}(t, (u, y), (v, y^{\prime}) )
\psi_{2}(v) \right)|_{(v, y^{\prime}) = (u, y)} \hspace{0.1 cm}  d vol(M) & \sim &
\sum_{j=0}^{\infty} b_{j} t^{- \frac{m + 1 - j}{2}}.
\end{eqnarray}

\noindent
The pole structure of the eta function associated to
$\B_{\even, {\mathcal P}_{-, {\mathcal L}_{0}}}$ at $s=0$ is closely related to $b_{m}$ in (\ref{E:3.7}).

\vspace{0.2 cm}

\begin{lemma}  \label{Lemma:3.20}
The coefficient $b_{m}$ appearing in (\ref{E:3.7}) is equal to zero.
\end{lemma}

\begin{proof}
Let ${\widetilde \B}^{\trivial}_{\even} : \Omega^{\even} ({\widetilde M}, {\Bbb C}) \rightarrow \Omega^{\even}({\widetilde M}, {\Bbb C})$
be the odd signature operator on ${\widetilde M}$
obtained by the trivial connection on the trivial line bundle ${\widetilde M} \times {\Bbb C} \rightarrow {\widetilde M}$.
We consider the asymptotic expansion of the trace of the heat kernel

\begin{eqnarray*}
\Tr \left(  {\widetilde \B}^{\trivial}_{\even} e^{-t ({\widetilde \B}^{\trivial}_{\even})^{2}} \right) & \sim &
\sum_{j=1}^{\infty} {\widetilde b}_{j} t^{- \frac{m + 1 - j}{2}},
\end{eqnarray*}

\noindent
where ${\widetilde b}_{j} = \int_{{\widetilde M}} {\widetilde b}_{j}(x) d vol({\widetilde M})$ for some local invariant ${\widetilde b}_{j}(x)$.
Since ${\widetilde \B}^{\trivial}_{\even}$ is a compatible Dirac operator, the local invariant ${\widetilde b}_{m}(x) = 0$,
which was shown in Theorem 3.2 in \cite{BG}.
The coefficient $b_{m}$ in (\ref{E:3.7}) is given by $b_{m} = \int_{M} b_{m}(x) \psi_{2}(x) d vol(M)$ for some local invariant $b_{m}(x)$
on ${\widetilde M}$. Let $\Id_{n}$ be the $n \times n$ identity matrix, where $n = \text{rank} (E)$.
Since ${\widetilde \B}^{\trivial}_{\even} \Id_{n}$ and ${\widetilde \B}_{\even}$ are locally same operators, $n {\widetilde b}_{m}(x) = b_{m}(x) = 0$,
from which the result follows.
\end{proof}

\subsection{Asymptotics of $\Tr(e^{-t \B^{2}_{q, {\mathcal P}_{-, {\mathcal L}_{0}}}})$ and
$\Tr(\B_{\even} e^{-t \B^{2}_{\even, {\mathcal P}_{-, {\mathcal L}_{0}}}})$ for $t \rightarrow 0^{+}$}

From (\ref{E:3.10}) - (\ref{E:3.16}), we have

\begin{eqnarray}  \label{E:3.17}
 & &\int_{0}^{\infty} \int_{Y} \Tr \hspace{0.1 cm} \E^{\cyl}_{q}(t, (u, y), (u, y)) \psi_{1}(u) \hspace{0.1 cm} dvol(Y) du   \\
 & = & \sum_{\lambda^{-}_{q, k}} \frac{e^{-t \lambda^{-}_{q, k}}}{\sqrt{4 \pi t}} \hspace{0.1 cm}
\int_{0}^{\infty} ( 1 - e^{- \frac{u^{2}}{t}} ) \psi_{1}(u) \hspace{0.1 cm} du \hspace{0.1 cm}
+ \hspace{0.1 cm} \sum_{\lambda^{+}_{q, k}} \frac{e^{-t \lambda^{+}_{q, k}}}{\sqrt{4 \pi t}}
\int_{0}^{\infty} ( 1 + e^{- \frac{u^{2}}{t}} ) \psi_{1}(u) \hspace{0.1 cm} du   \nonumber \\
& + & \sum_{\lambda^{-}_{q-1, k}} \frac{ e^{-t \lambda^{-}_{q-1, k}}}{\sqrt{4 \pi t}} \hspace{0.1 cm}
\int_{0}^{\infty} ( 1 - e^{- \frac{u^{2}}{t}} ) \psi_{1}(u) \hspace{0.1 cm} du \hspace{0.1 cm}
+ \hspace{0.1 cm} \sum_{\lambda^{+}_{q-1, k}} \frac{e^{-t \lambda^{+}_{q-1, k}}}{\sqrt{4 \pi t}} \hspace{0.1 cm}
\int_{0}^{\infty} ( 1 + e^{- \frac{u^{2}}{t}} ) \psi_{1}(u) \hspace{0.1 cm} du. \nonumber
\end{eqnarray}

\noindent
It is a well-known fact that for $t \rightarrow 0^{+}$ the following sums
$$
\sum_{\lambda^{-}_{q, k}} e^{-t \lambda^{-}_{q, k}}, \qquad \sum_{\lambda^{+}_{q, k}}  e^{-t \lambda^{+}_{q, k}}, \qquad
\sum_{\lambda^{-}_{q-1, k}} e^{-t \lambda^{-}_{q-1, k}}, \qquad \sum_{\lambda^{+}_{q-1, k}}  e^{-t \lambda^{+}_{q-1, k}},
$$
have asymptotic expansions of the type

\begin{equation} \label{E:3.18}
\sum_{j=0}^{\infty} c_{j} t^{- \frac{m-1}{2} + j}.
\end{equation}

\noindent
The equalities (\ref{E:3.4}), (\ref{E:3.7}), (\ref{E:3.17}) and (\ref{E:3.18}) with the fact
$\hspace{0.1 cm} \int_{0}^{\infty} e^{- \frac{u^{2}}{t}} \psi_{1}(u) du = \frac{\sqrt{\pi t}}{2} + O(e^{- \frac{c}{t}}) \hspace{0.1 cm}$
show that for $t \rightarrow 0^{+}$, $\hspace{0.1 cm} \Tr \left( e^{-t \B^{2}_{q, {\mathcal P}_{-, {\mathcal L}_{0}}}} \right)$ has an asymptotic expansion
of the following type

\begin{eqnarray}  \label{E:3.20}
\Tr \left( e^{-t \B^{2}_{q, {\mathcal P}_{-, {\mathcal L}_{0}}}} \right) & \sim &
\sum_{j=0}^{\infty} {\widetilde c}_{j} t^{- \frac{m-j}{2}}.
\end{eqnarray}

\vspace{0.3 true cm}

We next discuss the asymptotic expansion, for $t \rightarrow 0^{+}$, of
$$
\Tr \left( \phi_{1}(u) \B_{\cyl, \even} \E^{\cyl}_{\even}(t, (u, y), (v, y^{\prime})) \psi_{1}(v) \right)|_{u=v, y=y^{\prime}}.
$$


\noindent
From (\ref{E:2.8}) and (\ref{E:3.16}), we have
\begin{equation}  \label{E:3.21}
\Tr \left\{ \phi_{1}(u) \left( - i \beta \right) \left( \begin{array}{clcr}  0 & -1
\\ -1 & 0 \end{array} \right) \left( \Gamma^{Y} \nabla^{Y}  + \nabla^{Y} \Gamma^{Y} \right)
\left( \E^{\cyl}(t, (u, y), (v, y^{\prime})) \psi_{1}(v) \right)|_{u=v, y=y^{\prime}} \right\} = 0.
\end{equation}

\noindent
Since $\Gamma^{Y}$ maps $(\Imm {\mathcal P}_{-, {\mathcal L}_{0}})$ onto $(\Imm {\mathcal P}_{+, {\mathcal L}_{1}})$
and vice versa, we have

\begin{equation}  \label{E:3.22}
\Tr \left\{ \phi_{1}(u) \left( - i \beta \right) \Gamma^{Y} \left( \begin{array}{clcr} 1 & 0 \\ 0 & 1 \end{array} \right) \nabla_{\partial_{u}}
\left( \E^{\cyl}(t, (u, y), (v, y^{\prime})) \psi_{1}(v) \right) \right\}_{u=v, y=y^{\prime}} = 0.
\end{equation}

\noindent
Hence, (\ref{E:3.21}) and (\ref{E:3.22}) yield

\begin{equation}  \label{E:3.23}
\Tr \left( \phi_{1}(u) \B_{\cyl, \even} \E^{\cyl}_{\even}(t, (u, y), (v, y^{\prime})) \psi_{1}(v) \right)|_{u=v, y=y^{\prime}} = 0,
\end{equation}

\noindent
which together with (\ref{E:3.7}) and Lemma \ref{Lemma:3.20} shows that

\begin{equation} \label{E:3.24}
\Tr \left( \B_{\even} \E_{\even}(t, (u, y), (u, y)) \right) \sim \sum_{j=0}^{\infty} b_{j} t^{- \frac{m + 1 - j}{2}} \quad
\text{with} \quad b_{m} = 0.
\end{equation}

\vspace{0.2 cm}

\subsection{The regularities of the zeta and eta functions at $s=0$}

For any angle $\phi$ with $- \frac{\pi}{2} < \phi < 0$ Theorem \ref{Theorem:2.22} shows that we can choose an angle $\theta$
arbitrarily close to $\phi$ so that $\theta$ is an Agmon angle for $\B_{\even, {\mathcal P}_{-, {\mathcal L}_{0}}}$ and
$2 \theta$ is an Agmon angle for $\B^{2}_{\even, {\mathcal P}_{-, {\mathcal L}_{0}}}$. In this subsection we fix an angle $\theta$
satisfying this property.
Moreover,
Theorem \ref{Theorem:2.22} shows that all the generalized eigenvalues of $\B^{2}_{q, {\mathcal P}_{-, {\mathcal L}_{0}}}$
except only finitely many ones have positive real parts.
Let $\hspace{0.1 cm} \tau_{1}, \cdots, \tau_{l} \hspace{0.1 cm}$ and $\hspace{0.1 cm} \lambda_{1}, \lambda_{2}, \cdots \hspace{0.1 cm}$ be
generalized non-zero eigenvalues of $\B^{2}_{q, {\mathcal P}_{-, {\mathcal L}_{0}}}$ counted with their multiplicities,
where $\hspace{0.1 cm} \re \tau_{j} \leq 0 \hspace{0.1 cm}$ and
$\hspace{0.1 cm} \re \lambda_{j} > 0$. We define the zeta function $\zeta_{\B^{2}_{q, {\mathcal P}_{-, {\mathcal L}_{0}}}} (s)$ by

\begin{equation} \label{E:3.1}
\zeta_{\B^{2}_{q, {\mathcal P}_{-, {\mathcal L}_{0}}}} (s) \hspace{0.1 cm} = \hspace{0.1 cm}
\sum_{j=1}^{l} \tau_{j}^{-s} \hspace{0.1 cm} + \hspace{0.1 cm} \sum_{j=1}^{\infty} \lambda_{j}^{-s}.
\end{equation}

\noindent
Similarly, let $\hspace{0.1 cm} i r_{1}, \cdots, i r_{k}$, $\kappa_{1}, \cdots \kappa_{l_{1}},
\kappa_{l_{1}+1}, \cdots, \kappa_{l_{1}+l_{2}} \hspace{0.1 cm}$ and
$\hspace{0.1 cm} \mu_{1}, \mu_{2}, \cdots \hspace{0.1 cm}$ be generalized non-zero eigenvalues of
$\hspace{0.1 cm}\B_{\even, {\mathcal P}_{-, {\mathcal L}_{0}}}$  counted with their multiplicities, where
$\hspace{0.1 cm} r_{j} \in {\Bbb R}$, $\hspace{0.2 cm}\kappa_{j} \in {\Bbb C} - i {\Bbb R} \hspace{0.1 cm}$ with
$\hspace{0.1 cm} \re \kappa_{j}^{2} \leq 0 \hspace{0.1 cm}$ and $\hspace{0.1 cm} \re \kappa_{j} > 0 \hspace{0.1 cm}$ for $1 \leq j \leq l_{1}$,
$\re \kappa_{j} < 0 \hspace{0.1 cm}$ for $l_{1}+1 \leq j \leq l_{1} + l_{2}$,
and $\mu_{j} \in {\Bbb C} - i {\Bbb R} \hspace{0.1 cm}$ with $\hspace{0.1 cm} \re \mu_{j}^{2} > 0$.
We define the eta function $\eta_{\B_{\even, {\mathcal P}_{-, {\mathcal L}_{0}}}} (s)$ by

\begin{equation} \label{3.2}
\eta_{\B_{\even, {\mathcal P}_{-, {\mathcal L}_{0}}}} (s) \hspace{0.1 cm} = \hspace{0.1 cm}
\sum_{j=1}^{l_{1}} \kappa_{j}^{-s} - \sum_{j=l_{1}+1}^{l_{1}+l_{2}} (-\kappa_{j})^{-s}
+ \sum_{\re \mu_{j} > 0} \mu_{j}^{-s} - \sum_{\re \mu_{j}<0} (- \mu_{j})^{-s}.
\end{equation}

\noindent
The asymptotic expansions (\ref{E:3.20}) and (\ref{E:3.24})  imply that for some constants $C_{1}$, $C_{2}$ and $\delta > 0$,
\begin{eqnarray*}
\mid \Tr e^{- t \B^{2}_{q, {\mathcal P}_{-, {\mathcal L}_{0}}}} - \sum_{j=1}^{l} e^{- \tau_{j} t} \mid & \leq &
C_{1} \hspace{0.1 cm}  t^{- \frac{m}{2}} e^{- \delta t},  \\
\mid \Tr \left( \B_{\even} e^{- t\B^{2}_{\even, {\mathcal P}_{-, {\mathcal L}_{0}}}} \right) - i \sum_{j=1}^{k} r_{j} e^{r_{j}^{2}t} -
\sum_{j=1}^{l_{1}+l_{2}} \kappa_{j} e^{-t \kappa_{j}^{2}} \mid  & \leq & C_{2} \hspace{0.1 cm}  t^{- \frac{m+1}{2}} e^{- \delta t},
\end{eqnarray*}
which shows (cf. Lemma 2.3 and Lemma 2.4 in \cite{GS}) that
\begin{eqnarray} \label{E:3.30}
H_{1}(s) & = & \frac{1}{\Gamma(s)} \int_{0}^{\infty} t^{s-1} \left( \Tr e^{- t \B^{2}_{q, {\mathcal P}_{-, {\mathcal L}_{0}}}} -
\sum_{j=1}^{l} e^{- \tau_{j} t} \right) dt,  \nonumber \\
H_{2}(s) & = & \frac{1}{\Gamma(\frac{s+1}{2})} \int_{0}^{\infty} t^{\frac{s-1}{2}} \left(
\Tr \left( \B_{\even} e^{- t\B^{2}_{\even, {\mathcal P}_{-, {\mathcal L}_{0}}}} \right) - i \sum_{j=1}^{k} r_{j} e^{r_{j}^{2}t} -
\sum_{j=1}^{l_{1}+l_{2}} \kappa_{j} e^{-t \kappa_{j}^{2}} \right) dt
\end{eqnarray}
are holomorphic functions for $\re s \gg 0$. We note that
\begin{eqnarray} \label{E:3.31}
\zeta_{\B^{2}_{q, {\mathcal P}_{-, {\mathcal L}_{0}}}} (s) & = & \sum_{j=1}^{l} \tau_{j}^{-s} + H_{1}(s), \nonumber  \\
\eta_{\B_{\even, {\mathcal P}_{-, {\mathcal L}_{0}}}} (s) & = & \sum_{j=1}^{l_{1}} \kappa_{j}^{-s} - \sum_{j= l_{1}+1}^{l_{1}+l_{2}} (- \kappa_{j})^{-s}
+ H_{2}(s).
\end{eqnarray}

\noindent
Then, (\ref{E:3.20}) and (\ref{E:3.24}) imply that $\zeta_{\B^{2}_{q, {\mathcal P}_{-, {\mathcal L}_{0}}}}(s)$ and
$\eta_{\B_{\even, {\mathcal P}_{-, {\mathcal L}_{0}}}}(s)$
have regular values at $s=0$.
Similarly, $\zeta_{\B^{2}_{q, {\mathcal P}_{+, {\mathcal L}_{1}}}}(s)$
and $\eta_{\B_{\even, {\mathcal P}_{+, {\mathcal L}_{1}}}}(s)$ have regular values at $s=0$. We summarize the above arguments as follows.

\begin{lemma} \label{Lemma:3.1}
Let $(M, Y, g^{M})$ be a compact oriented Riemannian manifold with boundary $Y$ and $E \rightarrow M$ be a complex flat vector
bundle satisfying the Assumption A. We assume that $g^{M}$ is a product metric near $Y$.
Then :  \newline
(1) For each $q$, $\zeta_{\B^{2}_{q, {\mathcal P}_{-, {\mathcal L}_{0}}}}(s)$ and
$\zeta_{\B^{2}_{q, {\mathcal P}_{+, {\mathcal L}_{1}}}}(s)$ have regular values
at $s=0$ and hence the zeta-determinants $\Det \left( \B^{2}_{q, {\mathcal P}_{-, {\mathcal L}_{0}}} \right)$ and
$\Det \left( \B^{2}_{q, {\mathcal P}_{+, {\mathcal L}_{1}}} \right)$ are well-defined.  \newline
(2) $\eta_{\B_{\even, {\mathcal P}_{-, {\mathcal L}_{0}}}}(s)$ and
$\eta_{\B_{\even, {\mathcal P}_{+, {\mathcal L}_{1}}}}(s)$ have regular values at $s=0$.  \newline
(3) The zeta-determinants $\Det \left(\B_{\even, {\mathcal P}_{-, {\mathcal L}_{0}}} \right)$ and
$\Det \left(\B_{\even, {\mathcal P}_{+, {\mathcal L}_{1}}} \right)$ are well-defined.
\end{lemma}

\vspace{0.2 cm}

\noindent
For later use we define the eta invariant $\eta(\B_{even, {\mathcal P}_{-, {\mathcal L}_{0}}})$ for
$\B_{even, {\mathcal P}_{-, {\mathcal L}_{0}}}$ as follows.

\begin{definition} \label{Definition:3.51}
Let ${\frak L}^{+}_{{\mathcal P}_{-, {\mathcal L}_{0}}}$ (${\frak L}^{-}_{{\mathcal P}_{-, {\mathcal L}_{0}}}$)
and ${\frak L}^{0}_{{\mathcal P}_{-, {\mathcal L}_{0}}}$
be the dimensions of the generalized eigenspaces corresponding to positive (negative) imaginary generalized eigenvalues and zero generalized eigenvalue
of $\B_{even, {\mathcal P}_{-, {\mathcal L}_{0}}}$, respectively.
We define $\eta(\B_{even, {\mathcal P}_{-, {\mathcal L}_{0}}})$ by

$$
\eta(\B_{even, {\mathcal P}_{-, {\mathcal L}_{0}}}) = \frac{1}{2} \left( \eta_{\B_{even, {\mathcal P}_{-, {\mathcal L}_{0}}}}(0)
+ {\frak L}^{+}_{{\mathcal P}_{-, {\mathcal L}_{0}}} - {\frak L}^{-}_{{\mathcal P}_{-, {\mathcal L}_{0}}} +
{\frak L}^{0}_{{\mathcal P}_{-, {\mathcal L}_{0}}}
\right).
$$

\noindent
We define $\eta(\B_{even, {\mathcal P}_{+, {\mathcal L}_{1}}})$ in the same way.
\end{definition}

\vspace{0.3 cm}

\subsection{The value of zeta functions at $s=0$}

Before finishing this section we compute the values of the zeta functions
$\zeta_{\B^{2}_{q, {\mathcal P}_{-, {\mathcal L}_{0}}/ {\mathcal P}_{+, {\mathcal L}_{1}}}}(s)$ at $s=0$ for later use.
Let ${\frak L}^{0, q}_{{\mathcal P}_{-, {\mathcal L}_{0}}}$ and ${\frak L}^{0, q}_{{\mathcal P}_{+, {\mathcal L}_{1}}}$ be
the dimensions of the generalized
$0$-eigenspaces of $\B^{2}_{q, {\mathcal P}_{-, {\mathcal L}_{0}}}$ and $\B^{2}_{q, {\mathcal P}_{+, {\mathcal L}_{1}}}$, respectively.
Then by (\ref{E:3.3}) and (\ref{E:3.4}) we have

\begin{eqnarray}  \label{E:3.25}
\zeta_{\B^{2}_{q, {\mathcal P}_{-, {\mathcal L}_{0}}}}(0) & + & {\frak L}^{0, q}_{{\mathcal P}_{-, {\mathcal L}_{0}}}
\hspace{0.1 cm} = \hspace{0.1 cm}
\lim_{s \rightarrow 0} \frac{1}{\Gamma(s)} \int_{0}^{1} t^{s-1} \Tr e^{-t \B^{2}_{q, {\mathcal P}_{-, {\mathcal L}_{0}}}} \hspace{0.1 cm} dt \nonumber \\
& = & \lim_{s \rightarrow 0} \frac{1}{\Gamma(s)} \int_{0}^{1} t^{s-1} \int_{M} \Tr {\mathcal Q}_{q}(t, x, x) \hspace{0.1 cm} d vol(x) dt  \nonumber \\
& = & \lim_{s \rightarrow 0} \frac{1}{\Gamma(s)} \int_{0}^{1} t^{s-1} \int_{Y \times [0, \infty)}
\Tr {\mathcal E}^{\cyl}_{q}(t, (u, y), (u, y)) \psi_{1}(u) \hspace{0.1 cm} d vol(y) du dt,
\end{eqnarray}

\noindent
where in the last equality we used the fact that $m = \Dim {\widetilde M}$ is odd.
Moreover, (\ref{E:3.17}) shows that

\begin{eqnarray*}
& & \lim_{s \rightarrow 0} \frac{1}{\Gamma(s)} \int_{0}^{1} t^{s-1}
\left( \sum_{\lambda^{-}_{q, k}} \frac{e^{- t \lambda^{-}_{q, k}}}{\sqrt{4 \pi t}} \int_{0}^{\infty} \psi_{1}(u) du +
\sum_{\lambda^{+}_{q, k}} \frac{e^{- t \lambda^{+}_{q, k}}}{\sqrt{4 \pi t}} \int_{0}^{\infty} \psi_{1}(u) du \right) dt  \\
& = & \left( \lim_{s \rightarrow 0} \frac{1}{\Gamma(s)} \int_{0}^{1} t^{s-\frac{3}{2}} \left( \Tr e^{- t \B_{Y, q}^{2}} \right) dt \right)
\left( \frac{1}{\sqrt{4 \pi}} \int_{0}^{\infty} \psi_{1}(u) du \right) \hspace{0.1 cm} = \hspace{0.1 cm} 0,
\end{eqnarray*}

\noindent
since $\hspace{0.1 cm} \Tr e^{-t \B_{Y, q}^{2}} \hspace{0.1 cm}$ has an asymptotic expansion of the form
$\hspace{0.1 cm} \sum_{j=0}^{\infty} a_{j} t^{- \frac{m-1}{2} + j}$.
Similarly, we have
$$
\lim_{s \rightarrow 0} \frac{1}{\Gamma(s)} \int_{0}^{1} t^{s-1}
\left( \sum_{\lambda^{-}_{q-1, k}} \frac{e^{- t \lambda^{-}_{q-1, k}}}{\sqrt{4 \pi t}} \int_{0}^{\infty} \psi_{1}(u) du +
\sum_{\lambda^{+}_{q-1, k}} \frac{e^{- t \lambda^{+}_{q-1, k}}}{\sqrt{4 \pi t}} \int_{0}^{\infty} \psi_{1}(u) du \right) dt
= 0.
$$
These two equalities with (\ref{E:3.17}) and (\ref{E:3.25}) lead to

\begin{eqnarray}  \label{E:3.26}
& & \zeta_{\B^{2}_{q, {\mathcal P}_{-, {\mathcal L}_{0}}}}(0) + {\frak L}^{0, q}_{{\mathcal P}_{-, {\mathcal L}_{0}}} \\
& = &  \lim_{s \rightarrow 0} \frac{1}{\Gamma(s)}
 \int_{0}^{1} t^{s-1} \Tr \left( e^{- t \B_{Y, q}^{2, +}} -  e^{- t \B_{Y, q}^{2, -}}
+  e^{- t \B_{Y, q-1}^{2, +}} -  e^{- t \B_{Y, q-1}^{2, -}}  \right)
\left( \frac{1}{\sqrt{4 \pi t}} \int_{0}^{\infty} e^{- \frac{u^{2}}{t}} \psi_{1}(u) du \right) dt. \nonumber
\end{eqnarray}

\noindent
We denote
\begin{equation} \label{E:3.270}
l_{q} := \Dim \Lambda^{q}_{0}(Y, E|_{Y}), \qquad
l_{q}^{-} := \Dim {\mathcal K}^{q}, \qquad
l_{q}^{+} := \Dim ( \Gamma^{Y} {\mathcal K} )^{q},
\end{equation}
so that $l_{q} = l_{q}^{-} + l_{q}^{+}$.
Using the fact $\hspace{0.1 cm} \zeta_{\B_{Y, q}^{2, -}}(s) = \zeta_{\B_{Y, q-1}^{2, +}}(s) \hspace{0.1 cm}$ (cf. (\ref{E:3.10}), (\ref{E:3.14}))
and (\ref{E:3.26}), we have

\begin{eqnarray}  \label{E:3.27}
& & \zeta_{\B^{2}_{q, {\mathcal P}_{-, {\mathcal L}_{0}}}}(0) + {\frak L}^{0, q}_{{\mathcal P}_{-, {\mathcal L}_{0}}} \nonumber \\
& = & \frac{1}{4} \left(
\zeta_{\B_{Y, q}^{2, +}}(0) + l_{q}^{+} - \zeta_{\B_{Y, q}^{2, -}}(0) - l_{q}^{-}
+ \zeta_{\B_{Y, q-1}^{2, +}}(0) + l_{q-1}^{+} - \zeta_{\B_{Y, q-1}^{2, -}}(0) - l_{q-1}^{-} \right)  \nonumber \\
& = & \frac{1}{4} \left( \zeta_{\B_{Y, q}^{2, +}}(0) - \zeta_{\B_{Y, q-1}^{2, -}}(0) + l_{q}^{+} - l_{q}^{-} + l_{q-1}^{+}  - l_{q-1}^{-} \right).
\end{eqnarray}

\noindent
Similarly, we have

\begin{eqnarray}   \label{E:3.28}
\zeta_{\B^{2}_{q, {\mathcal P}_{+, {\mathcal L}_{1}}}}(0) + {\frak L}^{0, q}_{{\mathcal P}_{+, {\mathcal L}_{1}}} & = &
\frac{1}{4} \left( - \zeta_{\B_{Y, q}^{2, +}}(0)  + \zeta_{\B_{Y, q-1}^{2, -}}(0) - l_{q}^{+} + l_{q}^{-} - l_{q-1}^{+}  + l_{q-1}^{-} \right).
\end{eqnarray}

\noindent
Summarizing the above argument, we have the following result.

\vspace{0.2 cm}

\begin{lemma}  \label{Lemma:3.2}
\begin{eqnarray*}
\zeta_{\B^{2}_{q, {\mathcal P}_{-, {\mathcal L}_{0}}}}(0) + {\frak L}^{0, q}_{{\mathcal P}_{-, {\mathcal L}_{0}}}
& = & \frac{1}{4} \left( \zeta_{\B_{Y, q}^{2, +}}(0) - \zeta_{\B_{Y, q-1}^{2, -}}(0) + l_{q}^{+} - l_{q}^{-} + l_{q-1}^{+}  - l_{q-1}^{-} \right)  \\
& = & - \left( \zeta_{\B^{2}_{q, {\mathcal P}_{+, {\mathcal L}_{1}}}}(0) + {\frak L}^{0, q}_{{\mathcal P}_{+, {\mathcal L}_{1}}} \right).
\end{eqnarray*}
\end{lemma}

\vspace{0.3 cm}

\section{The graded zeta-determinant of the odd signature operator on a compact manifold with boundary}

\vspace{0.2 cm}

\subsection{The graded determinant of $\B_{\even, {\mathcal P}_{-, {\mathcal L}_{0}}}$}

In this section we are going to define the graded determinant of $\B_{\even, {\mathcal P}_{-, {\mathcal L}_{0}}}$ and
the refined analytic torsion under the Assumption I and II below.
We begin with the following definitions.

\begin{definition} \label{Definition:4.1}
We define projections ${\widetilde {\mathcal P}}_{0}$,
${\widetilde {\mathcal P}}_{1} : \Omega^{\bullet}(Y, E|_{Y}) \oplus \Omega^{\bullet}(Y, E|_{Y}) \rightarrow
\Omega^{\bullet}(Y, E|_{Y}) \oplus \Omega^{\bullet}(Y, E|_{Y})$ as follows.
For $\phi \in \Omega^{q}(M, E)$
$$
{\widetilde {\mathcal P}}_{0} (\phi|_{Y}) = \begin{cases} {\mathcal P}_{-, {\mathcal L}_{0}} (\phi|_{Y}) \quad
\text{if} \quad q \quad \text{is} \quad \text{even} \\
{\mathcal P}_{+, {\mathcal L}_{1}} (\phi|_{Y}) \quad \text{if} \quad q \quad \text{is} \quad \text{odd} ,
\end{cases}
\qquad
{\widetilde {\mathcal P}}_{1} (\phi|_{Y}) = \begin{cases} {\mathcal P}_{+, {\mathcal L}_{1}} (\phi|_{Y}) \quad
\text{if} \quad q \quad \text{is} \quad \text{even} \\
{\mathcal P}_{-, {\mathcal L}_{0}} (\phi|_{Y}) \quad \text{if} \quad q \quad \text{is} \quad \text{odd} .
\end{cases}
$$

\end{definition}

\noindent
For each $q$ we define

\begin{eqnarray}  \label{E:4.1}
\Omega^{q, \infty}_{{\mathcal P}_{-, {\mathcal L}_{0}}}(M, E) & = & \{ \phi \in \Omega^{q}(M, E) \mid
{\mathcal P}_{-, {\mathcal L}_{0}} \left( \left( \B^{l} \phi \right)|_{Y}\right) = 0, \quad l = 0, 1, 2, \cdots  \}, \nonumber \\
\Omega^{q, \infty}_{{\mathcal P}_{+, {\mathcal L}_{1}}}(M, E) & = & \{ \phi \in \Omega^{q}(M, E) \mid
{\mathcal P}_{+, {\mathcal L}_{1}} \left( \left( \B^{l} \phi \right)|_{Y}\right) = 0, \quad l = 0, 1, 2, \cdots  \}.
\end{eqnarray}

\noindent
Then by Lemma \ref{Lemma:2.12} we have the following cochain complexes

\begin{eqnarray} \label{E:4.2}
(\Omega^{\bullet, \infty}_{{\widetilde {\mathcal P}}_{0}}(M, E), \hspace{0.1 cm} \nabla) & : &
 0 \longrightarrow \Omega^{0, \infty}_{{\mathcal P}_{-, {\mathcal L}_{0}}}(M,E) \stackrel{\nabla}{\longrightarrow}
\Omega^{1, \infty}_{{\mathcal P}_{+, {\mathcal L}_{1}}}(M,E) \stackrel{\nabla}{\longrightarrow} \cdots
\stackrel{\nabla}{\longrightarrow}\Omega^{m, \infty}_{{\mathcal P}_{+, {\mathcal L}_{1}}}(M,E) \longrightarrow 0, \\
(\Omega^{\bullet, \infty}_{{\widetilde {\mathcal P}}_{1}}(M, E), \hspace{0.1 cm} \nabla) & : &
 0 \longrightarrow \Omega^{0, \infty}_{{\mathcal P}_{+, {\mathcal L}_{1}}}(M,E) \stackrel{\nabla}{\longrightarrow}
\Omega^{1, \infty}_{{\mathcal P}_{-, {\mathcal L}_{0}}}(M,E) \stackrel{\nabla}{\longrightarrow} \cdots
\stackrel{\nabla}{\longrightarrow}\Omega^{m, \infty}_{{\mathcal P}_{-, {\mathcal L}_{0}}}(M,E) \longrightarrow 0.
\end{eqnarray}

\noindent
Recall that $\nabla^{\prime}$ is the dual connection of $\nabla$ with respect to ${\frak h}^{E}$.
Since $\nabla$ and $\nabla^{\prime}$ are same on a collar neighborhood $N$ of $Y$,
Lemma \ref{Lemma:2.12} shows that $\Gamma \nabla^{\prime} \Gamma$ maps
$~ \Omega^{q, \infty}_{{\mathcal P}_{-, {\mathcal L}_{0}}/{\mathcal P}_{+, {\mathcal L}_{1}}}(M,E) ~$ into
$~ \Omega^{q-1, \infty}_{{\mathcal P}_{+, {\mathcal L}_{1}}/{\mathcal P}_{-, {\mathcal L}_{0}}}(M,E) ~$.
This fact together with (\ref{E:701}) and Lemma \ref{Lemma:2.11} leads to the following Hodge decomposition.
\begin{eqnarray}    \label{E:400}
\Omega^{q, \infty}_{{\mathcal P}_{-, {\mathcal L}_{0}}}(M,E) & = &
\Imm \nabla \oplus {\mathcal H}^{q}_{\rel}(M, E) \oplus \Imm \Gamma \nabla^{\prime} \Gamma, \nonumber \\
\Omega^{q, \infty}_{{\mathcal P}_{+, {\mathcal L}_{1}}}(M,E) & = &
\Imm \nabla \oplus {\mathcal H}^{q}_{\Abs}(M, E) \oplus \Imm \Gamma \nabla^{\prime} \Gamma.
\end{eqnarray}
This decomposition leads to the following result (cf. Lemma \ref{Lemma:2.11}).
\begin{lemma} The complexes (4.2) and (4.3) compute the following cohomologies.
\begin{eqnarray*}
H^{q}_{{\mathcal P}_{-, {\mathcal L}_{0}}}(M, E) ~ \cong ~ H^{q}(M, Y ; \rho), \qquad
H^{q}_{{\mathcal P}_{+, {\mathcal L}_{1}}}(M, E) ~ \cong ~ H^{q}(M ; \rho).
\end{eqnarray*}
\end{lemma}

\vspace{0.2 cm}

Now we are going to discuss the refined analytic torsion for the cochain complex (4.2).
The exact same method works for the cochain complex (4.3).
In this section we make the following two assumptions and the general case will be discussed in Section 5.

\vspace{0.2 cm}

\noindent
{\bf Assumption $\I$} : \hspace{0.2 cm}
The cochain complex (4.2) is acyclic.  \newline

\vspace{0.1 cm}

\noindent
{\bf Assumption $\II$} : \hspace{0.2 cm}
$\B : \Omega^{\bullet, \infty}_{{\widetilde {\mathcal P}}_{0}}(M, E) \rightarrow \Omega^{\bullet, \infty}_{{\widetilde {\mathcal P}}_{0}}(M, E)$
is invertible.

\vspace{0.3 cm}

\noindent
For each $q$ we define

\begin{equation}  \label{E:4.4}
\Omega^{q, \infty}_{{\widetilde {\mathcal P}}_{0}, -}(M, E)  =  \Omega^{q, \infty}_{{\widetilde {\mathcal P}}_{0}}(M, E)
\cap \ker \nabla , \qquad
\Omega^{q, \infty}_{{\widetilde {\mathcal P}}_{0}, +}(M, E)  =  \Omega^{q, \infty}_{{\widetilde {\mathcal P}}_{0}}(M, E)
\cap \ker \Gamma \nabla \Gamma .
\end{equation}

\noindent
Then the Assumption $\I$ and $\II$ imply (cf. Subsection 6.9 in \cite{BK1}) that

\begin{equation}  \label{E:4.5}
\Omega^{q, \infty}_{{\widetilde {\mathcal P}}_{0}}(M, E) = \Omega^{q, \infty}_{{\widetilde {\mathcal P}}_{0}, -}(M, E) \oplus
\Omega^{q, \infty}_{{\widetilde {\mathcal P}}_{0}, +}(M, E).
\end{equation}

\noindent
We denote by

\begin{equation} \label{E:4.6}
    \Omega^{\even, \infty}_{{\widetilde {\mathcal P}}_{0}, \pm}(M,E)\ =\ \bigoplus_{p=0}^{r-1}\,
    \Omega^{2p, \infty}_{{\mathcal P}_{-, {\mathcal L}_{0}}, \pm}(M,E),
    \qquad
\Omega^{\odd, \infty}_{{\widetilde {\mathcal P}}_{0}, \pm}(M,E)\ =\ \bigoplus_{p=0}^{r-1}\,
    \Omega^{2p+1, \infty}_{{\mathcal P}_{+, {\mathcal L}_{1}}, \pm}(M,E).
\end{equation}

\noindent
Then, $\B$ maps $\Omega^{\even/\odd, \infty}_{{\widetilde {\mathcal P}}_{0}, \pm}(M,E)$ onto
$\Omega^{\even/\odd, \infty}_{{\widetilde {\mathcal P}}_{0}, \pm}(M,E)$ and induces the following commutative diagram.

\vspace{0.2 cm}

\begin{equation} \label{E:4.7}
\begin{CD}
  \Omega^{\even, \infty}_{{\widetilde {\mathcal P}}_{0}, \pm}(M,E)      & @>\B >>  &
\Omega^{\even, \infty}_{{\widetilde {\mathcal P}}_{0}, \pm}(M,E) \\
                     @V{\Gamma}VV        &                          & @VV \Gamma V \\
 \Omega^{\odd, \infty}_{{\widetilde {\mathcal P}}_{1}, \mp}(M,E)     & @> \B >>  &
\Omega^{\odd, \infty}_{{\widetilde {\mathcal P}}_{1}, \mp}(M,E)
\end{CD}
\end{equation}

\vspace{0.3 cm}

\noindent
We note that the spectra of $\B$ defined on $\Omega^{\even, \infty}_{{\mathcal P}_{-, {\mathcal L}_{0}}, \pm} (M, E)$
($\Omega^{\odd, \infty}_{{\mathcal P}_{+, {\mathcal L}_{1}}, \pm} (M, E)$) are equal to those of $\B$ defined on
$\Dom \left( \B^{\pm}_{\even, {\mathcal P}_{-, {\mathcal L}_{0}}} \right)$
$\left( \Dom \left( \B^{\pm}_{\odd, {\mathcal P}_{+, {\mathcal L}_{1}}} \right) \right)$
and the spectra of $\B^{2}$ defined on $\Omega^{q, \infty}_{{\widetilde {\mathcal P}}_{0}, \pm} (M, E)$ are equal to those of
$\B^{2}$ defined on $\Dom \left( \B^{2, \pm}_{q, {\widetilde {\mathcal P}}_{0}} \right)$.
From this observation we abuse notations a little bit so that we write $\B$ or $\B^{2}$ defined on
$\Omega^{\even, \infty}_{{\mathcal P}_{-, {\mathcal L}_{0}}, \pm} (M, E)$,
$\Omega^{\odd, \infty}_{{\mathcal P}_{+, {\mathcal L}_{1}}, \pm} (M, E)$, $\Omega^{q, \infty}_{{\widetilde {\mathcal P}}_{0}, \pm} (M, E)$
by $\B^{\pm}_{\even, {\mathcal P}_{-, {\mathcal L}_{0}}}$, $\B^{\pm}_{\odd, {\mathcal P}_{+, {\mathcal L}_{1}}}$, $\B^{2, \pm}_{q, {\widetilde {\mathcal P}}_{0}}$, respectively.
The diagram (\ref{E:4.7}) leads to the following result.

\begin{lemma} \label{Lemma:4.2}
(1) For each $q$, the zeta-determinants
$\Det \left( \B^{2, \pm}_{q, {\mathcal P}_{-, {\mathcal L}_{0}}} \right)$ and
$\Det \left( \B^{2, \pm}_{q, {\mathcal P}_{+, {\mathcal L}_{1}}} \right)$ are well-defined and
\begin{equation} \label{E:4.8}
\Det \left( \B^{2, +}_{q, {\mathcal P}_{-, {\mathcal L}_{0}}} \right) =
\Det \left( \B^{2, -}_{q+1, {\mathcal P}_{+, {\mathcal L}_{1}}} \right),  \qquad
\Det \left( \B^{2, -}_{q, {\mathcal P}_{-, {\mathcal L}_{0}}} \right) =
\Det \left( \B^{2, +}_{q-1, {\mathcal P}_{+, {\mathcal L}_{1}}} \right).
\end{equation}
(2) Recall that $\dim M = m = 2r -1$.  If $r$ is odd,
\begin{eqnarray}  \label{E:4.9}
& & \eta_{\B_{\even, {\mathcal P}_{-, {\mathcal L}_{0}}}}(s) \hspace{0.1 cm} = \hspace{0.1 cm} \eta_{\B^{+}_{\even, {\mathcal P}_{-, {\mathcal L}_{0}}}}(s)
\hspace{0.1 cm} = \hspace{0.1 cm} \eta_{\B^{+}_{r-1, {\mathcal P}_{-, {\mathcal L}_{0}}}}(s), \qquad
\eta_{\B^{-}_{\even, {\mathcal P}_{-, {\mathcal L}_{0}}}}(s) \hspace{0.1 cm} = \hspace{0.1 cm} 0,  \nonumber \\
& & \eta_{\B_{\odd, {\mathcal P}_{-, {\mathcal L}_{0}}}}(s) \hspace{0.1 cm} = \hspace{0.1 cm} \eta_{\B^{-}_{\odd, {\mathcal P}_{-, {\mathcal L}_{0}}}}(s)
\hspace{0.1 cm} = \hspace{0.1 cm}
\eta_{\B^{-}_{r, {\mathcal P}_{-, {\mathcal L}_{0}}}}(s), \qquad
\eta_{\B^{+}_{\odd, {\mathcal P}_{-, {\mathcal L}_{0}}}}(s) \hspace{0.1 cm} = \hspace{0.1 cm} 0.
\end{eqnarray}
If $r$ is even,
\begin{eqnarray}  \label{E:4.10}
& & \eta_{\B_{\even, {\mathcal P}_{-, {\mathcal L}_{0}}}}(s) \hspace{0.1 cm} = \hspace{0.1 cm} \eta_{\B^{-}_{\even, {\mathcal P}_{-, {\mathcal L}_{0}}}}(s)
\hspace{0.1 cm} = \hspace{0.1 cm} \eta_{\B^{-}_{r, {\mathcal P}_{-, {\mathcal L}_{0}}}}(s), \qquad
\eta_{\B^{+}_{\even, {\mathcal P}_{-, {\mathcal L}_{0}}}}(s) \hspace{0.1 cm} = \hspace{0.1 cm} 0,  \nonumber \\
& & \eta_{\B_{\odd, {\mathcal P}_{-, {\mathcal L}_{0}}}}(s) \hspace{0.1 cm} = \hspace{0.1 cm} \eta_{\B^{+}_{\odd, {\mathcal P}_{-, {\mathcal L}_{0}}}}(s)
\hspace{0.1 cm} = \hspace{0.1 cm}
\eta_{\B^{+}_{r-1, {\mathcal P}_{-, {\mathcal L}_{0}}}}(s), \qquad
\eta_{\B^{-}_{\odd, {\mathcal P}_{-, {\mathcal L}_{0}}}}(s) \hspace{0.1 cm} = \hspace{0.1 cm} 0.
\end{eqnarray}
\end{lemma}

\begin{proof}
The first assertion comes from Lemma \ref{Lemma:3.1}.
For the second assertion, we consider only $\B^{+}_{\even, {\mathcal P}_{-, {\mathcal L}_{0}}}$. Other cases are treated in the same way.
We note that $\B_{\even}$ maps $\Omega^{q}_{+, {\mathcal P}_{-, {\mathcal L}_{0}}}(M, E) \oplus \Omega^{m-q-1}_{+, {\mathcal P}_{-, {\mathcal L}_{0}}}(M, E)$
into $\Omega^{q}_{+} (M, E) \oplus \Omega^{m-q-1}_{+} (M, E)$.
Let $\phi \in \Omega^{q}_{+, {\mathcal P}_{-, {\mathcal L}_{0}}}(M, E)$ and $\psi \in \Omega^{m-q-1}_{+, {\mathcal P}_{-, {\mathcal L}_{0}}}(M, E)$.
If $\phi + \psi$ is a generalized eigensection of $\B_{\even}$ with a generalized eigenvalue $\lambda$, then $\phi - \psi$
is a generalized eigensection of $\B_{\even}$ with a generalized eigenvalue $-\lambda$.
This observation leads to the second assertion.
\end{proof}

\noindent
Throughout this paper, for an elliptic differential operator $D$ we mean by $\log \Det_{\theta} D$
the particular value of the negative of the derivative of the zeta function at zero, {\it i.e.}
$$
\log \Det_{\theta} D := - \left( \frac{d}{ds} \zeta_{D, \theta} (s) \right)|_{s=0}.
$$

\vspace{0.2 cm}

\noindent
The following result is straightforward.

\begin{lemma} \label{Lemma:4.53}
$$
\log \Det_{\theta} \B^{+}_{\even, {\mathcal P}_{-, {\mathcal L}_{0}}}    =
\frac{1}{2} \log \Det_{2 \theta} \B^{2, +}_{\even, {\mathcal P}_{-, {\mathcal L}_{0}}} - i \pi
\eta(\B^{+}_{\even, {\mathcal P}_{-, {\mathcal L}_{0}}}) + \frac{i \pi}{2}  \zeta_{\B^{2, +}_{\even, {\mathcal P}_{-, {\mathcal L}_{0}}}, 2 \theta} (0).
$$
If $r$ is even with $\dim M = 2r -1$, then $\hspace{0.1 cm} \eta(\B^{+}_{\even, {\mathcal P}_{-, {\mathcal L}_{0}}}) = 0$.
Similar statement holds for $\B^{+}_{\even, {\mathcal P}_{+, {\mathcal L}_{1}}}$.
\end{lemma}

We now define the graded zeta-determinant of $\B_{\even}$ with respect to
${\mathcal P}_{-, {\mathcal L}_{0}}$ and ${\mathcal P}_{+, {\mathcal L}_{1}}$.

\vspace{0.2 cm}

\begin{definition} \label{Definition:4.3}
Under the Assumption $\I$ and $\II$,
we define the graded determinant $\Det_{\gr, \theta} \B_{\even, {\mathcal P}_{-, {\mathcal L}_{0}}}$ and
$\Det_{\gr, \theta} \B_{\even, {\mathcal P}_{+, {\mathcal L}_{1}}}$ by
$$
\Det_{\gr, \theta} \B_{\even, {\mathcal P}_{-, {\mathcal L}_{0}}} = \frac{\Det_{\theta}
\B^{+}_{\even, {\mathcal P}_{-, {\mathcal L}_{0}}}}{\Det_{\theta} (- \B^{-}_{\even, {\mathcal P}_{-, {\mathcal L}_{0}}})},
\qquad
\Det_{\gr, \theta} \B_{\even, {\mathcal P}_{+, {\mathcal L}_{1}}} = \frac{\Det_{\theta}
\B^{+}_{\even, {\mathcal P}_{+, {\mathcal L}_{1}}}}{\Det_{\theta} (- \B^{-}_{\even, {\mathcal P}_{+, {\mathcal L}_{1}}})}.
$$
\end{definition}

\vspace{0.2 cm}

\noindent
We introduce the notations $\B^{2}_{q, {\widetilde {\mathcal P}_{0}}}$ and $\B^{2}_{q, {\widetilde {\mathcal P}_{1}}}$
defined by

\begin{equation}  \label{E:4.11}
\B^{2}_{q, {\widetilde {\mathcal P}_{0}}}  =  \begin{cases} \B^{2}_{q, {\mathcal P}_{-, {\mathcal L}_{0}}}
\hspace{0.3 cm} \text{for} \hspace{0.3 cm}  q  \hspace{0.2 cm} \text{even}   \\
\B^{2}_{q, {\mathcal P}_{+, {\mathcal L}_{1}}}
\hspace{0.3 cm} \text{for} \hspace{0.3 cm}  q  \hspace{0.2 cm} \text{odd},
\end{cases}
\qquad
\B^{2}_{q, {\widetilde {\mathcal P}_{1}}}  =  \begin{cases} \B^{2}_{q, {\mathcal P}_{+, {\mathcal L}_{1}}}
\hspace{0.3 cm} \text{for} \hspace{0.2 cm}  q  \hspace{0.3 cm} \text{even}   \\
\B^{2}_{q, {\mathcal P}_{-, {\mathcal L}_{0}}}
\hspace{0.3 cm} \text{for} \hspace{0.2 cm}  q  \hspace{0.3 cm} \text{odd}.
\end{cases}
\end{equation}

\noindent
Then simple computation shows that

\begin{eqnarray}  \label{E:4.12}
& & \log \Det_{\gr, \theta} (\B_{\even, {\mathcal P}_{-, {\mathcal L}_{0}}}) \hspace{0.1 cm}
 = \hspace{0.1 cm}  \log \Det_{\theta}
(\B^{+}_{\even, {\mathcal P}_{-, {\mathcal L}_{0}}}) - \log \Det_{\theta}(- \B^{-}_{\even, {\mathcal P}_{-, {\mathcal L}_{0}}})  \nonumber \\
& = & \frac{1}{2} \left( \log \Det_{2\theta}\B^{2, +}_{\even, {\mathcal P}_{-, {\mathcal L}_{0}}} -
\log \Det_{2\theta} \B^{2, -}_{\even, {\mathcal P}_{-, {\mathcal L}_{0}}} \right) \nonumber \\
 & & + \frac{\pi i }{2} \left(\zeta_{ \B^{2, +}_{\even, {\mathcal P}_{-, {\mathcal L}_{0}}}}(0) -
\zeta_{\B^{2, -}_{\even, {\mathcal P}_{-, {\mathcal L}_{0}}}}(0) \right)
 - i \pi \left( \eta (\B^{+}_{\even, {\mathcal P}_{-, {\mathcal L}_{0}}})
+ \eta (\B^{-}_{\even, {\mathcal P}_{-, {\mathcal L}_{0}}}) \right) \nonumber \\
& = & \frac{1}{2} \left( \log \Det_{2\theta} \B^{2, +}_{\even, {\mathcal P}_{-, {\mathcal L}_{0}}} -
\log \Det_{2\theta} \B^{2, +}_{\odd, {\mathcal P}_{+, {\mathcal L}_{1}}} \right)    \nonumber  \\
& & + \frac{\pi i }{2} \left(\zeta_{ \B^{2, +}_{\even, {\mathcal P}_{-, {\mathcal L}_{0}}}}(0)  -
\zeta_{\B^{2, +}_{\odd, {\mathcal P}_{+, {\mathcal L}_{1}}}}(0) \right)
 - i \pi \eta (\B_{\even, {\mathcal P}_{-, {\mathcal L}_{0}}}) \nonumber \\
& = & \frac{1}{2} \sum_{q=0}^{m} (-1)^{q+1} \cdot q \cdot \log \Det_{2\theta} \B^{2}_{q, {\widetilde {\mathcal P}}_{0}}
-  i \pi \eta (\B_{\even, {\mathcal P}_{-, {\mathcal L}_{0}}})
  + \frac{\pi i }{2} \sum_{q=0}^{m} (-1)^{q+1} \cdot q \cdot \zeta_{\B^{2}_{q, {\widetilde {\mathcal P}_{0}}}}(0).
\end{eqnarray}

\vspace{0.2 cm}

\noindent
Similarly, we have

\begin{eqnarray}  \label{E:4.13}
\log \Det_{\gr, \theta} (\B_{\even, {\mathcal P}_{+, {\mathcal L}_{1}}})
& = & \frac{1}{2} \sum_{q=0}^{m} (-1)^{q+1} \cdot q \cdot \log \Det_{2\theta} \B^{2}_{q, {\widetilde {\mathcal P}_{1}}}
- i \pi \eta (\B_{\even, {\mathcal P}_{+, {\mathcal L}_{1}}})   \nonumber \\
& &  + \hspace{0.1 cm}
\frac{\pi i }{2} \sum_{q=0}^{m} (-1)^{q+1} \cdot q \cdot \zeta_{\B^{2}_{q, {\widetilde {\mathcal P}_{1}}}}(0).
\end{eqnarray}

\vspace{0.3 cm}

\noindent
The equations (\ref{E:4.12}) and (\ref{E:4.13}) with Lemma \ref{Lemma:3.2} lead to the following results.

\vspace{0.2 cm}

\begin{corollary}  \label{Corollary:4.5}
Putting  $l_{q}^{\pm} = \dim \ker \B^{2, \pm}_{Y, q}$, the following equalities hold under the Assumption $\I$ and $\II$.

\begin{eqnarray*}
(1) \hspace{0.2 cm} \log \Det_{\gr, \theta} (\B_{\even, {\mathcal P}_{-, {\mathcal L}_{0}}}) & = &
\frac{1}{2} \sum_{q=0}^{m} (-1)^{q+1} \cdot q \cdot \log \Det_{2\theta} \B^{2}_{q, {\widetilde {\mathcal P}_{0}}}
-  i \pi \hspace{0.1 cm} \eta(\B_{\even, {\mathcal P}_{-, {\mathcal L}_{0}}}) \\
& + & \frac{\pi i }{2} \left( \frac{1}{4} \sum_{q=0}^{m-1} \zeta_{\B_{Y, q}^{2}}(0) + \sum_{q=0}^{r-2}(r-1-q) (l_{q}^{+} - l_{q}^{-})  \right). \\
(2) \hspace{0.2 cm}  \log \Det_{gr, \theta} (\B_{\even, {\mathcal P}_{+, {\mathcal L}_{1}}})  & = &
 \frac{1}{2} \sum_{q=0}^{m} (-1)^{q+1} \cdot q \cdot \log \Det_{2\theta} \B^{2}_{q, {\widetilde {\mathcal P}_{1}}}
- i \pi \hspace{0.1 cm} \eta(\B_{\even, {\mathcal P}_{+, {\mathcal L}_{1}}}) \\
& - & \frac{\pi i }{2} \left(\frac{1}{4} \sum_{q=0}^{m-1} \zeta_{\B_{Y, q}^{2}}(0)  +  \sum_{q=0}^{r-2}(r-1-q) (l_{q}^{+} - l_{q}^{-}) \right).
\end{eqnarray*}

\end{corollary}

\vspace{0.2 cm}

\subsection{The metric dependency of the graded determinant of $\B_{\even, {\mathcal P}_{-, {\mathcal L}_{0}}}$}

In this subsection we discuss the metric dependency of the graded determinant $\Det_{\gr, \theta} \B_{\even, {\mathcal P}_{-, {\mathcal L}_{0}}}$.
Let $\{ g_{v}^{M} \mid -\delta < v < \delta \}$, $\delta > 0$, be a family of Riemannian metrics on $M$ such that each $g_{v}^{M}$
is a product metric and does not vary on $[0, \epsilon ) \times Y$ for fixed $\epsilon > 0$.
We denote by $\B_{\even}(v)$ ($\B^{2}_{q}(v)$) the (square of) odd signature operator corresponding to $g_{v}^{M}$.
In this subsection we assume that $\B_{\even}(v)$ and $\B^{2}_{q}(v)$ have the same Agmon angles $\theta$ and $2 \theta$ ($- \frac{\pi}{2} < \theta < 0$)
which does not depend on $v$.

\vspace{0.2 cm}

\begin{lemma} \label{Lemma:4.6}
We assume that for each $q$,
$\B^{2}_{q, {\widetilde {\mathcal P}}_{0}}(v) : \Omega^{q, \infty}_{{\widetilde {\mathcal P}}_{0}} (M, E) \rightarrow
\Omega^{q, \infty}_{{\widetilde {\mathcal P}}_{0}} (M, E)$ is an invertible operator and has the same Agmon angle $2 \theta$. Then :
$$
\frac{d}{dv} \sum_{q=0}^{m} (-1)^{q+1} \cdot q \cdot \log \Det_{2\theta} \B^{2}_{q, {\widetilde {\mathcal P}_{0}}}(v)
\hspace{0.1 cm} = \hspace{0.1 cm} 0.
$$
The same assertion holds for $\B^{2}_{q, {\widetilde {\mathcal P}}_{1}}(v)$.
\end{lemma}

\begin{proof}
The metric $g_{v}^{M}$ varies only in the interior of the manifold, which shows that
the boundary condition does not change and  the derivative of $g_{v}$ vanishes near the boundary.
Since the argument of Lemma 9.9 in \cite{BK1} (or in \cite{RS}) is purely local, verbatim repetition gives the proof.
\end{proof}

\vspace{0.2 cm}

We next discuss the variation of the eta invariant $\eta(\B_{\even, {\mathcal P}_{-, {\mathcal L}_{0}}}(v))$ for $- \delta < v < \delta$.
We are going to follow the arguments in \cite{Mu}.
We choose $c > 0$ such that (i) any generalized eigenvalue $\kappa$ of $\B_{\even, {\mathcal P}_{-, {\mathcal L}_{0}}}(v)$
with $| \kappa | > c$ satisfies $\re (\kappa^{2}) > 0$ and (ii)
$\B_{\even, {\mathcal P}_{-, {\mathcal L}_{0}}}(v)$ does not have any generalized eigenvalues of absolute value $c$ for $- \delta_{0} < v < \delta_{0}$,
$0 < \delta_{0} \leq \delta$.
Let $K(v)$ be the subspace of $\Omega^{\bullet} (M, E)$ spanned by the generalized eigenforms with generalized eigenvalues
$\kappa$ with $|\kappa| \leq c$.
Theorem \ref{Theorem:2.22} shows that $K(v)$ is a finite dimensional vector space with a constant dimension with respect to $v$
for $- \delta_{0} < v < \delta_{0}$.
We denote by $Q(v)$ the spectral projection onto $K(v)$ and denote $P(v) = I - Q(v)$.
Then $Q(v)$ and $P(v)$ are smooth $1$-parameter families of projections.
We define
$$
{\widetilde \eta}(s, v) = \frac{1}{\Gamma(\frac{s+1}{2})} \int_{0}^{\infty} t^{\frac{s-1}{2}} \Tr \left( P(v) \B_{\even}(v)
e^{- t \B_{\even, {\mathcal P}_{-, {\mathcal L}_{0}}}(v)^{2}} \right) dt.
$$
For fixed $v$, $\hspace{0.1 cm} \eta_{\B_{\even, {\mathcal P}_{-, {\mathcal L}_{0}}}(v)} (s) - {\widetilde \eta}(s, v) \hspace{0.1 cm}$ is an entire function of $s$ and hence
${\widetilde \eta}(s, v)$ has a regular value at $s=0$. Moreover,
$\left( \eta_{\B_{\even, {\mathcal P}_{-, {\mathcal L}_{0}}}(v)} (s) - {\widetilde \eta}(s, v) \right)|_{s=0}$ does not depend on $v$
modulo $2 {\Bbb Z}$.
To compute the variation of ${\widetilde \eta}(s, v)$ we note that

\begin{eqnarray} \label{E:4.22}
& & \frac{\partial}{\partial v} \Tr \left( P(v) \B_{\even}(v) e^{- t \B_{\even, {\mathcal P}_{-, {\mathcal L}_{0}}}(v)^{2}} \right) \hspace{0.1 cm}
= \hspace{0.1 cm} \Tr \left( \dot{P}(v) \B_{\even}(v) e^{- t \B_{\even, {\mathcal P}_{-, {\mathcal L}_{0}}}(v)^{2}} \right)   \nonumber \\
& & \hspace{0.2 cm} + \hspace{0.1 cm}
\Tr \left( P(v) \dot{\B}_{\even}(v) e^{- t \B_{\even, {\mathcal P}_{-, {\mathcal L}_{0}}}(v)^{2}} \right) \hspace{0.1 cm}
+ \hspace{0.1 cm}  \Tr \left( P(v) \B_{\even}(v) \frac{\partial}{\partial v} e^{- t \B_{\even, {\mathcal P}_{-, {\mathcal L}_{0}}}(v)^{2}} \right),
\end{eqnarray}

\noindent
where $\dot{P}(v)$ and $\dot{\B}_{\even}(v)$ denote the derivatives of  $P(v)$ and $\B_{\even}(v)$ with respect to $v$.
Since $\dot{P}(v) = - \dot{Q}(v)$, $P(v) Q(v) = 0$, $\dot{P}(v) = \dot{P}(v) P(v) + P(v) \dot{P}(v)$, $\dot{Q}(v) = \dot{Q}(v) Q(v) + Q(v) \dot{Q}(v)$
and $\B_{\even}(v)$ commutes with $P(v)$ and $Q(v)$, we have

\begin{eqnarray}  \label{E:4.23}
& &  \Tr \left( \dot{P}(v) \B_{\even}(v) e^{- t \B_{\even, {\mathcal P}_{-, {\mathcal L}_{0}}}(v)^{2}} \right)
\hspace{0.1 cm} = \hspace{0.1 cm} 2 \hspace{0.1 cm} \Tr \left( \dot{P}(v) \B_{\even}(v) e^{- t \B_{\even, {\mathcal P}_{-, {\mathcal L}_{0}}}(v)^{2}} P(v) \right)  \nonumber  \\
& = & - \hspace{0.1 cm} 2 \hspace{0.1 cm} \Tr \left( \dot{Q}(v) \B_{\even}(v) e^{- t \B_{\even, {\mathcal P}_{-, {\mathcal L}_{0}}}(v)^{2}} P(v) \right)
\hspace{0.1 cm} = \hspace{0.1 cm}  - \hspace{0.1 cm} 2 \hspace{0.1 cm} \Tr \left( \dot{Q}(v) Q(v) \B_{\even}(v) e^{- t \B_{\even, {\mathcal P}_{-, {\mathcal L}_{0}}}(v)^{2}} P(v) \right)   \nonumber  \\
& & - \hspace{0.1 cm} 2 \hspace{0.1 cm} \Tr \left( Q(v) \dot{Q}(v) \B_{\even}(v) e^{- t \B_{\even, {\mathcal P}_{-, {\mathcal L}_{0}}}(v)^{2}} P(v) \right)
\hspace{0.1 cm} = \hspace{0.1 cm}  0.
\end{eqnarray}

\noindent
The following formula is well-known (cf. \cite{Mu}).
\begin{eqnarray}  \label{E:4.24}
& & \frac{\partial}{\partial v} e^{- t \B_{\even, {\mathcal P}_{-, {\mathcal L}_{0}}}(v)^{2}}  \nonumber  \\
& = & - \int_{0}^{t} e^{- (t - r) \B_{\even, {\mathcal P}_{-, {\mathcal L}_{0}}}(v)^{2}} ( \dot{\B}_{\even}(v) \B_{\even}(v) +
\B_{\even}(v) \dot{\B}_{\even}(v) ) e^{- r \B_{\even, {\mathcal P}_{-, {\mathcal L}_{0}}}(v)^{2}} dr.
\end{eqnarray}
Hence, we have
\begin{equation}
\frac{\partial}{\partial v} \Tr \left( P(v) \B_{\even}(v) e^{- t \B_{\even, {\mathcal P}_{-, {\mathcal L}_{0}}}(v)^{2}} \right) \hspace{0.1 cm}
= \hspace{0.1 cm} (1 + 2t \frac{d}{dt} )
\Tr \left( P(v) \dot{\B}_{\even}(v) e^{- t \B_{\even, {\mathcal P}_{-, {\mathcal L}_{0}}}(v)^{2}} \right) \hspace{0.1 cm},
\end{equation}
which leads to

\begin{equation} \label{E:4.25}
\frac{\partial}{\partial v} {\widetilde \eta}(s, v) \hspace{0.1 cm} = \hspace{0.1 cm}
- \frac{s}{\Gamma(\frac{s+1}{2})} \int_{0}^{\infty} t^{\frac{s-1}{2}}
\Tr \left( P(v) \dot{\B}_{\even}(v) e^{- t \B_{\even, {\mathcal P}_{-, {\mathcal L}_{0}}}(v)^{2}} \right) dt.
\end{equation}

\noindent
By the same way as in Section $3$, there exists an asymptotic expansion for $t \rightarrow 0^{+}$
\begin{equation} \label{E:4.26}
\Tr \left( P(v) \dot{\B}_{\even}(v) e^{- t \B_{\even, {\mathcal P}_{-, {\mathcal L}_{0}}}(v)^{2}} \right) \hspace{0.1 cm} \sim \hspace{0.1 cm}
\sum_{j=0}^{\infty} {\widetilde c}_{j}(v) t^{- \frac{m+1-j}{2}},
\end{equation}
which shows that $\frac{\partial}{\partial v} {\widetilde \eta}(s, v)$ has a regular value at $s=0$ and
\begin{equation} \label{E:4.27}
\frac{\partial}{\partial v} {\widetilde \eta}(s, v)|_{s=0} \hspace{0.1 cm} = \hspace{0.1 cm}
- \frac{2 {\widetilde c}_{m}(v)}{\sqrt{\pi}}.
\end{equation}
Setting
\begin{equation} \label{E:4.28}
\Tr \left( \dot{\B}_{\even}(v) e^{- t \B_{\even, {\mathcal P}_{-, {\mathcal L}_{0}}}(v)^{2}} \right) \hspace{0.1 cm} \sim \hspace{0.1 cm}
\sum_{j=0}^{\infty} c_{j}(v) t^{- \frac{m+1-j}{2}},
\end{equation}
we have ${\widetilde c}_{j}(v) = c_{j}(v)$ for $0 \leq j \leq m$.
Hence we have the following result.

\vspace{0.2 cm}

\begin{lemma} \label{Lemma:4.20}
Let $\B_{\even, {\mathcal P}_{-, {\mathcal L}_{0}}}(v)$ be a smooth $1$-parameter family of odd signature operators which does not vary
near the boundary and have the same Agmon angle $\theta$ ($- \frac{\pi}{2} < \theta < 0$). Then :
$$
\frac{\partial}{\partial v} \eta(\B_{\even, {\mathcal P}_{-, {\mathcal L}_{0}}}(v)) = - \frac{c_{m}(v)}{\sqrt{\pi}} \qquad
(\Mod \hspace{0.1 cm} {\Bbb Z}),
$$
where $c_{m}(v)$ is the coefficient of $t^{-\frac{1}{2}}$ in (\ref{E:4.28}).
\end{lemma}

\vspace{0.2 cm}

To cancel the metric dependence of the eta invariant we consider the following odd signature operator
$\B^{\trivial}_{\even, {\mathcal P}_{-, {\mathcal L}_{0}}}(v) : \Omega^{\even, \infty}_{{\mathcal P}_{-, {\mathcal L}_{0}}} (M, {\Bbb C}) \rightarrow
\Omega^{\even, \infty}_{{\mathcal P}_{-, {\mathcal L}_{0}}} (M, {\Bbb C})$ obtained by the trivial connection $\nabla^{\trivial}$
on the trivial line bundle $M \times {\Bbb C} \rightarrow M$
with respect to the metric $g_{v}$, $- \delta < v < \delta$.
Since $\B^{\trivial}_{\even, {\mathcal P}_{-, {\mathcal L}_{0}}} (v)$ is a self-adjoint operator,
we define the eta invariant $\eta(\B^{\trivial}_{\even, {\mathcal P}_{-, {\mathcal L}_{0}}} (v))$ for
$\B^{\trivial}_{\even, {\mathcal P}_{-, {\mathcal L}_{0}}} (v)$ by
\begin{eqnarray}   \label{E:4.44}
\eta(\B^{\trivial}_{\even, {\mathcal P}_{-, {\mathcal L}_{0}}} (v)) =
\frac{1}{2} \left( \eta_{\B^{\trivial}_{\even, {\mathcal P}_{-, {\mathcal L}_{0}}} (v)} (0)
+ \dim \ker \B^{\trivial}_{\even, {\mathcal P}_{-, {\mathcal L}_{0}}}(v) \right).
\end{eqnarray}

\noindent
Lemma \ref{Lemma:2.11} shows that $\dim \ker \B^{\trivial}_{\even, {\mathcal P}_{-, {\mathcal L}_{0}}}(v)$ is a topological invariant
and hence it does not depend on the metric $g_{v}^{M}$.
Lemma \ref{Lemma:4.20} implies that

$$
\rank(E) \cdot \frac{d}{dv} \eta (\B^{\trivial}_{\even, {\mathcal P}_{-, {\mathcal L}_{0}}}(v)) \hspace{0.1 cm} = \hspace{0.1 cm}
\frac{1}{2} \hspace{0.1 cm} \rank(E) \cdot \frac{d}{dv} \left( \eta_{\B^{\trivial}_{\even, {\mathcal P}_{-, {\mathcal L}_{0}}} (v)} (0) \right)
\hspace{0.1 cm} = \hspace{0.1 cm} - \rank(E) \cdot \frac{c_{m}^{\trivial}(v)}{\sqrt{\pi}},
$$

\noindent
where $c_{m}^{\trivial}(v)$ is the coefficient of $t^{-\frac{1}{2}}$ in (\ref{E:4.28}) for $\B^{\trivial}_{\even, {\mathcal P}_{-, {\mathcal L}_{0}}}(v)$.
Since $\dot{\B}_{\even {\mathcal P}_{-, {\mathcal L}_{0}}}(v)$ vanishes near the boundary, the asymptotic expansion (\ref{E:4.28}) is determined completely
by the interior data (cf. Section 3).
Since the coefficients $c_{m}(v)$ and $c_{m}^{\trivial}(v)$ are given locally, we have $c_{m}(v) = \rank(E) \cdot c_{m}^{\trivial}(v)$ and hence
we have the following corollary.

\vspace{0.2 cm}

\begin{corollary} \label{Corollary:4.9}
$$
\eta (\B_{\even, {\mathcal P}_{-, {\mathcal L}_{0}}}) -
\rank(E) \cdot \frac{1}{2} \eta_{\B^{\trivial}_{\even, {\mathcal P}_{-, {\mathcal L}_{0}}}}(0)
$$
is invariant on modulo ${\Bbb Z}$ under the change of metrics in the interior of $M$.
\end{corollary}

\vspace{0.2 cm}

Now we are ready to define the refined analytic torsion on a compact manifold with boundary when $\nabla$ satisfies the Assumption $\I$ and $\II$.
\vspace{0.2 cm}

\begin{definition}  \label{Definition:4.10}
Let $(M, Y, g^{M})$ be a compact Riemannian manifold with boundary $Y$, where $g^{M}$ is a product metric near the boundary.
We denote by $E$ the complex flat bundle with the flat connection $\nabla$ associated to the representation
$\rho : \pi_{1}(M) \rightarrow GL(n, {\Bbb C})$.
We assume the Assumption A made in Subsection 2.3.
We also assume that $\nabla$ satisfies the Assumption $\I$ and $\II$.
We choose an angle $\theta$ with $- \frac{\pi}{2} < \theta < 0$ such that $\theta$ is an Agmon angle for $\B_{\even, {\mathcal P}_{-, {\mathcal L}_{0}}}$
and $2 \theta$ for each $\B^{2}_{q, {\mathcal P}_{-, {\mathcal L}_{0}}}$.
Then we define the refined analytic torsion $T_{{\mathcal P}_{-, {\mathcal L}_{0}}}(g^{M}, \nabla)$ by
$$
T_{{\mathcal P}_{-, {\mathcal L}_{0}}}(g^{M}, \nabla) \hspace{0.1 cm} =
\hspace{0.1 cm} \Det_{\gr, \theta} \B_{\even, {\mathcal P}_{-, {\mathcal L}_{0}}} \cdot
e^{\frac{\pi i}{2} (\rank E) \cdot \eta_{{\B^{\trivial}_{\even, {\mathcal P}_{-, {\mathcal L}_{0}}}}}(0)}.
$$
Similarly,
we define $T_{{\mathcal P}_{+, {\mathcal L}_{1}}}(g^{M}, \nabla)$ by
$$
T_{{\mathcal P}_{+, {\mathcal L}_{1}}}(g^{M}, \nabla) \hspace{0.1 cm} =
\hspace{0.1 cm} \Det_{\gr, \theta} \B_{\even, {\mathcal P}_{+, {\mathcal L}_{1}}}
\cdot e^{\frac{\pi i}{2} (\rank E) \cdot \eta_{{\B^{\trivial}_{\even, {\mathcal P}_{+, {\mathcal L}_{1}}}}}(0)}.
$$
\end{definition}

\vspace{0.2 cm}

\begin{theorem}  \label{Theorem:4.11}
We assume the same assumptions as in Definition \ref{Definition:4.10}.
Then the refined analytic torsion $T_{{\mathcal P}_{-, {\mathcal L}_{0}}}(g^{M}, \nabla)$ and
$T_{{\mathcal P}_{+, {\mathcal L}_{1}}}(g^{M}, \nabla)$  are independent of the choice of Agmon angles and invariant
under the change of metrics among metrics which satisfy the Assumption $\I$ and $\II$ and are fixed
as a product metric on some collar neighborhood of $Y$.
\end{theorem}

\begin{proof}
The independence of the choice of the Agmon angles follows from the argument in the Subsection 3.10 in \cite{BK1}.
Let $\{ g^{M}_{v} \mid  - \delta < v < \delta \}$ be a family of metrics satisfying the Assumption $\I$ and $\II$ and $g^{M}_{v} = g^{M}_{0} (= g^{M})$
on some collar neighborhood of $Y$.
Corollary \ref{Corollary:4.5}, Lemma \ref{Lemma:4.6} and Corollary \ref{Corollary:4.9} imply that
$T_{{\mathcal P}_{-, {\mathcal L}_{0}}}(g^{M}_{v}, \nabla)$ and
$T_{{\mathcal P}_{+, {\mathcal L}_{1}}}(g^{M}_{v}, \nabla)$  are invariant up to modulo ${\Bbb Z}$.
Lemma \ref{Lemma:2.11} shows that the dimensions of $\Ker {\B^{\trivial}_{\even, {\mathcal P}_{-, {\mathcal L}_{0}}}}(v)$ and
$\Ker {\B^{\trivial}_{\even, {\mathcal P}_{+, {\mathcal L}_{1}}}}(v)$ are topological invariants. Hence
$\eta_{{\B^{\trivial}_{\even, {\mathcal P}_{-, {\mathcal L}_{0}}}}(g^{M}_{v})}(0)$ and
$\eta_{{\B^{\trivial}_{\even, {\mathcal P}_{+, {\mathcal L}_{1}}}}(g^{M}_{v})}(0)$ are continuous with respect to $v$ and have no integer jump.
Since $\Det_{\gr, \theta} \B_{\even, {\mathcal P}_{-, {\mathcal L}_{0}}}(v)$ and $\Det_{\gr, \theta} \B_{\even, {\mathcal P}_{+, {\mathcal L}_{1}}}(v)$
are continuous with respect to $v$, the result follows.
\end{proof}

\vspace{0.2 cm}

\noindent
{\it Remark} : Corollary \ref{Corollary:4.5} and Theorem \ref{Theorem:4.1111} below imply that both $T_{{\mathcal P}_{-, {\mathcal L}_{0}}}(g^{M}, \nabla)$ and
$T_{{\mathcal P}_{+, {\mathcal L}_{1}}}(g^{M}, \nabla)$ may depend on a metric on the boundary $Y$.

\vspace{0.2 cm}

\subsection{The case of an acyclic Hermitian connection}
Before finishing this section we include some results obtained in \cite{HL1} and \cite{HL2}.
Here we assume that $\nabla$ is a Hermitian connection {\it i.e.} $\nabla = \nabla^{\prime}$.
The following result was obtained in Theorem 3.11 in \cite{HL2} and Theorem 2.12 in \cite{HL1}.

\begin{theorem}  \label{Theorem:4.1111}
Let $(M, Y, g^{M})$ be a compact Riemannian manifold with boundary $Y$ and $g^{M}$ be a product metric near $Y$.
We assume that $\nabla$ is a Hermitian connection. Then :

\begin{eqnarray*}
& (1) & \sum_{q=0}^{m} (-1)^{q+1} q \cdot \left( \log \Det \B^{2}_{q, {\widetilde {\mathcal P}}_{0}} +
\log \Det \B^{2}_{q, {\widetilde {\mathcal P}}_{1}} \right) \hspace{0.1 cm} = \hspace{0.1 cm}
\sum_{q=0}^{m} (-1)^{q+1} q \cdot \left( \log \Det \B^{2}_{q, \rel} + \log \Det \B^{2}_{q, \Abs}  \right) .
\end{eqnarray*}
\noindent
Moreover, if $H^{q}(M ; E) = H^{q}(M, Y ; E) = \{ 0 \}$ for each $0 \leq q \leq m$, then :
\begin{eqnarray*}
& (2) &  \sum_{q=0}^{m} (-1)^{q+1} \cdot q \cdot \log \Det_{2 \theta} \B_{q, {\widetilde {\mathcal P}}_{0}}^{2}  \hspace{0.1 cm} = \hspace{0.1 cm}
\sum_{q=0}^{m} (-1)^{q+1} \cdot q \cdot \log \Det_{2 \theta} \B_{q, \rel}^{2}
\hspace{0.1 cm} + \hspace{0.1 cm} \frac{1}{4} \sum_{q = 0}^{m-1} \log \Det_{2 \theta} \B_{Y, q}^{2}  \\
& (3) &  \sum_{q=0}^{m} (-1)^{q+1} \cdot q \cdot \log \Det_{2 \theta} \B_{q, {\widetilde {\mathcal P}}_{1}}^{2}  \hspace{0.1 cm} = \hspace{0.1 cm}
\sum_{q=0}^{m} (-1)^{q+1} \cdot q \cdot \log \Det_{2 \theta} \B_{q, \rel}^{2}
\hspace{0.1 cm} - \hspace{0.1 cm} \frac{1}{4} \sum_{q = 0}^{m-1} \log \Det_{2 \theta} \B_{Y, q}^{2}
\end{eqnarray*}
\end{theorem}

\vspace{0.2 cm}

We denote by $\left(\Omega^{\even}(M, E)|_{Y} \right)^{\ast}$ the orthogonal complement of
$\left( \begin{array}{clcr} {\mathcal H}^{\even}(Y, E|_{Y}) \\ {\mathcal H}^{\odd}(Y, E|_{Y}) \end{array} \right)$ in
$\left( \Omega^{\even}(M, E)|_{Y} \right)$, {\it i.e.}

$$
\Omega^{\even}(M, E)|_{Y} \hspace{0.1 cm} = \hspace{0.1 cm}
\left( \Omega^{\even}(M, E)|_{Y} \right)^{\ast} \oplus
\left( \begin{array}{clcr} {\mathcal H}^{\even}(Y, E|_{Y}) \\ {\mathcal H}^{\odd}(Y, E|_{Y}) \end{array} \right),
$$
and denote by ${\mathcal P}_{\ast}$ the orthogonal projection onto $\left(\Omega^{\even}(M, E)|_{Y} \right)^{\ast}$.
We define one parameter families of orthogonal projections ${\widetilde {\frak P}}_{-}(\theta)$,
${\widetilde {\frak P}}_{+}(\theta) : \Omega^{\even}(M, E)|_{Y} \rightarrow \Omega^{\even}(M, E)|_{Y}$ by

\begin{eqnarray*}
{\widetilde {\frak P}}_{-}(\theta) & = & \Pi_{>} \cos \theta + {\mathcal P}_{-} \sin \theta + \frac{1}{2} ( 1 - \cos \theta - \sin \theta ) {\mathcal P}_{\ast} + {\mathcal P}_{{\mathcal L}_{0}} , \\
{\widetilde {\frak P}}_{+}(\theta) & = & \Pi_{>} \cos \theta + {\mathcal P}_{+} \sin \theta + \frac{1}{2} ( 1 - \cos \theta - \sin \theta ) {\mathcal P}_{\ast} + {\mathcal P}_{{\mathcal L}_{1}} ,
\qquad (0 \leq \theta \leq \frac{\pi}{2}),
\end{eqnarray*}

\vspace{0.2 cm}

\noindent
where $\Pi_{>}$ is an orthogonal projection onto the positive eigenspace of ${\mathcal A}$ (cf. (\ref{E:2.888})),
${\mathcal L}_{i}$ ($i = 0, 1$) are chosen by (\ref{E:2.46}), (\ref{E:2.47}) and (\ref{E:2.14}),
and ${\mathcal P}_{{\mathcal L}_{i}}$ are orthogonal projections onto ${\mathcal L}_{i}$.
${\widetilde {\frak P}}_{-}(\theta)$ (${\widetilde {\frak P}}_{+}(\theta)$)
is a smooth curve of orthogonal projections connecting ${\mathcal P}_{-, {\mathcal L}_{0}}$ (${\mathcal P}_{+, {\mathcal L}_{1}}$) and
$\Pi_{>, {\mathcal L}_{0}}$ ($\Pi_{>, {\mathcal L}_{1}}$).
We denote the Calder\'on projector for $\B$ by ${\mathcal C}_{M}$.
We also denote the spectral flow for $(\B_{{\widetilde {\frak P}}_{\pm}(\theta)})_{\theta \in [0, \frac{\pi}{2}]}$ and Maslov index for
$( {\widetilde {\frak P}}_{\pm}(\theta), \hspace{0.1 cm} {\mathcal C}_{M})_{\theta \in [0, \frac{\pi}{2}]} $
by $\SF (\B_{{\widetilde {\frak P}}_{\pm}(\theta)})_{\theta \in [0, \frac{\pi}{2}]}$ and
$ \Mas ( {\widetilde {\frak P}}_{\pm}(\theta), \hspace{0.1 cm} {\mathcal C}_{M})_{\theta \in [0, \frac{\pi}{2}]}$.
We refer to \cite{BW}, \cite{KL} and \cite{Ni} for the definitions of the Calder\'on projector, the spectral flow and Maslov index.
We obtained the following result in Theorem 3.12 in \cite{HL1}.

\begin{theorem}   \label{Theorem:4.1112}
Let $(M, Y, g^{M})$ be a compact Riemannian manifold with boundary $Y$ and $g^{M}$ be a product metric near $Y$.
We assume that $\nabla$ is a Hermitian connection. Then :

\vspace{0.2 cm}

\noindent
(1) $\hspace{0.2 cm} \eta( \B_{{\mathcal P}_{-, {\mathcal L}_{0}}} ) - \eta ( \B_{\Pi_{>, {\mathcal L}_{0}}} )
\hspace{0.1 cm} = \hspace{0.1 cm} \SF (\B_{{\widetilde {\frak P}}_{-}(\theta)})_{\theta \in [0, \frac{\pi}{2}]}
\hspace{0.1 cm} = \hspace{0.1 cm}  \Mas ( {\widetilde {\frak P}}_{-}(\theta), \hspace{0.1 cm} {\mathcal C}_{M})_{\theta \in [0, \frac{\pi}{2}]} $.  \newline
(2) $\hspace{0.2 cm} \eta( \B_{{\mathcal P}_{+, {\mathcal L}_{1}}} ) - \eta ( \B_{\Pi_{>, {\mathcal L}_{1}}} )
\hspace{0.1 cm} = \hspace{0.1 cm}  \SF (\B_{{\widetilde {\frak P}}_{+}(\theta)})_{\theta \in [0, \frac{\pi}{2}]}
\hspace{0.1 cm} = \hspace{0.1 cm} \Mas ( {\widetilde {\frak P}}_{+}(\theta), \hspace{0.1 cm} {\mathcal C}_{M})_{\theta \in [0, \frac{\pi}{2}]} $. \newline
\end{theorem}

\vspace{0.2 cm}

\noindent
Using Theorem \ref{Theorem:4.1111} and Theorem \ref{Theorem:4.1112} together with results in \cite{BFK}, \cite{Lu}, \cite{BL}, \cite{KL},
we obtained the gluing formula of the refined analytic torsion as follows (Theorem 4.3 in \cite{HL1}).

\begin{theorem}   \label{Theorem:4.1113}
Let $({\widetilde M}, g^{\widetilde M})$ be a closed Riemannian manifold of dimension $m = 2r -1$ and $Y$ be a hypersurface so that
${\widetilde M} = M_{1} \cup_{Y} M_{2}$. We assume that $g^{\widetilde M}$ is a product metric near $Y$ and
for each $0 \leq q \leq m$, $i = 1, 2$, $H^{q}({\widetilde M}, {\widetilde E}) = H^{q}(M_{i}, Y; E_{i}) = H^{q}(M_{i};E_{i}) = 0$,
where $E_{i} = {\widetilde E}|_{M_{i}}$. Then for a Hermitian connection $\nabla$, we have
\begin{eqnarray*}
T_{{\widetilde M}}(g^{{\widetilde M}}, \nabla)  & = & T_{M_{1}, {\mathcal P}_{+}}(g^{M_{1}}, \nabla) \cdot
T_{M_{2}, {\mathcal P}_{-}}(g^{M_{2}}, \nabla).
\end{eqnarray*}
\end{theorem}

\vspace{0.3 cm}

\section{The refined analytic torsion as an element of a determinant line}

\vspace{0.2 cm}

In this section we are going to define the refined analytic torsion
without assuming the Assumption $\I$ and $\II$ in Subsection 4.1.
For this purpose we begin with the cochain complex (4.2).
We can define the refined analytic torsion for the complex (4.3) in the same way.

\vspace{0.3 cm}

\subsection{Decomposition of the complex and the graded determinant}
For $\lambda \geq 0$ we define a spectral projection
$\Pi_{\B^{2}_{{\widetilde {\mathcal P}}_{0}}, [0, \lambda]} : \Omega^{\bullet, \infty}_{{\widetilde {\mathcal {\mathcal P}}}_{0}}(M, E)
\rightarrow \Omega^{\bullet, \infty}_{{\widetilde {\mathcal {\mathcal P}}}_{0}}(M, E)$
onto the generalized eigenspace of $\B^{2}_{{\widetilde {\mathcal P}}_{0}}$ corresponding to generalized eigenvalues in
$\{ \, z \in \mathbb{C} : |z| \le \lambda \, \}$ by

$$
\Pi_{\B^{2}_{{\widetilde {\mathcal P}}_{0}}, [0, \lambda]} \hspace{0.1 cm}  = \hspace{0.1 cm}
\frac{i}{2\pi}\int_{|z|= \lambda + \epsilon}\big(\, \B^{2}_{{\widetilde {\mathcal P}}_{0}} - z \, \big)^{-1} \hspace{0.1 cm} d z,
$$

\noindent
where $\epsilon > 0$ is small enough so that there are no
generalized eigenvalues of $\B^{2}_{{\widetilde {\mathcal P}}_{0}}$ whose absolute values are in the interval $(\lambda, \lambda + \epsilon]$.
We note that $\Pi_{\B^{2}_{{\widetilde {\mathcal P}}_{0}}, [0, \lambda]}$ commutes with $\B^{2}_{{\widetilde {\mathcal P}}_{0}}$
and $\Imm \Pi_{\B^{2}_{{\widetilde {\mathcal P}}_{0}}, [0, \lambda]}$ is a finite dimensional vector space.
Putting

$$
\Omega^{\bullet, \infty}_{{\widetilde {\mathcal P}}_{0}, [0, \lambda]}(M, E) \, := \,
\Pi_{\B^{2}_{{\widetilde {\mathcal P}}_{0}}, [0, \lambda]} \big(\, \Omega^{\bullet, \infty}_{{\widetilde {\mathcal {\mathcal P}}}_{0}}(M, E) \, \big),
\qquad
\Omega^{\bullet, \infty}_{{\widetilde {\mathcal P}}_{0}, (\lambda, \infty)}(M, E) \, := \,
\big( \, \Id - \Pi_{\B^{2}_{{\widetilde P}_{0}}, [0, \lambda]} \, \big) \big(\, \Omega^{\bullet, \infty}_{{\widetilde {\mathcal P}}_{0}}(M, E) \, \big),
$$

\vspace{0.2 cm}

\noindent
we have

\begin{equation}  \label{E:6.1}
\Omega^{\bullet, \infty}_{{\widetilde {\mathcal P}}_{0}}(M, E) \, = \, \Omega^{\bullet, \infty}_{{\widetilde {\mathcal P}}_{0}, [0, \lambda]}(M, E) \, \oplus \,
\Omega^{\bullet, \infty}_{{\widetilde {\mathcal P}}_{0}, (\lambda, \infty)}(M, E).
\end{equation}

\vspace{0.2 cm}

\noindent
Let $\mathcal{I}$ be either $[0, \lambda]$ or $(\lambda, \infty)$.
Since $\nabla$ commutes with $\B^{2}_{\widetilde{\mathcal{P}}_{0}}$, $\nabla$ commutes with $\Pi_{\B^{2}_{{\widetilde P}_{0}}, [0, \lambda]}$
and hence ($\Omega^{\bullet, \infty}_{\widetilde{\mathcal{P}}_{0},\mathcal{I}}(M,E), \nabla)$ is a subcomplex of
$(\Omega^{{\bullet, \infty}}_{\widetilde{\mathcal{P}}_{0}}(M,E), \nabla)$.
Similarly, $\Gamma$ acts on ($\Omega^{\bullet, \infty}_{\widetilde{\mathcal{P}}_{0},\mathcal{I}}(M,E), \nabla)$.
We denote by $H_{\widetilde{\mathcal{P}}_{0}}^\bullet(M,E)$ and $H_{\widetilde{\mathcal{P}}_{0},[0,\lambda]}^\bullet(M,E)$
the cohomologies of the complexes $\big( \, \Omega^{\bullet, \infty}_{\widetilde{\mathcal{P}}_{0}}(M,E), {\nabla} \, \big)$ and
$\big( \, \Omega^{\bullet, \infty}_{\widetilde{\mathcal{P}}_{0}, [0,\lambda]}(M,E), \nabla_{[0,\lambda]} \, \big) $, respectively,
where $\nabla_{[0,\lambda]}:= \nabla |_{\Omega^\bullet_{\widetilde{\mathcal{P}}_{0},[0,\lambda]}(M,E)}$.
Then, we have the following lemma.
Since the proof is similar to the proof of \cite[Proposition 3]{Mo} or \cite[Proposition 2.2]{Su2}, we skip the details.

\begin{lemma} \label{Lemma:6.1}
For each $\lambda \ge 0$, the complex $(\Omega^{\bullet, \infty}_{\widetilde{\mathcal{P}}_{0},(\lambda,\infty)}(M,E), \nabla)$ is acyclic and
\begin{eqnarray*}
H_{\widetilde{\mathcal{P}}_{0}}^{\bullet}(M,E) & \cong & H_{\widetilde{\mathcal{P}}_{0}, [0,\lambda]}^{\bullet}(M,E).
\end{eqnarray*}
In particular, the inclusion $(\Omega^{\bullet, \infty}_{\widetilde{\mathcal{P}}_{0},[0, \lambda]}(M,E), \nabla) \hookrightarrow (\Omega^{\bullet, \infty}_{\widetilde{\mathcal{P}}_{0}}(M,E), \nabla)$ induces an isomorphism in cohomology and
$$
H^{q}_{\widetilde{\mathcal{P}}_{0}}(M, E) \hspace{0.1 cm} = \hspace{0.1 cm}
\begin{cases} H^{q}(M, Y ; \rho)  \quad  \text{if} \quad q \quad \text{is} \quad \text{even} \\
H^{q}(M ; \rho)  \quad  \text{if} \quad q \quad \text{is} \quad \text{odd}.
\end{cases}
$$
\end{lemma}

\vspace{0.3 cm}

\noindent
Setting

\begin{eqnarray*}
\Omega^{q, -, \infty}_{{\widetilde {\mathcal P}}_{0}, (\lambda, \infty)}(M, E)  :=
\ker \nabla \cap \Omega^{q, \infty}_{{\widetilde {\mathcal P}}_{0}, (\lambda, \infty)}(M, E),
\qquad
\Omega^{q, +, \infty}_{{\widetilde {\mathcal P}}_{0}, (\lambda, \infty)}(M, E)  :=
\ker (\Gamma \nabla \Gamma) \cap \Omega^{q, \infty}_{{\widetilde {\mathcal P}}_{0}, (\lambda, \infty)}(M, E), \quad
\end{eqnarray*}

\noindent
we have

\begin{equation}  \label{E:6.2}
\Omega_{\widetilde{\mathcal{P}}_{0}, (\lambda, \infty)}^{q, \infty}(M,E) \, = \,
\Omega_{\widetilde{\mathcal{P}}_{0}, (\lambda, \infty)}^{q,-, \infty}(M,E) \oplus
\Omega_{\widetilde{\mathcal{P}}_{0}, (\lambda, \infty)}^{q,+, \infty}(M,E).
\end{equation}

\vspace{0.2 cm}

\noindent
We denote by $\B_{q,\widetilde{\mathcal{P}}_{0}}^{2, \mathcal{I}}$, $\B_{\even, \mathcal{P}_{-, {\mathcal L}_{0}}}^\mathcal{I}$ and
$\B_{\even, \mathcal{P}_{-, {\mathcal L}_{0}}}^{\pm,(\lambda,\infty)}$
the restriction of $\B^{2}$, $\B_{\even}$ and $\B_{\even, \mathcal{P}_{-, {\mathcal L}_{0}}}$ to $\Omega^{q, \infty}_{\widetilde{\mathcal{P}}_{0}, \mathcal{I}}(M,E)$,
$\Omega^{\even, \infty}_{\mathcal{P}_{-, {\mathcal L}_{0}}, \mathcal{I}}(M,E)$ and $\Omega_{\mathcal{P}_{-, {\mathcal L}_{0}}, (\lambda,\infty)}^{\even,\pm, \infty}(M,E)$.
Then we have the following direct sum decomposition

\begin{equation} \label{E:6.3}
\B_{\even, \mathcal{P}_{-, {\mathcal L}_{0}}}^{(\lambda,\infty)} \, = \,
\B_{\even, \mathcal{P}_{-, {\mathcal L}_{0}}}^{-,(\lambda,\infty)} \, \oplus \, \B_{\even, \mathcal{P}_{-, {\mathcal L}_{0}}}^{+,(\lambda,\infty)}.
\end{equation}

\begin{definition} \label{Definition:6.2}
Let $\theta \in (-\frac{\pi}{2}, 0)$ be an Agmon angle for
$\B_{\even, \mathcal{P}_{-, {\mathcal L}_{0}}}^{(\lambda,\infty)}$. Then
the graded determinant associated to
$\B_{\even,\widetilde{\mathcal{P}}_{0}}^{(\lambda,\infty)}$ and $\theta$ is defined by
\begin{eqnarray} \label{E:6.4}
\Det_{\gr,\theta}(\B_{\even, \mathcal{P}_{-, {\mathcal L}_{0}}}^{(\lambda,\infty)}) & := &
\frac{\Det_\theta(\B_{\even, \mathcal{P}_{-, {\mathcal L}_{0}}}^{+,(\lambda,\infty)})}
{\Det_\theta(-\B_{\even, \mathcal{P}_{-, {\mathcal L}_{0}}}^{-,(\lambda,\infty)})} \hspace{0.1 cm}.
\end{eqnarray}
We define the graded determinant $\Det_{\gr,\theta}(\B_{\even, \mathcal{P}_{+, {\mathcal L}_{1}}}^{(\lambda,\infty)})$ in the same way.
\end{definition}

\vspace{0.3 cm}

\subsection{The canonical element of the determinant line}

In this subsection we briefly review the refined torsion for a finite complex. We refer to \cite{BK2} for more details.
Let $(C^{\bullet}, \partial, \Gamma)$ be a finite complex consisting of finite dimensional vector spaces as follows.

$$
(C^{\bullet}, \hspace{0.1 cm} \partial, \hspace{0.1 cm} \Gamma) \hspace{0.1 cm} : \hspace{0.1 cm}
 0 \longrightarrow C^{0} \stackrel{\partial}{\longrightarrow}
C^{1} \stackrel{\partial}{\longrightarrow} \cdots \stackrel{\partial}{\longrightarrow} C^{m-1}
\stackrel{\partial}{\longrightarrow} C^{m} \longrightarrow 0.
$$

\vspace{0.2 cm}

\noindent
Here $\Gamma : C^{\bullet} \rightarrow C^{\bullet}$ is an involution satisfying
$\Gamma (C^{q}) = C^{m-q}$ for $0 \leq q \leq m$. We call $\Gamma$ a chirality operator.
We define the determinant line of $(C^{\bullet}, \partial, \Gamma)$ by

\begin{equation} \label{E:5.11}
\Det(C^\bullet) \, = \, \bigotimes_{q=0}^m \Det (C^q)^{(-1)^q}, \quad m=2r-1,
\end{equation}

\noindent
where $\Det (C^q)$ is the top exterior power of $C^q$ and
$\Det(C^\bullet)^{(-1)} \, := \, \operatorname{Hom}\big(\, \Det(C^\bullet), \mathbb{C}\, \big)$.
We define the canonical element $c_{\Gamma}$ of $\Det(C^\bullet)$ as follows.
We first extend the chirality operator $\Gamma \, : \, C^{\bullet} \to C^{\bullet}$ to determinant lines
and choose arbitrary non-zero elements $c_{q} \in \Det (C^q)$. Then we define $c_{\Gamma}$ by

\begin{eqnarray} \label{E:6.5}
c_\Gamma & := & (-1)^{{\mathcal{R}} (C^{\bullet})} \cdot c_0 \otimes c_1^{-1}\otimes \cdots \otimes c_{r-1}^{(-1)^{r-1}} \otimes \nonumber  \\
& & \hspace{1.0 cm}  (\Gamma c_{r-1})^{(-1)^{r}} \otimes (\Gamma c_{r-2})^{(-1)^{r-1}} \otimes \cdots
\cdots \otimes (\Gamma c_1) \otimes (\Gamma c_0)^{(-1)} \in \Det(C^\bullet),
\end{eqnarray}

\noindent
where $(-1)^{{\mathcal{R}} (C^{\bullet})}$ is a normalization factor defined by (cf. (4-2) in \cite{BK2})

$$
{\mathcal{R}} (C^{\bullet}) = \frac{1}{2} \sum_{q=0}^{r-1} \dim C^{q} \cdot \left( \dim C^{q} + (-1)^{r + q} \right),
$$

\noindent
and $c^{-1} \in \Det(C^q)^{-1}$ is the unique dual element of $c \in \Det(C^q)$ such that $c^{-1}(c)=1$.
We denote by $H^\bullet(\partial)$ the cohomology of the complex $(C^\bullet,\partial)$.
Then there is a standard isomorphism

$$
\phi_{C^\bullet} \, : \, \Det(C^\bullet) \to \Det(H^\bullet(\partial)).
$$

\noindent
We refer to Subsection 2.4 in \cite{BK2} for the definition of the isomorphism $\phi_{C^\bullet}$.
The refined torsion of the pair $(C^\bullet, \Gamma)$ is the element $\rho_{\Gamma}$ in $\Det(H^\bullet(\partial))$
defined by

\begin{equation}  \label{E:6.6}
\rho_{\Gamma} \, = \, \rho_{C^\bullet, \Gamma} \, := \, \phi_{C^\bullet}(c_\Gamma) \in \Det(H^\bullet(\partial)).
\end{equation}

\vspace{0.2 cm}

Denote by $\rho_{\widetilde{\mathcal{P}}_{0}, [0,\lambda]}$ the refined torsion of
the complex $\big( \, \Omega^\bullet_{\widetilde{\mathcal{P}}_{0}, [0,\lambda]}(M,E),\nabla_{[0,\lambda]}\, \big)$.
Then we have the following lemma (cf. Proposition 7.8 in \cite{BK2} and Proposition 3.18 in \cite{Ve1}).

\vspace{0.2 cm}

\begin{definition} \label{Definition:6.3}
Let $\theta \in (-\frac{\pi}{2}, 0)$ be an Agmon angle for the operator
${\B}_{\even, \widetilde{\mathcal{P}}_{0}}^{(\lambda,\infty)}$.
We define

\begin{equation} \label{E:6.7}
\rho(\widetilde{\mathcal{P}}_{0}, \nabla, g^M) \, := \,
\Det_{\gr, \theta}({\B}_{\even, \widetilde{\mathcal{P}}_{0}}^{(\lambda,\infty)}) \cdot
\rho_{\widetilde{\mathcal{P}}_{0}, [0,\lambda]} \in \Det \big(\, H_{{\widetilde {\mathcal{P}}_{0}}}^\bullet(M,E)\, \big).
\end{equation}
\end{definition}

\begin{lemma}  \label{Lemma:6.3}
$\rho(\widetilde{\mathcal{P}}_{0}, \nabla, g^M)$
is independent of the choice of $\lambda \ge 0$ and choice of Agmon angle $\theta \in (-\frac{\pi}{2},0)$
for the operator ${\B}_{\widetilde{\mathcal{P}}_{0}}^{(\lambda,\infty)}$.
\end{lemma}

\begin{proof}
For $0 \leq \lambda < \mu$ we have
$$
\Det_{\gr, \theta}({\B}_{\even, \widetilde{\mathcal{P}}_{0}}^{(\lambda,\infty)}) =
\Det_{\gr, \theta}({\B}_{\even, \widetilde{\mathcal{P}}_{0}}^{(\lambda,\mu]}) \cdot
\Det_{\gr, \theta}({\B}_{\even, \widetilde{\mathcal{P}}_{0}}^{(\mu,\infty)}).
$$
Further, Proposition 5.10 in \cite{BK2} shows that
$$
\rho_{\widetilde{\mathcal{P}}_{0}, [0,\mu]} = \Det_{\gr, \theta}({\B}_{\even, \widetilde{\mathcal{P}}_{0}}^{(\lambda,\mu]}) \cdot
\rho_{\widetilde{\mathcal{P}}_{0}, [0,\lambda]}.
$$
These two equalities show that (\ref{E:6.7}) is independent of the choice of $\lambda \geq 0$.
If $\theta$, $\theta^{\prime} \in (-\frac{\pi}{2}, 0)$ are both Agmon angles for ${\B}_{\even, \widetilde{\mathcal{P}}_{0}}^{(\lambda,\infty)}$,
Theorem \ref{Theorem:2.22} shows that there are only finitely many eigenvalues of ${\B}_{\even, \widetilde{\mathcal{P}}_{0}}^{(\lambda,\infty)}$ in the solid angle between  $\theta$ and $\theta^{\prime}$ and hence the independence of the Agmon angles follows from the argument in the Subsection 3.10 in \cite{BK1}.
\end{proof}

\vspace{0.2 cm}

\noindent
{\it Remark} :
Since different Hermitian metrics give rise to equivalent $L^{2}$-norms over compact manifolds,
the domain of $\B_{\widetilde{\mathcal{P}}_{0}}$ is independent of the choice of the Hermitian metric $h^E$.
Moreover, since the definition of $\B_{\widetilde{\mathcal{P}}_{0}}$ does not use the Hermitian structure,
$\rho(\widetilde{\mathcal{P}}_{0}, \nabla, g^M)$ is independent of
the choice of the Hermitian metric $h^E$ (cf. p.2004 in \cite{Ve1}).

\vspace{0.2 cm}

We recall the Definition \ref{Definition:3.51} for the eta invariants $\eta(\B^{[0, \lambda]}_{\even, {\widetilde {\mathcal P}}_{0}})$,
$\eta(\B^{(\lambda, \infty)}_{\even, {\widetilde {\mathcal P}}_{0}})$.
Then
$$
\eta(\B_{\even, {\widetilde {\mathcal P}}_{0}}) \, = \, \eta(\B^{[0, \lambda]}_{\even, {\widetilde {\mathcal P}}_{0}})
\, + \,  \eta(\B^{(\lambda, \infty)}_{\even, {\widetilde {\mathcal P}}_{0}}).
$$
We next define $\xi_{\lambda,\widetilde{\mathcal{P}}_{0}}(\nabla,g^M)$ and $L_{q,\lambda,\widetilde{\mathcal{P}}_{0}}$ by
\begin{eqnarray*}
\xi_{\lambda, \widetilde{\mathcal{P}}_{0}}(\nabla,g^M) & := &
- \frac{1}{2} \sum_{q=0}^{m} (-1)^{q+1} \cdot q \cdot \frac{d}{ds}\big|_{s=0}
\zeta_{\B^{2,(\lambda,\infty)}_{q, \widetilde{\mathcal{P}}_{0}}}(s),  \\
L_{q,\lambda,\widetilde{\mathcal{P}}_{0}} & := & \dim \Omega^q_{\widetilde{\mathcal{P}}_{0}, (0,\lambda]}(M,E).
\end{eqnarray*}
The following lemma is analogous to Corollary \ref{Corollary:4.5}.

\vspace{0.2 cm}

\begin{lemma} \label{Lemma:6.4}
Recall that $~ {\frak L}^{0, q}_{{\widetilde {\mathcal P}}_{0}} ~$ the dimension of $0$-generalized eigenspace of
$\B^{2}_{q, {\widetilde {\mathcal P}}_{0}}$. Then :
\begin{eqnarray*}
 \rho(\widetilde{\mathcal{P}}_{0}, \nabla, g^M)
& = &  \rho_{\widetilde{\mathcal{P}}_{0}, [0,\lambda]} \cdot e^{\xi_{\lambda, \widetilde{\mathcal{P}}_{0}}(\nabla,g^M)}
\cdot e^{- i \pi \eta( \B^{(\lambda, \infty)}_{\even,\widetilde{\mathcal{P}}_{0}})} \cdot \\
& & e^{\frac{\pi i}{2}\big(\, -\sum_{q=0}^m(-1)^{q+1} \cdot q \cdot L_{q, \lambda, \widetilde{\mathcal{P}}_{0}}
\,-\, \sum_{q=0}^m(-1)^{q+1} \cdot {q} \cdot {\frak L}^{0, q}_{{\widetilde {\mathcal P}}_{0}}
\,+\, \frac{1}{4}\sum_{q=0}^{m-1}\zeta_{\B^2_{Y,q}}(0) \, + \, \sum_{q=0}^{r-2}(r-1-q) (l_{q}^{+} - l_{q}^{-}) \,\big)}.
\end{eqnarray*}
\end{lemma}

\vspace{0.2 cm}
\begin{proof}
It's enough to compute $\log \Det_{\gr, \theta}({\B}_{\even, \widetilde{\mathcal{P}}_{0}}^{(\lambda,\infty)})$.
The computation similar to (\ref{E:4.12}) shows that

\begin{eqnarray*}
 \log \Det_{\gr, \theta} (\B^{(\lambda, \infty)}_{\even, {\mathcal P}_{-, {\mathcal L}_{0}}}) \hspace{0.1 cm}
& = &  \log \Det_{\theta} (\B^{+, (\lambda, \infty)}_{\even, {\mathcal P}_{-, {\mathcal L}_{0}}}) -
\log \Det_{\theta}(- \B^{-, (\lambda, \infty)}_{\even, {\mathcal P}_{-, {\mathcal L}_{0}}}) \\
& = & \frac{1}{2} \sum_{q=0}^{m} (-1)^{q+1} \cdot q \cdot \log \Det_{2\theta} \B^{2, (\lambda, \infty)}_{q, {\widetilde {\mathcal P}_{0}}}
\hspace{0.1 cm} + \hspace{0.1 cm}
\frac{\pi i }{2} \sum_{q=0}^{m} (-1)^{q+1} \cdot q \cdot \zeta_{\B^{2, (\lambda, \infty)}_{q, {\widetilde {\mathcal P}_{0}}}}(0)  \\
&  & - \frac{\pi i}{2} \left(\eta_{\B^{(\lambda, \infty)}_{\even, {\mathcal P}_{-, {\mathcal L}_{0}}}}(0)
+ {\frak L}^{+, (\lambda, \infty)}_{{\mathcal P}_{-, {\mathcal L}_{0}}} -
{\frak L}^{-, (\lambda, \infty)}_{{\mathcal P}_{-, {\mathcal L}_{0}}} \right) \\
& = & \xi_{\lambda, \widetilde{\mathcal{P}}_{0}}(\nabla,g^M)
- i \pi \eta( \B^{(\lambda, \infty)}_{\even,\widetilde{\mathcal{P}}_{0}})
+ \frac{\pi i }{2} \sum_{q=0}^{m} (-1)^{q+1} \cdot q \cdot
\zeta_{\B^{2, (\lambda, \infty)}_{q, {\widetilde {\mathcal P}_{0}}}}(0).  \\
\end{eqnarray*}

\noindent
Using the fact
$\zeta_{\B^{2, (\lambda, \infty)}_{q, {\widetilde {\mathcal P}_{0}}}}(0) =
\zeta_{\B^{2}_{q, {\widetilde {\mathcal P}}_{0}}}(0) - L_{q,\lambda,\widetilde{\mathcal{P}}_{0}}$
and Lemma \ref{Lemma:3.2},
the result follows.
\end{proof}

\vspace{0.2 cm}

\noindent
Finally, we define the refined analytic torsion as an element of a determinant line as follows.

\vspace{0.2 cm}

\begin{definition} \label{Definition:6.9}
We assume the same assumptions as in Definition \ref{Definition:4.10} except the Assumption $\I$ and $\II$ in Subsection 4.1.
We define the refined analytic torsions
$\rho_{an,\widetilde{\mathcal{P}}_{0}}(g^{M}, \nabla)$ and
$\rho_{an,\widetilde{\mathcal{P}}_{1}}(g^{M}, \nabla)$ by
\begin{eqnarray*}
\rho_{an, \widetilde{\mathcal{P}}_{0}}(g^{M}, \nabla) & := & \rho(\widetilde{\mathcal{P}}_{0}, \nabla, g^{M}) \cdot
exp \left\{\frac{i \pi}{2} (\operatorname{rank}E) \hspace{0.05 cm} \eta_{\, \B^{\trivial}_{\even, {\mathcal P}_{-, {\mathcal L}_{0}}}(g^{M})}(0) \right\}
\hspace{0.1 cm} \in \hspace{0.1 cm} \Det \left( H_{\widetilde{\mathcal{P}}_{0}}^{\bullet}(M,E) \right),   \\
\rho_{an, \widetilde{\mathcal{P}}_{1}}(g^{M}, \nabla) & := & \rho(\widetilde{\mathcal{P}}_{1}, \nabla, g^{M}) \cdot
exp \left\{\frac{i \pi}{2} (\operatorname{rank}E) \hspace{0.05 cm} \eta_{\, \B^{\trivial}_{\even, {\mathcal P}_{+, {\mathcal L}_{1}}}(g^{M})}(0) \right\}
\hspace{0.1 cm} \in \hspace{0.1 cm} \Det \left( H_{\widetilde{\mathcal{P}}_{1}}^{\bullet}(M,E) \right).
\end{eqnarray*}
\end{definition}

\vspace{0.2 cm}

\subsection{The metric dependency of the refined analytic torsion}

In this subsection we discuss the metric dependency of $\rho(\widetilde{\mathcal{P}}_{0}, \nabla, g^{M})$.
We suppose that $\{g^{M}_{v} \mid  - \delta < v < \delta \}$, $\delta > 0$, is a family of Riemannian metrics on $M$ such that each $g_{v}^{M}$
is a product one and does not vary on $[0, \epsilon ) \times Y$ for some fixed $\epsilon > 0$.
We denote by $\Gamma_{v}$ and $\B_{{\widetilde {\mathcal P}}_{0}}(v) \, = \, \B_{\widetilde{\mathcal{P}}_{0}} (\nabla, g^{M}_{v})$ the chirality operator and
the odd signature operator corresponding to $g^{M}_{v}$.
We define $\B_{\even, \widetilde{\mathcal{P}}_{0}}(v)$, $\B^{[0,\lambda]}_{\even, \widetilde{\mathcal{P}}_{0}}(v)$ and
$\B^{(\lambda,\infty)}_{\even, \widetilde{\mathcal{P}}_{0}}(v)$ in similar ways.
Then we get the associated refined torsion $\rho_{\widetilde{\mathcal{P}},[0,\lambda]} (v)$
of the complex $(\Omega^{\bullet, \infty}_{\widetilde{\mathcal{P}}_{0}, [0,\lambda], g_{v}^{M}}(M,E), \nabla_{[0,\lambda]})$
(cf. (\ref{E:6.6})) and we denote by $\rho(v) \, = \,\rho(\widetilde{\mathcal{P}}_{0}, \nabla, g_{v}^{M})$
the canonical element defined in (\ref{E:6.7}).
Then we have the following lemma, whose proof is a verbatim repetition of Lemma 9.2 in \cite{BK2} (cf. Proposition 4.5 in \cite{Ve1}).

\vspace{0.2 cm}

\begin{lemma} \label{Lemma:6.5}
We choose $\lambda$ and ${\widetilde \delta_{0}}$  ($\hspace{0.1 cm} 0 < {\widetilde \delta_{0}} \leq \delta \hspace{0.1 cm}$)
so that $\B^{2}_{{\widetilde {\mathcal P}}_{0}}(v)$ does not have any
generalized eigenvalues whose absolute values are equal to $\lambda$ for $- {\widetilde \delta_{0}} < v < {\widetilde \delta_{0}}$.
Then the product
$$
e^{\xi_{\lambda, \widetilde{\mathcal{P}}_{0}}(\nabla, \, g^{M}_{v})} \cdot \rho_{\widetilde{\mathcal{P}}_{0}, [0,\lambda]} (v)
\in \Det \big(\, H_{\widetilde{\mathcal{P}}_{0}}^{\bullet}(M,E)\, \big)
$$
is independent of $v  \in (- {\widetilde \delta_{0}},  {\widetilde \delta_{0}})$.
\end{lemma}

\vspace{0.2 cm}

Let $\B^{\trivial}_{\even, {\mathcal P}_{-, {\mathcal L}_{0}}}(v)$ be
the odd signature operator induced from $\nabla^{\trivial}$ and the Riemannian metric $g^{M}_{v}$ (cf. (\ref{E:4.44})).
Then we have the following result.

\vspace{0.2 cm}

\begin{lemma}   \label{Lemma:6.7}
For large enough $\lambda \in {\Bbb R}$,
\begin{equation}  \label{E:5.15}
\eta(\B^{(\lambda,\infty)}_{\even, {\mathcal P}_{-, {\mathcal L}_{0}}}(v)) - \frac{1}{2} \operatorname{rank}(E) \hspace{0.1 cm}
\eta_{\B^{\trivial}_{\even, {\mathcal P}_{-, {\mathcal L}_{0}}}(v)}(0)
\end{equation}
is independent of $v \in (-\delta,  \delta)$.
\end{lemma}

\begin{proof}
From the definition of
$\eta(\B^{[0,\lambda]}_{\even, {\mathcal P}_{-, {\mathcal L}_{0}}} (v))$ and
$\eta(\B^{[0,\lambda], \trivial}_{\even, {\mathcal P}_{-, {\mathcal L}_{0}}} (v))$ (cf. Definition \ref{Definition:3.51}),
it is straightforward that
$\eta(\B^{[0,\lambda]}_{\even, {\mathcal P}_{-, {\mathcal L}_{0}}} (v))$ and
$\eta(\B^{[0,\lambda], \trivial}_{\even, {\mathcal P}_{-, {\mathcal L}_{0}}} (v))$ are,
on modulo ${\Bbb Z}$, independent of $v$.
Since
\begin{eqnarray*}
\eta(\B_{\even, {\mathcal P}_{-, {\mathcal L}_{0}}} (v)) & = & \eta(\B^{[0,\lambda]}_{\even, {\mathcal P}_{-, {\mathcal L}_{0}}} (v)) +
\eta(\B^{(\lambda, \infty)}_{\even, {\mathcal P}_{-, {\mathcal L}_{0}}} (v)),  \\
\eta(\B^{\trivial}_{\even, {\mathcal P}_{-, {\mathcal L}_{0}}} (v)) & = & \eta(\B^{[0,\lambda], \trivial}_{\even, {\mathcal P}_{-, {\mathcal L}_{0}}} (v)) +
\eta(\B^{(\lambda, \infty), \trivial}_{\even, {\mathcal P}_{-, {\mathcal L}_{0}}} (v)),
\end{eqnarray*}

\noindent
Corollary \ref{Corollary:4.9} implies that (\ref{E:5.15}) is independent of $v \in (-\delta, \delta)$ modulo ${\Bbb Z}$.
Theorem \ref{Theorem:2.22} shows that $\B^{(\lambda,\infty)}$ has no zero eigenvalue nor pure imaginary eigenvalues for large enough $\lambda \in {\Bbb R}$
and Lemma \ref{Lemma:2.11} shows that $\Dim \Ker \B^{\trivial}_{\even, {\mathcal P}_{-, {\mathcal L}_{0}}}(v)$ is a topological invariant and hence
independent of $v$. These facts show that both
$\eta(\B^{(\lambda,\infty)}_{\even, {\mathcal P}_{-, {\mathcal L}_{0}}}(v))$ and $\eta_{\B^{\trivial}_{\even, {\mathcal P}_{-, {\mathcal L}_{0}}}(v)}(0)$
are continuous with respect to $v$ and have no integer jump, which completes the proof of the lemma.
\end{proof}

\noindent
Summarizing the above argument, we obtain the following result.

\begin{theorem}  \label{Theorem:6.8}
We assume the same assumptions as in Definition \ref{Definition:6.9}. Then the refined analytic torsions
$\rho_{an, \widetilde{\mathcal{P}}_{0}}(g^{M}, \nabla)$ and $\rho_{an, \widetilde{\mathcal{P}}_{1}}(g^{M}, \nabla)$
are independent of the choice of the Agmon angles and invariant under the change of metrics in the interior of $M$.
\end{theorem}

\vspace{0.3 cm}

\subsection{Ray-Singer norm on the determinant line of cohomologies in case of Hermitian connections}

In this subsection we discuss briefly the Ray-Singer norm on the determinant line of the cohomologies
only when $\nabla$ is a Hermitian connection.
In this case $\B^{2}_{\widetilde{\mathcal{P}}_{0}}$ is a non-negative self-adjoint operator
and hence we take $- \pi$ as an Agmon angle.
For $\lambda \geq 0$ the cohomologies of
$\Omega^{\bullet, \infty}_{{\widetilde {\mathcal P}}_{0}, [0, \lambda]}(M, E)$ is naturally isomorphic to
$H^{\bullet}_{\widetilde{\mathcal{P}}_{0}}(M, E)$ (Lemma \ref{Lemma:6.1}).
We define the determinant lines $\Det \left( \Omega^{\bullet, \infty}_{{\widetilde {\mathcal P}}_{0}, [0, \lambda]}(M, E) \right)$ and
$\Det \left( H^{\bullet}_{\widetilde{\mathcal{P}}_{0}}(M, E) \right)$ as in (\ref{E:5.11}).
Then
$\Det \left( \Omega^{\bullet, \infty}_{{\widetilde {\mathcal P}}_{0}, [0, \lambda]}(M, E) \right)$ is naturally isomorphic to
$\Det \left( H^{\bullet}_{\widetilde{\mathcal{P}}_{0}}(M, E) \right)$
and we denote this natural isomorphism by $\phi_{\lambda}$, {\it i.e.}
\begin{equation} \label{E:5.12}
\phi_{\lambda} : \Det \left( \Omega^{\bullet, \infty}_{{\widetilde {\mathcal P}}_{0}, [0, \lambda]}(M, E) \right) \rightarrow
\Det \left( H^{\bullet}_{\widetilde{\mathcal{P}}_{0}}(M, E) \right).
\end{equation}
The scalar product $\langle \hspace{0.1 cm}, \hspace{0.1 cm} \rangle_{M}$ on $\Omega^{\bullet, \infty}_{{\widetilde {\mathcal P}}_{0}, [0, \lambda]}(M, E)$
defined by $g^{E}$ and ${\frak h}^{E}$ induces a scalar product on   \newline
$\Det \left( \Omega^{\bullet, \infty}_{{\widetilde {\mathcal P}}_{0}, [0, \lambda]}(M, E) \right)$.
Let $\parallel \cdot \parallel_{\lambda}$ denote the metric on $\Det \left( H^{\bullet}_{\widetilde{\mathcal{P}}_{0}}(M, E) \right)$
such that the isomorphism (\ref{E:5.12}) is an isometry.
We define the Ray-Singer norm on $\Det \left( H^{\bullet}_{\widetilde{\mathcal{P}}_{0}}(M, E) \right)$ by
\begin{equation} \label{E:5.13}
\parallel \cdot \parallel^{RS}_{\Det ( H^{\bullet}_{\widetilde{\mathcal{P}}_{0}}(M, E) )} \hspace{0.2 cm}:= \hspace{0.2 cm}
\parallel \cdot \parallel_{\lambda} \hspace{0.1 cm} \cdot \hspace{0.2 cm} T^{RS}_{(\lambda, \infty)}(\nabla),
\end{equation}
where $T^{RS}_{(\lambda, \infty)}(\nabla)$ is defined by
\begin{equation} \label{E:5.14}
T^{RS}_{(\lambda, \infty)}(\nabla) \hspace{0.2 cm} := \hspace{0.2 cm} exp\left\{ \frac{1}{2} \sum_{q=0}^{m} (-1)^{q} q \cdot \log \Det_{- \pi}
\left( \B_{q}^{2}|_{( \Omega^{q, \infty}_{{\widetilde {\mathcal P}}_{0}, (\lambda, \infty)}(M, E) )} \right) \right\}
\hspace{0.1 cm} = \hspace{0.1 cm} \frac{1}{e^{\xi_{\lambda, \widetilde{\mathcal{P}}_{0}}(\nabla,g^M)}}.
\end{equation}
The definition (\ref{E:5.13}) does not depend on the choice of $\lambda$.
Since $\Gamma$ is a unitary self-adjoint operator with respect to the scalar product $\langle \hspace{0.1 cm}, \hspace{0.1 cm} \rangle_{M}$,
we have $\parallel \rho_{\widetilde{\mathcal{P}}_{0}, [0, \lambda]} \parallel_{\lambda} = 1$ (cf. Lemma 4.5 in \cite{BK2} and Theorem 6.2 in \cite{Ve1}) and hence
$$
\parallel  \rho_{an, \widetilde{\mathcal{P}}_{0}}(\nabla) \parallel^{RS}_{\Det ( H^{\bullet}_{\widetilde{\mathcal{P}}_{0}}(M, E) )}
\hspace{0.2 cm} = 1.
$$

\vspace{0.3 cm}


\end{document}